\documentclass[12pt]{article}

\usepackage[utf8]{inputenc}
\usepackage[T1]{fontenc}
\usepackage[english]{babel}
\usepackage{amsmath,amssymb,amsthm}
\usepackage{mathtools}
\usepackage{graphicx}
\usepackage{booktabs}
\usepackage{array,colortbl,tikz,algorithm}
\usepackage[algo2e,ruled,vlined]{algorithm2e}
\usepackage{float}
\usepackage{placeins}
\usepackage{microtype}
\usepackage[margin=1in]{geometry}
\usepackage[authoryear,round]{natbib}
\usepackage{caption}
\usepackage{xcolor}
\usepackage{parskip}
\usepackage[colorlinks=true,
            linkcolor=blue!55!black,
            citecolor=blue!55!black,
            urlcolor=blue!55!black]{hyperref}
\hypersetup{
  pdftitle={Finite-Sample Hausdorff Bounds and Hadamard Sensitivity for Regressions with MNAR Covariates},
  pdfauthor={Hugo Dunias}
}
\graphicspath{{figures/}}
\newcommand{\R}{\mathbb{R}}
\newcommand{\E}{\mathbb{E}}
\newcommand{\PP}{\mathbb{P}}
\newcommand{\Xcal}{\mathcal{X}}
\newcommand{\Mcal}{\mathcal{M}}
\newcommand{\norm}[1]{\left\|#1\right\|}
\newcommand{\abs}[1]{\left|#1\right|}
\newcommand{\dH}{d_{\mathrm{H}}}

\newcommand{\Bhat}{\widehat{\mathcal{B}}}
\newcommand{\Bstar}{\mathcal{B}^*}

\newcommand{\alttext}[1]{\par\smallskip\noindent\textit{Alt text:} #1}

\theoremstyle{plain}
\newtheorem{theorem}{Theorem}[section]
\newtheorem{lemma}[theorem]{Lemma}
\newtheorem{proposition}[theorem]{Proposition}
\newtheorem{corollary}[theorem]{Corollary}
\newtheorem{assumption}[theorem]{Assumption}
\theoremstyle{definition}

\theoremstyle{remark}
\newtheorem{remark}[theorem]{Remark}

\numberwithin{equation}{section}

\begin{document}

\title{Finite-Sample Hausdorff Bounds and
Hadamard Sensitivity for Regressions with MNAR Covariates}

\author{Hugo Dunias\\
\small CentraleSup\'elec, 3 rue Joliot-Curie, 91192 Gif-sur-Yvette Cedex, France\\
\small \href{mailto:hugo.dunias@student-cs.fr}{hugo.dunias@student-cs.fr}}
\date{September 2026}

\maketitle

\begin{abstract}
Covariates missing not at random generally prevent point
identification of regression coefficients without untestable restrictions.
This paper studies linear regression when every missing covariate is restricted
to a prespecified compact interval.  The resulting population target is the
set of best linear predictor coefficients compatible with the observed-data
law.  Under a non-atomic observed-data law and the stated regularity
conditions, we can represent this identified set as the image, under the
least-squares moment map, of the Aumann expectation of a random moment set.
For the population set and its empirical analogue, we
derive explicit bounds under bounded,
sub-exponential and polynomial envelope conditions.  The bounds show
their dependence on sample size, dimension, confidence level and tail
parameters. Under a Donsker condition and a uniform zero-set error bound, an
oracle enlargement of the set-valued Z-estimator also converges at the
$n^{-1/2}$ rate. We separately study a Hadamard-design method for exploring the
effect of admissible imputations without enumerating all vertices and build intuition to detect situations where the population target can be approximated by a zonotope.  A numerical experiment illustrates how co-missingness and the
width of the imputation intervals affect this diagnostic.
\end{abstract}


\setcounter{tocdepth}{2}
\tableofcontents
\clearpage

\section{Introduction}
\label{sec:introduction}

\citet{Rubin1976} partitions missingness mechanisms into MCAR, MAR and MNAR.
In the last case no restriction is imposed on the dependence of the
missingness indicator on the complete data. The observed-data law does not in
general identify the complete-data law
\citep{LittleRubin2019,MolenberghsEtAl2008}. Rather than specify a parametric
selection model, we restrict each missing coordinate to a prespecified
interval and ask which regression coefficients remain compatible with the
observations.
Worst-case identification with missing covariates has also been studied by
\citet{HorowitzManski2006} and, for conditional moment models, by
\citet{AucejoEtAl2017}; the latter setting differs from the unconditional
estimating equations considered here.

Random-set methods provide a natural language for set-valued observations and
moment restrictions \citep{Aumann1965,Molchanov2005}.  In
particular, \citet{BeresteanuEtAl2011} include best linear predictors with
interval-valued covariates among their examples.  Root-$n$ Hausdorff limit for sample analogues constructed from Minkowski averages was already
established by \citet{BeresteanuMolinari2008}.  Finite-dimensional deviation
estimates for sums of random sets were studied by \citet{Artstein1984}, and
\citet{NorkinWets2012} give explicit concentration inequalities under
boundedness assumptions.  Our probabilistic contribution is deliberately
narrower: it specializes the construction to an unrestricted, possibly MNAR,
covariate-completion model and records explicit dependence on
$n$, $p$, the confidence level and the envelope parameters under three tail
regimes.

A complementary problem is computational.  The empirical completion set is
generated by a box with one coordinate for each missing scalar entry, so direct
vertex enumeration grows exponentially.  \citet{BatteyCox2023} propose a
conditional sensitivity analysis based on orthogonal factorial designs.  We
adapt this idea by associating factors with missing entries or groups of
entries.  Since this type of algorithm does not control the error of its sensitivity estimation, we will try to build intuition on the conditions under which it can give reasonable estimates.

The paper makes four contributions.  First, it specializes the random-set
representation to the unrestricted, possibly MNAR, covariate-completion model
and gives a direct observational-equivalence argument for the population set
of linear projection coefficients.  Second, it gives finite-sample Hausdorff
bounds between the empirical and population sets under bounded,
sub-exponential and polynomial envelopes, including the non-convexity
remainder. Third, under a Donsker condition and a uniform zero-set error bound,
it establishes root-$n$ Hausdorff convergence for an oracle enlargement of a
set-valued Z-estimator. Fourth, it gives analytical and empirical diagnostics for
Hadamard-based exploration of admissible imputations.  The probabilistic bounds
and the Hadamard analysis address complementary questions on a more heuristic point of view.

Section~\ref{sec:model} gives the mathematical formulation.
Sections~\ref{sec:random_sets} and~\ref{sec:bounds} establish identification
and finite-sample Hausdorff bounds. Section~\ref{sec:set_valued_z} extends the
construction to estimating equations. Section~\ref{sec:hadamard} then presents
the Hadamard method, its probabilistic conditions and the adaptive algorithm.
Detailed proofs and the algebraic curvature condition are in Appendices~\ref{app:section4}--\ref{app:section3}.

\section{Mathematical formulation}
\label{sec:model}

\subsection{Observed data and admissible completions}

Let $X=(X_1,\ldots,X_p)^T\in\R^p$ be the complete covariate vector, let
$Y\in\R$ be the response and let $M\in\{0,1\}^p$ be the missingness indicator,
with $M_j=1$ when $X_j$ is missing.  We observe independent copies
$O_i=(X_i^{\mathrm{obs}},Y_i,M_i)$, $i=1,\ldots,n$.
Let $\mathcal I_{\mathrm{miss}}=\{(i,k):M_{ik}=1\}$ and
$N_{\mathrm{miss}}=|\mathcal I_{\mathrm{miss}}|$.
For each coordinate $j$, fix before the analysis a non-empty compact interval
$I_j=[\ell_j,u_j]$.  

For each observation $i$, let
$\mathbf X_i=\{z\in\R^p:z_j=X_{ij}\ \text{if }M_{ij}=0,\ z_j\in I_j\ \text{if }M_{ij}=1\}$.
Thus
$\Xcal_{\mathrm{adm}}=\{Z\in\R^{n\times p}:z_i\in\mathbf X_i,
i=1,\ldots,n\}$, where $z_i^T$ is the $i$th row.  For admissible $Z$ with
$Z^TZ$ invertible, write $\widehat\beta(Z)=(Z^TZ)^{-1}Z^TY$ and
$\Bhat=\{\widehat\beta(Z):Z\in\Xcal_{\mathrm{adm}},\ Z^TZ\text{ invertible}\}$.

\subsection{Population target and assumptions}

Let $P$ be a distribution for the
complete random vector $(X,Y,M)$. Whenever the second moments exist and the
second-moment matrix is nonsingular, the population least-squares coefficient
is $\beta(P)=\{\E_P[XX^T]\}^{-1}\E_P[XY]$.

To define observational equivalence under partial identification, we
formalize the mapping from complete to observed data. Let
$\mathcal Z=\R^p\times\R\times\{0,1\}^p$ be the complete-data space, let
$\R_\star=\R\cup\{\star\}$ and let
$\mathcal O=\R_\star^p\times\R\times\{0,1\}^p$.
Define $(m_m(x))_j=x_j$ if $m_j=0$ and $(m_m(x))_j=\star$ otherwise; hence
$\mathsf{Obs}(X,Y,M)=(m_M(X),Y,M)$.
We restrict attention to distributions satisfying the support
restriction for the missing covariates. For $o=(x^{\mathrm{obs}},y,m)$, put
$\mathbf X(o)=\{x\in\R^p:x_j=x_j^{\mathrm{obs}}\ \text{if }m_j=0,
 x_j\in I_j\ \text{if }m_j=1\}$.
Let $P_{\mathrm{obs}}$ denote the observed-data law and define
\begin{align}
    \mathcal P_{\mathrm{adm}}
    &=\left\{P:P\circ\mathsf{Obs}^{-1}=P_{\mathrm{obs}},\quad
       P\bigl(X\in\mathbf X(\mathsf{Obs}(X,Y,M))\bigr)=1\right\},
       \label{eq:Padm_formal}\\
    \mathcal P_{\mathrm{adm}}^{++}
    &=\left\{P\in\mathcal P_{\mathrm{adm}}:
      \E_P[XX^T]\in\mathbb S_{++}^p,\quad
      \E_P[\norm{X}_2^2+\norm{X}_2|Y|]<\infty\right\}.
      \label{eq:Padm_positive_definite}
\end{align}
The observationally equivalent identified set is
\begin{equation}
    \Bstar_{\mathrm{Obs}}
    =\left\{\left(\E_P[XX^T]\right)^{-1}\E_P[XY]:
      P\in\mathcal P_{\mathrm{adm}}^{++}\right\}.
    \label{eq:pop_IS_obs_equivalent}
\end{equation}
Let $P_0\in\mathcal P_{\mathrm{adm}}^{++}$ denote the true complete-data law
and write $\beta^*=\beta(P_0)$. Then $\beta^*\in\Bstar_{\mathrm{Obs}}$,
although it is not point identified from $P_{\mathrm{obs}}$.

\begin{assumption}[Sampling and completion correspondence]
\label{ass:sampling}
The observations $O_1,\ldots,O_n$ are i.i.d. with non-atomic law
$P_{\mathrm{obs}}$.  The correspondence $o\mapsto\mathbf X(o)$ has a
measurable graph and non-empty compact values.
\end{assumption}

\begin{assumption}[Integrable envelope]
\label{ass:integrable_envelope}
$\E[\sup_{z\in\mathbf X}\norm{z}_2^2+|Y|\sup_{z\in\mathbf X}\norm{z}_2]<\infty$.
\end{assumption}

\begin{assumption}[Uniform positive-definiteness]
\label{ass:ure}
There exists $\kappa>0$ such that every admissible population completion
$X^{\mathrm{obs}}$, satisfying $X_j^{\mathrm{obs}}=X_j$ if $M_j=0$ and
$X_j^{\mathrm{obs}}\in I_j$ if $M_j=1$ almost surely, obeys
$\lambda_{\min}(\E[X^{\mathrm{obs}}(X^{\mathrm{obs}})^T])\ge\kappa$.
For the empirical admissible space, assume separately that
$\inf_{Z\in\Xcal_{\mathrm{adm}}}\lambda_{\min}(Z^TZ/n)\ge\kappa$ almost surely.
\end{assumption}

\section{Estimation of the Population Identified Set: Observational
Equivalence and Random Sets}
\label{sec:estimation_beta_set}
\label{sec:random_sets}

\subsection{Aumann expectation and equivalence of identification}

Calculating $\Bstar_{\mathrm{Obs}}$ directly involves searching over a complex
space of admissible complete-data joint distributions. We instead use Random
Set Theory \citep{Molchanov2005}.

For a non-empty random closed set $\mathbf Y$, its Aumann expectation is the
set of integrals of its integrable measurable selections,
\begin{equation}
    \E_{\mathrm{Aumann}}[\mathbf Y]
    =\left\{\int V(o)\,dP_{\mathrm{obs}}(o):
      V(o)\in\mathbf Y(o)\ \text{a.s.},\quad V\in\mathcal L^1\right\}.
    \label{eq:aumann_def}
\end{equation}
If $\mathbf Y$ is measurable, non-empty compact-valued and integrably bounded,
then this expectation is non-empty and compact; it is convex when
$P_{\mathrm{obs}}$ is non-atomic \citep{Molchanov2005}. Hence adding a closure
does not change it.

Define the random moment set
$\mathbf W(o)=\{(xx^T,xY(o)):x\in\mathbf X(o)\}$ and its Aumann expectation
$\mathcal K^*_{\mathrm{adm}}
=\E_{\mathrm{Aumann}}[\mathbf W]$.
Applying $\Phi(\Sigma,\gamma)=\Sigma^{-1}\gamma$ gives
\begin{equation}
    \Bstar_{\mathrm{Aumann}}
    =\left\{\Sigma^{-1}\gamma:(\Sigma,\gamma)
      \in\mathcal K^*_{\mathrm{adm}}\right\}.
    \label{eq:pop_IS_aumann}
\end{equation}

\begin{theorem}[Equivalence of identification]
\label{thm:identification_equivalence}
Under Assumptions~\ref{ass:sampling}--\ref{ass:ure},
\begin{equation}
    \Bstar_{\mathrm{Obs}}=\Bstar_{\mathrm{Aumann}}.
    \label{eq:identification_equivalence}
\end{equation}
\end{theorem}

The complete argument is given in Appendix~\ref{app:proofplan_z_ols}. This is the
least-squares case of the sharp Aumann representation in
Section~\ref{sec:set_valued_z}. In the rest of the paper, we write
$\Bstar=\Bstar_{\mathrm{Obs}}=\Bstar_{\mathrm{Aumann}}$.

We define the empirical joint set constructed from the admissible imputations:
\begin{equation}
    \widehat{\mathcal K}_{\mathrm{adm}}
    =\left\{\left(\frac1nZ^TZ,\frac1nZ^TY\right):
      Z\in\Xcal_{\mathrm{adm}}\right\}.
    \label{eq:empirical_joint}
\end{equation}

To measure the distance between the multidimensional estimator sets, we use
the Hausdorff metric. For non-empty compact sets $A,B\subset\R^d$, let
\begin{equation}
    \dH(A,B)=\max\left\{
      \sup_{a\in A}\inf_{b\in B}\norm{a-b}_2,
      \sup_{b\in B}\inf_{a\in A}\norm{a-b}_2\right\}.
    \label{eq:hausdorff}
\end{equation}
For compact convex sets, the support function
$\psi_A(u)=\sup_{a\in A}u^Ta$ satisfies
\begin{equation}
    \dH(A,B)=\sup_{u\in\mathbb S^{d-1}}
      \abs{\psi_A(u)-\psi_B(u)}
    \label{eq:hausdorff_support}
\end{equation}
\citep{Schneider1993}.

In order to use this characterization of the Hausdorff distance we have to
establish that $\mathcal K^*_{\mathrm{adm}}$ is compact and convex.

\begin{proposition}[Properties of $\mathcal K^*_{\mathrm{adm}}$]
\label{prop:Kadm_properties}
Under Assumptions~\ref{ass:sampling} and~\ref{ass:integrable_envelope},
$\mathcal K^*_{\mathrm{adm}}$ is a compact convex subset of
$\R^{p\times p}\times\R^p$.
\end{proposition}
\noindent\textit{Proof: Appendix~\ref{app:proof_aumann_properties}.}

This characterization yields bounds for the convergence of
the empirical joint set towards the population joint set, following the strong
law for random compact sets of \citet{ArtsteinVitale1975}.

\begin{proposition}[Integrable-envelope convergence]
\label{prop:l1_K}
Under the same assumptions, for fixed $p$,
\begin{equation}
    \dH\!\left(\operatorname{conv}(\widehat{\mathcal K}_{\mathrm{adm}}),
          \mathcal K^*_{\mathrm{adm}}\right)\longrightarrow0
    \qquad\text{almost surely}.
    \label{eq:l1_K}
\end{equation}
\end{proposition}
\noindent\textit{Proof: Appendix~\ref{app:proof_l1_K}.}

Using the convexity and compactness of $\mathcal K^*_{\mathrm{adm}}$
established in Proposition~\ref{prop:Kadm_properties}, we now show the
compactness of $\Bstar$.

\begin{proposition}[Compactness of $\Bstar$]
\label{prop:Bstar}
Under the assumptions of Proposition~\ref{prop:Kadm_properties} and
Assumption~\ref{ass:ure}, $\Bstar$ is a compact subset of $\R^p$.
\end{proposition}
\noindent\textit{Proof: Appendix~\ref{app:proof_bstar_compactness}.}

While $\mathcal K^*_{\mathrm{adm}}$ is a convex set, the map
$(\Sigma,\gamma)\mapsto\Sigma^{-1}\gamma$ is a non-linear rational mapping.
Consequently, $\Bstar$ can be non-convex. The Shapley--Folkman--Starr
remainder below \citep{Starr1969} allows us to compare the empirical and
population coefficient sets directly.

\begin{proposition}[Integrable-envelope Hausdorff consistency]
\label{prop:l1_B}
Suppose Assumptions~\ref{ass:sampling}, \ref{ass:integrable_envelope}
and~\ref{ass:ure} hold. Then
\begin{equation}
    \dH\!\left(\Bhat,\Bstar\right)\longrightarrow0
    \qquad\text{almost surely}.
    \label{eq:l1_B}
\end{equation}
\end{proposition}
\noindent\textit{Proof: Appendix~\ref{app:proof_l1_B}.}

We now aim to measure the distance between the resulting empirical estimator
set $\Bhat$ and $\Bstar$.

\subsection{Finite-sample Hausdorff bounds}
\label{sec:bounds}

The following bounds correspond to bounded, sub-exponential and polynomial
assumptions on the envelope
$W=\sup_{z\in\mathbf X}\norm{z}_2^2+|Y|\sup_{z\in\mathbf X}\norm{z}_2$,
which can arise, for example, from bounded, sub-Gaussian and polynomial-tailed
covariates, respectively.

\begin{assumption}[Bounded covariates and response]
\label{ass:bounded_xy}
There exist finite constants $M_X,M_Y$ such that
$\sup_{z\in\mathbf X}\norm{z}_2\le M_X$ and $|Y|\le M_Y$ almost surely.
\end{assumption}

\begin{theorem}[Bounded-envelope finite-sample bound]
\label{thm:main_bounded}
Suppose Assumptions~\ref{ass:sampling}, \ref{ass:ure} and
\ref{ass:bounded_xy} hold. Let $L=\sqrt{M_X^4+M_X^2M_Y^2}$,
$d=p(p+1)/2+p$ and
$\overline L_\Phi^*=\sqrt{\kappa^{-2}+M_X^2M_Y^2\kappa^{-4}}$.
For every $\delta\in(0,1)$, with probability at least $1-\delta$,
\begin{equation}
  \dH\!\left(\Bhat,\Bstar\right)
  \le L\overline L_\Phi^*
  \left[
      \sqrt{\frac{2\log(1/\delta)}{n}}
      +\frac{\sqrt{2\pi d}+2}{\sqrt n}
      +\frac{\sqrt d}{n}
  \right].
  \label{eq:bounded_beta_bound}
\end{equation}
\end{theorem}
\noindent\textit{Proof: Appendix~\ref{app:proof_main_asymptotic}.}

\begin{assumption}[Sub-exponential envelope]
\label{ass:subexp_envelope}
There exist $K_W<\infty$ such that 
\begin{equation}
\norm{W}_{\psi_1}\le K_W.
\label{eq:subexp_envelope}
\end{equation}
$\norm{.}_{\psi_1}$ being the Orlicz norm.
\end{assumption}

\begin{theorem}[Sub-exponential finite-sample bound]
\label{thm:main_subexp}
Suppose Assumptions~\ref{ass:sampling}, \ref{ass:ure} and
\ref{ass:subexp_envelope} hold.  Let $d=p(p+1)/2+p$ and
$L_{\Phi,\mathrm{sub}}(n,\delta)=
\{\kappa^{-2}+K_W^2(\log2+\log(3/\delta)/n)^2\kappa^{-4}\}^{1/2}$.
For every $\delta\in(0,1)$, with probability at least $1-\delta$,
\begin{equation}
\begin{aligned}
  &\dH\!\left(\Bhat,\Bstar\right)\\
  &\quad\le K_WL_{\Phi,\mathrm{sub}}(n,\delta)
  \left[
  \begin{aligned}
      &2\sqrt{\frac{\log(3/\delta)+3\log2+d\log3}{n}}\\
      &+4\sqrt{\frac{\log(3/\delta)+5\log2+d\log12}{n}}\\
      &+\frac{3\log(3/\delta)+13\log2+d(\log3+2\log12)}{n}\\
      &+\frac{\sqrt d\log(6n/\delta)}{n}
  \end{aligned}
  \right].
\end{aligned}
\label{eq:subexp_beta_bound}
\end{equation}
\end{theorem}
\noindent\textit{Proof: Appendix~\ref{app:proof_subexp_B}.}

\begin{assumption}[Polynomial envelope]
\label{ass:polynomial_envelope}
There exists $q>2$ such that,
\begin{equation}
K_q=\{\E[W^q]\}^{1/q}<\infty,\qquad
K_2=\{\E[W^2]\}^{1/2}<\infty.
\label{eq:polynomial_constants}
\end{equation}
\end{assumption}

For $q>2$, write $\mu_q=2+\max(4/3,q/3)$.

\begin{theorem}[Polynomial-envelope finite-sample bound]
\label{thm:main_polynomial}
Suppose Assumptions~\ref{ass:sampling}, \ref{ass:ure} and
\ref{ass:polynomial_envelope} hold.  Let $d=p(p+1)/2+p$ and
$L_{\Phi,q}(n,\delta)=
\{\kappa^{-2}+K_2^2[1+(3-\delta)/(n\delta)]\kappa^{-4}\}^{1/2}$.
For every $\delta\in(0,1)$, with probability at least $1-\delta$,
\begin{equation}
\begin{aligned}
  &\dH\!\left(\Bhat,\Bstar\right)\\
  &\quad\le L_{\Phi,q}(n,\delta)
  \left[
  \begin{aligned}
      &\frac{(\sqrt{2\pi d}+2)K_2}{\sqrt n}
      +5\sqrt2K_2\sqrt{\frac{\log(3/\delta)}{n}}
      \\
      &+\{3\mu_q(K_q+K_2)+\sqrt d\,K_q\}(3/\delta)^{1/q}n^{-1+1/q}
  \end{aligned}
  \right].
\end{aligned}
\label{eq:polynomial_beta_bound}
\end{equation}
\end{theorem}
\noindent\textit{Proof: Appendix~\ref{app:proof_poly_B}.}

\FloatBarrier

\begin{table}[H]
\centering
\caption{Two-term orders of the explicit Hausdorff bounds.  Every row concerns
$\dH(\Bhat,\Bstar)$.}
\label{tab:convergence_rates}
\begin{tabular}{lll}
\toprule
Envelope & Condition & Two-term bound \\
\midrule
Bounded
& Assumption~\ref{ass:bounded_xy}
& $\mathcal O_P\!\left(pn^{-1/2}+pn^{-1}\right)$ \\
Sub-exponential
& Assumption~\ref{ass:subexp_envelope}
& $\mathcal O_P\!\left(pn^{-1/2}+(p^2+p\log n)n^{-1}\right)$ \\
Polynomial
& Assumption~\ref{ass:polynomial_envelope}, $q>2$
& $\mathcal O_P\!\left(pn^{-1/2}+p n^{-1+1/q}\right)$ \\
Integrable
& Assumption~\ref{ass:integrable_envelope}
& almost-sure consistency only \\
\bottomrule
\end{tabular}
\end{table}
\FloatBarrier

The displayed orders treat $q$ as fixed and $\kappa^{-1}$ and the envelope
constants as uniformly bounded in $p$.

Root-$n$ Hausdorff limit for sample analogues constructed from
Minkowski averages was already established by
\citet{BeresteanuMolinari2008} under second-moment and regularity conditions.
The present results provide explicit finite-sample upper bounds and display
the dependence on $p$, $\delta$ and the envelope constants. The next section
extends this random-set representation to a broader class of estimating
equations and gives an oracle root-$n$ result under Donsker and error-bound
conditions. These bounds concern the exact empirical completion set.
\section{Set-Valued Z-Estimation under MNAR Covariates}
\label{sec:set_valued_z}

This section extends the previous least-squares construction to general
estimating equations.

The main difference with the previous framework is
that the population target is no longer a fixed transformation of one Aumann
expectation. Here we generalize it to the zero set of a set-valued map
$\mathcal G$ defined below: $\Bstar=\{\beta\in\Theta:0\in\mathcal G(\beta)\}$.

\subsection{Aumann representation and empirical zero sets}

Let $\Theta\subset\R^p$ be non-empty and compact and let
$\varphi:\mathcal Z\times\Theta\to\R^r$ be an estimating function. For an
observation $o=(x^{\mathrm{obs}},y,m)$, set
$\mathbf Z(o)=\mathbf X(o)\times\{y\}\times\{m\}$, so that every
$z\in\mathbf Z(o)$ is compatible with $o$.

This framework includes, for example, least squares,
$\varphi((x,y,m),\beta)=x(x^T\beta-y)$; GLMs,
$\varphi((x,y,m),\beta)=x\{\mu(x^T\beta)-y\}$;
and M-estimators $\varphi((x,y,m),\beta)=-x\psi(y-x^T\beta)$ when
$\psi$ is continuous. The usual quantile-regression score is discontinuous
and is therefore not covered by this framework.

\begin{assumption}[Completion and moment correspondence]
\label{ass:z_completion}
The correspondence $o\mapsto\mathbf X(o)$ is measurable and has non-empty
compact values. The map $(z,\beta)\mapsto\varphi(z,\beta)$ is jointly
measurable and, almost surely, continuous on $\mathbf Z(O)\times\Theta$.
Moreover, $\E[\sup_{\beta\in\Theta}\sup_{z\in\mathbf Z(O)}
\norm{\varphi(z,\beta)}_2]<\infty$.
\end{assumption}

For each $\beta\in\Theta$, define
$\Phi_\beta(o)=\{\varphi(z,\beta):z\in\mathbf Z(o)\}$ and
$\mathcal G(\beta)=\E_{\mathrm{Aumann}}[\Phi_\beta(O)]$.
The observationally equivalent identified set is
\begin{equation}
    \Bstar_{\mathrm{Obs}}
    =\left\{\beta\in\Theta:\exists P\in\mathcal P_{\mathrm{adm}},\
      \E_P[\varphi((X,Y,M),\beta)]=0\right\}.
    \label{eq:z_obs_identified_set}
\end{equation}

We now characterize the identified set using the Aumann expectation. This
characterization is crucial for the concentration results and was used for
least squares in Theorem~\ref{thm:identification_equivalence}.

\begin{theorem}[Sharp Aumann representation]
\label{thm:z_sharp_equivalence}
Under Assumptions~\ref{ass:sampling} and~\ref{ass:z_completion}, for every
$\beta\in\Theta$,
\begin{equation}
    \mathcal G(\beta)
    =\{\E_P[\varphi((X,Y,M),\beta)]:P\in\mathcal P_{\mathrm{adm}}\}.
    \label{eq:z_attainable_moments_equal_aumann}
\end{equation}
Consequently,
\begin{equation}
    \Bstar_{\mathrm{Obs}}
    =\Bstar_{\mathrm{Aumann}}
    =\{\beta\in\Theta:0\in\mathcal G(\beta)\},
    \label{eq:z_identification_equivalence}
\end{equation}
and $\mathcal G(\beta)$ is non-empty, compact and convex.
\end{theorem}

A detailed proof roadmap is given in Appendix~\ref{app:proofplan_z_sharp_equivalence}. The proof
uses three standard ingredients: disintegration of admissible laws, Aumann's
convexity theorem for non-atomic probability spaces, and measurable lifting
from a selection of $\Phi_\beta$ to a compatible selection of $\mathbf X$.

For $u\in\mathbb S^{r-1}$, let
$\psi_{\beta,u}(o)=\sup_{z\in\mathbf Z(o)}u^T\varphi(z,\beta)$.
Measurable maximum arguments give
$\psi_{\mathcal G(\beta)}(u)=\E[\psi_{\beta,u}(O)]$, which links
set-valued identification to scalar empirical-process theory.

Given $O_1,\ldots,O_n$, the deterministic Minkowski average is
$\widehat{\mathcal G}^{\mathrm d}_n(\beta)=n^{-1}\sum_i\Phi_\beta(O_i)$;
its convexified version is
$\widehat{\mathcal G}_n(\beta)=n^{-1}\sum_i\operatorname{conv}(\Phi_\beta(O_i))
=\operatorname{conv}(\widehat{\mathcal G}^{\mathrm d}_n(\beta))$.
Thus convexification preserves support-function bounds and directional
extrema.

We now study the asymptotic properties of the empirical moment process and the
convergence of its zero set. We begin with a pointwise central limit theorem
for the random moment sets, which adapts the methods of
\citet{BeresteanuMolinari2008} to our setting.

\begin{corollary}[Pointwise random-set central limit theorem]
\label{cor:z_pointwise_clt}
Under Assumption~\ref{ass:sampling}, fix $\beta\in\Theta$. Suppose that
$\Phi_\beta$ is measurable,
non-empty and compact-valued, and\linebreak[2]
\hspace{-0.4pt}$\E[(\sup_{v\in\Phi_\beta(O)}\norm{v}_2)^2]<\infty$. Then
\begin{equation}
    \sqrt n\,\dH\bigl(\widehat{\mathcal G}_n(\beta),\mathcal G(\beta)\bigr)
    \rightsquigarrow
    \sup_{u\in\mathbb S^{r-1}}|\mathbb Z_\beta(u)|,
    \label{eq:z_pointwise_hausdorff_limit}
\end{equation}
where $\mathbb Z_\beta$ is centered Gaussian with covariance
$\operatorname{Cov}(\psi_{\beta,u}(O),\psi_{\beta,v}(O))$.
\end{corollary}
\noindent\textit{Proof: Appendix~\ref{app:proof_z_pointwise_clt}.}

To establish the convergence of the zero sets, we require a functional central
limit theorem ensuring uniform convergence of the empirical moment process over the parameter space $\Theta$.

\begin{theorem}[Functional random-set central limit theorem]
\label{thm:z_functional_clt}
Under Assumptions~\ref{ass:sampling} and~\ref{ass:z_completion}, if the class
$\{\psi_{\beta,u}:\beta\in\Theta,u\in\mathbb S^{r-1}\}$ is
$P_{\mathrm{obs}}$-Donsker and is either pointwise measurable or admits a
separable version, then the empirical support process converges
weakly in $\ell^\infty(\Theta\times\mathbb S^{r-1})$ to a centered Gaussian
process. Consequently,
\begin{equation}
    \Delta_n
    :=\sup_{\beta\in\Theta}
      \dH(\widehat{\mathcal G}_n(\beta),\mathcal G(\beta))
    =\mathcal O_P(n^{-1/2}).
    \label{eq:z_uniform_moment_rate}
\end{equation}
\end{theorem}
\noindent\textit{Proof: Appendix~\ref{app:proof_z_functional_clt}.}

Uniform convergence of $\widehat{\mathcal G}_n(\beta)$ is not, by itself,
enough to guarantee convergence of its zero set. We therefore add an error bound for the estimating equation,
as specified in the following assumption. The completion and moment
assumptions make the zero sets closed; since $\Theta$ is compact, they are
compact whenever non-empty.

\begin{assumption}[Uniform zero-set error bound]
\label{ass:z_error_bound}
The identified set $\Bstar$ is non-empty and there exists $\kappa_0>0$ such
that, for every $\beta\in\Theta$,
\begin{equation}
    d(0,\mathcal G(\beta))
    \ge\kappa_0d(\beta,\Bstar).
    \label{eq:z_metric_subregularity}
\end{equation}
\end{assumption}

We therefore enlarge the
empirical zero set by the error $\Delta_n$ defined above.
Using this error bound, the following enlargement attains the root-$n$
rate of the uniform moment process.

\begin{corollary}[Root-$n$ convergence of the enlarged zero set]
\label{cor:z_oracle_zero_root_n}
Under Assumption~\ref{ass:z_error_bound}, define
\begin{equation}
    \widehat{\mathcal B}_n^+
    =\{\beta\in\Theta:
       d(0,\widehat{\mathcal G}_n(\beta))\le\Delta_n\}.
    \label{eq:z_enlarged_zero_set}
\end{equation}
Then $\Bstar\subseteq\widehat{\mathcal B}_n^+$ and
\begin{equation}
    \dH(\widehat{\mathcal B}_n^+,\Bstar)
    \le\frac{2\Delta_n}{\kappa_0}
    =\mathcal O_P(n^{-1/2})
    \label{eq:z_zero_set_rate}
\end{equation}
under Theorem~\ref{thm:z_functional_clt}.
\end{corollary}
\noindent\textit{Proof: Appendix~\ref{app:proofplan_z_oracle_zero_root_n}.}

Because $\Delta_n$ depends on the population moment correspondence, this is an
oracle outer approximation rather than a directly implementable confidence
region.

The identification and moment-process results above allow $r\ne p$, whereas
the implicit-function sensitivity analysis below requires $r=p$ or a separate
square estimating map.

The final part of the paper concerns computation. Direct enumeration of the
completion box is exponential in the number of missing entries so we will try to identify cases where the population set is zonotopal and can be estimated by a Hadamard-inspired completion algorithm.

\section{Hadamard approximation of imputation sensitivity}
\label{sec:hadamard}

We adapt the approach of \citet{BatteyCox2023} by treating each missing entry
as a factor in a fractional factorial design. The main-effect contrasts define
a candidate zonotopal approximation, while diagonal derivatives and
two-factor contrasts diagnose departures from this approximation.

Let $X^{\mathrm{obs}}$ be the observed data matrix and
$M\in\{0,1\}^{n\times p}$ the missingness indicator matrix, where
$M_{ik}=1$ implies the covariate $k$ for observation $i$ is missing. Recall
that $\mathcal I_{\mathrm{miss}}$ contains the $N_{\mathrm{miss}}$ missing
scalar entries. In our design, each of them is treated as an independent
factor.

For each missing entry $(i,k)\in\mathcal I_{\mathrm{miss}}$, we define two
extreme plausible imputation values based on the prespecified interval $I_k$.
Define its midpoint and half-width by
$z_{ik}^*=(\ell_k+u_k)/2$ and
$c_k=(u_k-\ell_k)/2$, and set
$Z_{ik}(x)=z_{ik}^*+c_kx_{ik}$ for $x_{ik}\in[-1,1]$, with observed entries
fixed. This parameterizes the admissible box and induces
$\widehat\beta(x)=\{Z(x)^TZ(x)\}^{-1}Z(x)^TY$.

To avoid an exhaustive evaluation of all $2^{N_{\mathrm{miss}}}$ possible
combinations, we can reduce the number by using the method presented in
\citet{BatteyCox2023}. We select a Hadamard matrix $H\in\{-1,+1\}^{N\times N}$,
where $N=2^q$ is chosen so that $N\ge N_{\mathrm{miss}}+1$. We exclude the
all-ones column and extract $N_{\mathrm{miss}}$ balanced, mutually orthogonal
columns of $H$.
Let $h^{(r)}\in\{-1,+1\}^{N_{\mathrm{miss}}}$ denote the $r$-th row of this
truncated matrix. Each row $h^{(r)}$ specifies an imputation scenario for the
entire dataset: if $h_{r,(i,k)}=+1$, the
missing entry for observation $i$ on variable $k$ is replaced by the upper
endpoint, and if $h_{r,(i,k)}=-1$, it is replaced by the lower endpoint.
This yields $N$ fully imputed, dense design matrices. For each of the $N$
imputed matrices, we estimate the OLS regression coefficients. For a target coefficient
$\widehat\beta_j$, the Hadamard main-effect contrast of entry $(i,k)$ is
\begin{equation}
    \tau_{j,(i,k)}=\frac2N\sum_{r=1}^Nh_{r,(i,k)}
        \widehat\beta_j(h^{(r)}).
    \label{eq:microscopic_contrast}
\end{equation}

When two-factor interactions are extracted, a Resolution~V design is used on
the relevant block. For a block of $d_g$ factors, this requires at least
$1+d_g+\binom{d_g}{2}$ runs, so Resolution~V is not imposed on the full
microscopic screening design.

For a compact set $A\subset\R^p$, write
$\operatorname{wid}_j(A)=\sup_{\beta\in A}\beta_j-
\inf_{\beta\in A}\beta_j$.

\begin{theorem}[Deterministic Hadamard width bound]
\label{thm:hadamard_width_bound}
Under the empirical part of Assumption~\ref{ass:ure}, write
$r=(i,k)$ for an element of $\mathcal I_{\mathrm{miss}}$ and
$\partial_r=\partial/\partial x_{ik}$. Define $\widehat W_{\mathrm H,j}
   :=\sum_{r\in\mathcal I_{\mathrm{miss}}}|\tau_{j,r}| $, $\varepsilon_{\mathrm H,j} :=\sup_{x\in[-1,1]^{N_{\mathrm{miss}}}}
      \sum_{r,s\in\mathcal I_{\mathrm{miss}}}
      |\partial_{rs}^2\widehat\beta_j(x)|
      +\sum_{r\in\mathcal I_{\mathrm{miss}}}
      |\tau_{j,r}-2\partial_r\widehat\beta_j(0)|
$
Then
\begin{equation}
 \left|\widehat W_{\mathrm H,j}
       -\operatorname{wid}_j(\Bhat)\right|
 \le \varepsilon_{\mathrm H,j}.
 \label{eq:hadamard_width_bound}
\end{equation}
For each $r\in\mathcal I_{\mathrm{miss}}$, define the vector contrast
$\tau_r=2N^{-1}\sum_{b=1}^Nh_{b,r}\widehat\beta(h^{(b)})$ and the zonotope
\[
 \widehat{\mathcal Z}_{\mathrm H}
 =\widehat\beta(0)+
 \left\{\frac12\sum_{r\in\mathcal I_{\mathrm{miss}}}t_r\tau_r:
 t_r\in[-1,1]\right\}.
\]
Then
\begin{equation}
\begin{split}
 \dH(\Bhat,\widehat{\mathcal Z}_{\mathrm H})
 \le{}&\frac12\sup_{x\in[-1,1]^{N_{\mathrm{miss}}}}
 \sum_{r,s\in\mathcal I_{\mathrm{miss}}}
 \norm{\partial_{rs}^2\widehat\beta(x)}_2\\
 &+\frac12\sum_{r\in\mathcal I_{\mathrm{miss}}}
 \norm{\tau_r-2\partial_r\widehat\beta(0)}_2.
\end{split}
\label{eq:hadamard_zonotope_bound}
\end{equation}
The same statement applies to any partition of missing entries parametrization after
replacing the factor index set.
\end{theorem}
\noindent\textit{Proof: Appendix~\ref{app:proof_hadamard_width_bound}.}

In general this error term is not well controlled, we will build in the rest heuristic conditions indicating if this error can be reasonably small.

\begin{corollary}[Statistical and computational error]
\label{cor:total_width_error}
Let $r_n(\delta)$ denote the right-hand side of the applicable finite-sample
bound in Theorems~\ref{thm:main_bounded}--\ref{thm:main_polynomial}. Under its
assumptions, with probability at least $1-\delta$,
\begin{equation}
 \left|\widehat W_{\mathrm H,j}
       -\operatorname{wid}_j(\Bstar)\right|
 \le \varepsilon_{\mathrm H,j}+2r_n(\delta).
 \label{eq:total_width_error}
\end{equation}
\end{corollary}
\noindent\textit{Proof: Appendix~\ref{app:proof_hadamard_width_bound}.}

This is direct and works for any other estimation algorithm $\widehat W$ and if its associated error bound is well - controlled it can be a great application of the finite sample error bounds of Theorems~\ref{thm:main_bounded}--\ref{thm:main_polynomial}.
Thus the two terms respectively measure the algorithmic approximation error and
the statistical error. 

\subsection{Coordinate-level probabilistic condition}

The exact bound depends on the Hessian over the full cube. For interpretation,
the following comparisons require some simplifications but give some more interpretable conditions. Their derivation is
given in Appendix~\ref{app:section3} and they are only heuristics.

\begin{remark}[Simplified probabilistic condition for GLMs]
\label{rem:z_probabilistic_domination_glm}
For the canonical GLM score
$\varphi((z,y,m),\beta)=z\{\mu(z^T\beta)-y\}$, set
$\Omega_\mu=\{\E_P[\mu'(z^T\beta^*)zz^T]\}^{-1}$, let
$\omega_{\mu,j}^T$ be its $j$-th row, and let
$\varepsilon^*(z)=Y-\mu(z^T\beta^*)$. Suppose that $\mu$ is twice continuously differentiable and that the aggregated global coupling terms are negligible relative to the first-order sensitivity. The implicit function theorem then makes
$x\mapsto\widehat\beta(x)$ twice continuously differentiable, so by implicit
differentiation we get that the linear approximation can be reliable if:
\begin{equation}
    \sum_l\PP(M_l=1\mid M_k=1)c_l\mathcal C_{j,kl}^*
    \ll
    \E_P\!\left[
      |\Omega_{\mu,jk}\varepsilon^*(z)
       -\mu'(z^T\beta^*)\beta_k^*(\omega_{\mu,j}^Tz)|
      \mid M_k=1\right],
    \label{cond:z_probabilistic_glm}
\end{equation}
where $\mathcal C_{j,kl}^*=\E_P[\abs{\mu'(z^T\beta^*)
(\beta_l^*\Omega_{\mu,jk}+\beta_k^*\Omega_{\mu,jl})
+\mu''(z^T\beta^*)\beta_k^*\beta_l^*(\omega_{\mu,j}^Tz)}
\mid M_k=1,M_l=1]$.
For $l=k$, the repeated conditioning event is read as $M_k=1$.
\end{remark}

The bounded curvature of the logistic mean function makes this condition
potentially more favorable than for Poisson regression, whose curvature may
grow rapidly.

In the remainder of this section, we treat only linear regression, for which
the corresponding conditions are more directly interpretable. Write
$\Sigma=\E_P[XX^T]$, $\Omega=\Sigma^{-1}$,
$\omega_j^T=e_j^T\Omega$, $\beta^*=\Sigma^{-1}\E_P[XY]$, and
$\varepsilon=Y-X^T\beta^*$.

\begin{remark}[Simplified probabilistic condition]
\label{rem:orthogonal_simplification}
Assuming orthogonal covariates ($\Omega=I$), $c_l\asymp c$ as for standardized
covariates, and $|\beta_l^*|\asymp|\beta_k^*|$,
and in the high signal-to-noise regime, the heuristic rule is
\begin{equation}
\begin{cases}
 c\left\{1+\displaystyle\sum_{l\ne k}\PP(M_l=1\mid M_k=1)\right\}
 \ll\E_P[|X_k|\mid M_k=1],&j=k,\\[2mm]
 c\PP(M_j=1\mid M_k=1)
 \ll\E_P[|X_j|\mid M_k=1],&j\ne k.
\end{cases}
\label{eq:final_heuristic_snr}
\end{equation}

In the remainder of the paper, we will consider coarser imputation schemes to reduce computational complexity. However, as shown in the proposition, their accuracy cannot exceed that of microscopic imputation. Therefore, although this condition is heuristic, it is irreducible: if it is not satisfied, none of the imputation schemes considered here can achieve a small error.
\end{remark}

\subsection{Covariate-level imputation and probabilistic condition}

Evaluating sensitivity over a space of dimension $N_{\mathrm{miss}}$ can be
computationally long. We try to reduce this complexity by projecting the
microscopic imputation space onto a lower-dimensional covariate-level
subspace. Let $\Mcal_k=\{i:M_{ik}=1\}$,
$\mathcal J=\{k:\Mcal_k\ne\varnothing\}$ and $m=|\mathcal J|$. Assuming a
uniform missingness policy $x\in[-1,1]^m$ across observations, set
$Z_{ik}(x)=z_{ik}^*+c_kx_k$ for $i\in\Mcal_k$.
This searches a much lower-dimensional subset of the microscopic completion box.

\begin{remark}[Simplified probabilistic condition]
\label{rem:uncorrelated_rules}
For orthogonal covariates, $c_l\asymp c$,
$|\beta_l^*|\asymp|\beta_k^*|$ and
$\E_P[\varepsilon\mid M_k=1]=0$, we get
\begin{equation}
\begin{cases}
 c\left\{1+\displaystyle\sum_{l\ne k}\PP(M_l=1\mid M_k=1)\right\}
 \ll|\E_P[X_k\mid M_k=1]|,&j=k,\\[2mm]
 c\PP(M_j=1\mid M_k=1)
 \ll|\E_P[X_j\mid M_k=1]|,&j\ne k.
\end{cases}
\label{cond:uncorrelated_rule}
\end{equation}
The main difference from the microscopic condition is that we obtain
$|\E_P[X_j\mid M_k=1]|$ instead of
$\E_P[|X_j|\mid M_k=1]$. With some missingness distributions this can change
substantially because opposite perturbations cancel under covariate-level
imputation.
\end{remark}

\subsection{Practical interaction diagnostic}

For a covariate-level run $h^{(r)}\in\{-1,1\}^{|\mathcal J|}$, define the main
effect
\begin{equation}
    \tau_{j,k}=\frac2N\sum_{r=1}^N
      h_{r,k}\widehat\beta_j(h^{(r)}).
    \label{eq:macroscopic_contrast}
\end{equation}
For factors $k,l\in\mathcal J$, define
\begin{equation}
    \tau_{j,kl}=\frac2N\sum_{r=1}^N
      h_{r,k}h_{r,l}\widehat\beta_j(h^{(r)}).
    \label{eq:interaction_contrast}
\end{equation}
If the chosen fractional factorial design has Resolution~V, substitution of
the second-order expansion gives
$\tau_{j,kl}\approx2\,\partial^2\widehat\beta_j(0)/
(\partial x_k\partial x_l)$. A direct numerical diagnostic of linear approximability is:
\begin{equation}
    2\left|\frac{\partial^2\widehat\beta_j}{\partial x_k^2}(0)\right|
      +\sum_{l\in\mathcal J,\,l\ne k}|\tau_{j,kl}|
      \ll |\tau_{j,k}|.
    \label{eq:interaction_diagnostic}
\end{equation}
Without the first term, the same comparison is a cross-interaction
diagnostic.

\subsubsection{Numerical illustration}
\label{sec:numerical}

In the numerical illustrations, the endpoints may be calibrated as empirical
quantiles. This is a sensitivity choice rather than part of the identification
assumptions.

To assess Remark~\ref{rem:uncorrelated_rules} and
condition~\eqref{eq:interaction_diagnostic}, we compare
$\sum_k|\tau_{j,k}|$ with the numerically optimized continuous width and
report the absolute continuous-width error and the second-order-to-main-effect
ratio
\begin{equation}
    \left|
      \max_{x\in[-1,1]^{|\mathcal J|}}\widehat\beta_j(x)
      -\min_{x\in[-1,1]^{|\mathcal J|}}\widehat\beta_j(x)
      -\sum_k|\tau_{j,k}|
    \right|,
    \qquad
    \frac{2\sum_k\left|\frac{\partial^2\widehat\beta_j}{\partial x_k^2}(0)\right|
      +2\sum_{k<l}|\tau_{j,kl}|}{\sum_k|\tau_{j,k}|}.
    \label{eq:numerical_metrics}
\end{equation}

\begin{figure}[H]
\centering
\includegraphics[width=0.82\textwidth]{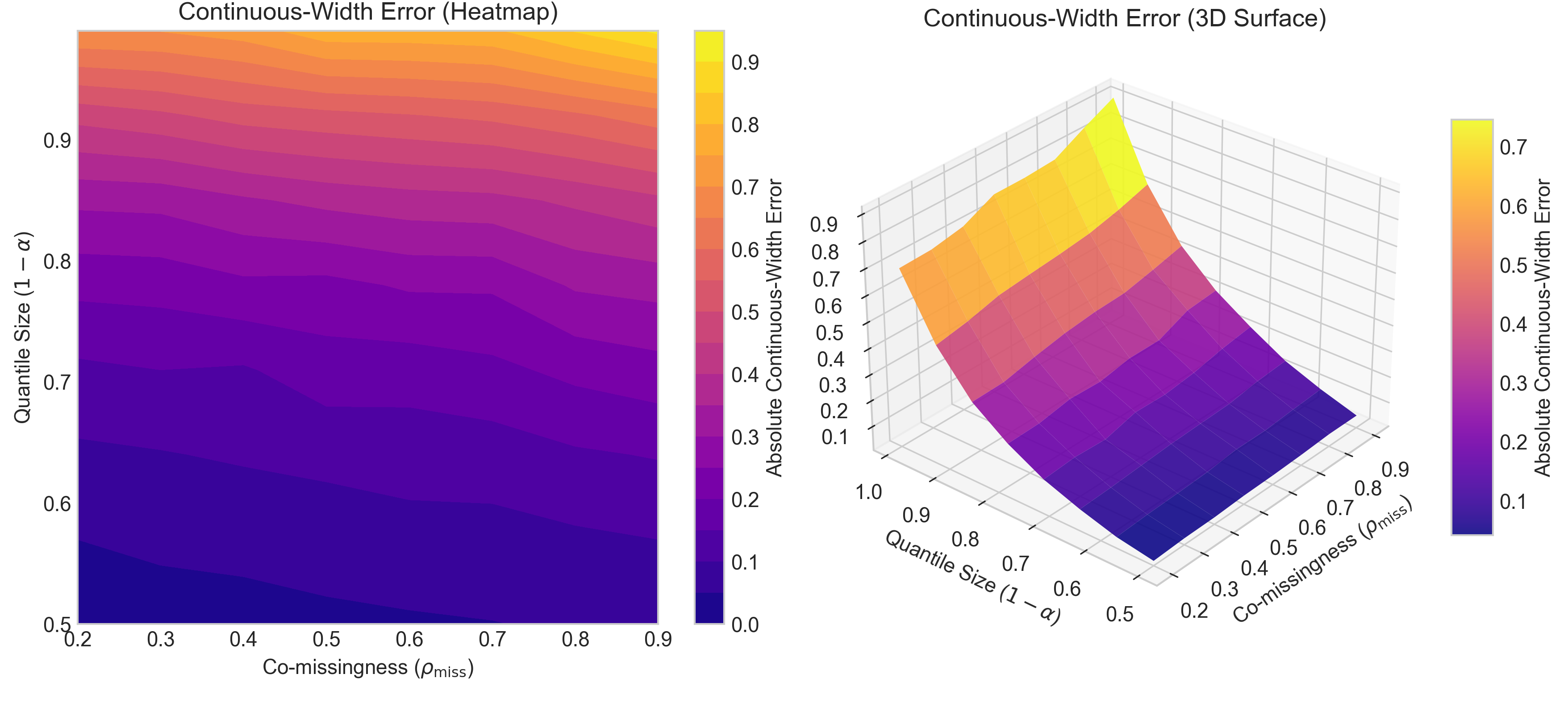}
\caption{Continuous-width Hadamard error and second-order-to-main-effect ratio across
co-missingness $(\rho_{\mathrm{miss}})$ and quantile size $(1-\alpha)$.}
\label{fig:hadamard_grid1}
\alttext{Heat map and surface of continuous-width error.}
\end{figure}

\begin{figure}[H]
\ContinuedFloat
\centering
\includegraphics[width=0.95\textwidth]{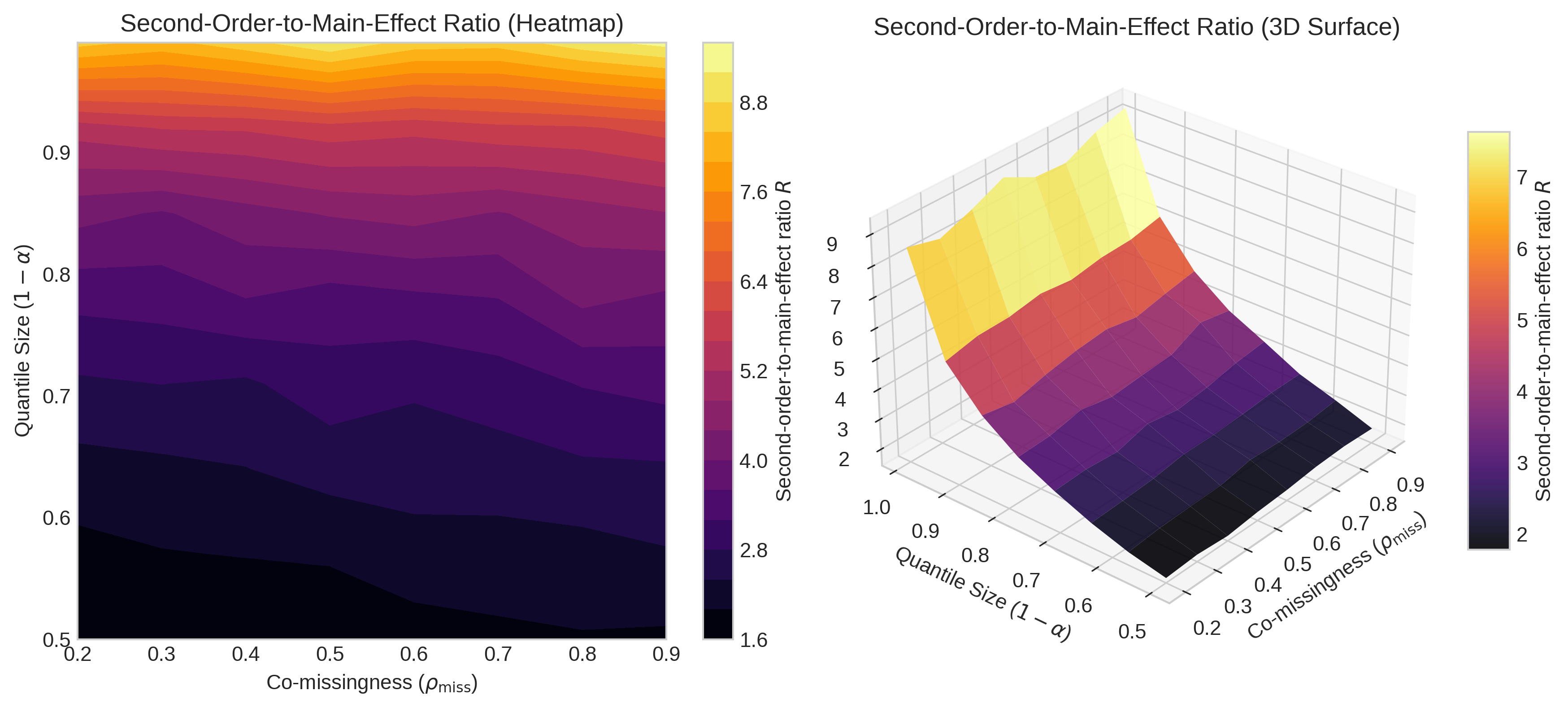}
\caption[]{Continuous-width Hadamard error and second-order-to-main-effect ratio
(continued).}
\alttext{Heat map and surface of the second-order-to-main-effect ratio.}
\end{figure}

\FloatBarrier

Both quantities tend to increase with co-missingness
$(\rho_{\mathrm{miss}})$ and quantile size $(1-\alpha)$. This is consistent
with the probabilistic condition in Remark~\ref{rem:uncorrelated_rules} and
the practical condition~\eqref{eq:interaction_diagnostic}. Reproducibility
details are in Appendix~\ref{app:reproducibility}.

\subsection{Groupwise and adaptive Hadamard analysis}
\label{sec:groupwise_adaptive_hadamard}

The microscopic Hadamard procedure treats every missing scalar entry as an
independent factor, whereas the macroscopic procedure treats each covariate
with missing values as a single factor. These two approaches correspond to two
extreme parametrizations of the admissible imputation space. The microscopic
method is flexible but computationally expensive; the macroscopic method is
computationally efficient but may hide heterogeneity among missing entries
belonging to the same covariate.

This subsection introduces an intermediate family of groupwise procedures and
an adaptive macro/micro refinement rule. The objective is to keep the
computational advantages of the macroscopic method whenever it is reliable,
while switching to a more local parametrization when the macroscopic
approximation is not justified.

Let $\mathcal C=\{C_1,\ldots,C_d\}$ be a partition of
$\mathcal I_{\mathrm{miss}}$. Each cell $C_a$ is assigned $u_a\in[-1,1]$;
for $(i,k)\in C_a$ set $Z_{ik}^{\mathcal C}(u)=z_{ik}^*+c_ku_a$ and
$\Xcal_{\mathcal C}=\{Z^{\mathcal C}(u):u\in[-1,1]^d\}$,
with observed entries fixed. The corresponding restricted coefficient set is
$\Bhat_{\mathcal C}=\{\widehat\beta(Z):Z\in\Xcal_{\mathcal C}\}$. 
Note that the microscopic and macroscopic methods are special cases of this construction, the microscopic method being the discrete partition composed of all the singletons of
the missing data and the macroscopic method being the case with a trivial partition with one element containing all of the missing
data.

\begin{proposition}[Monotonicity of mixed refinement]
\label{prop:mixed_refinement_monotonicity}
If $\mathcal C'$ refines $\mathcal C$, then
$\Xcal_{\mathcal C}\subseteq\Xcal_{\mathcal C'}$ and
$\Bhat_{\mathcal C}\subseteq\Bhat_{\mathcal C'}$. In particular, for every
coefficient $j$,
\begin{equation}
    \sup_{\beta\in\Bhat_{\mathcal C}}\beta_j
    -\inf_{\beta\in\Bhat_{\mathcal C}}\beta_j
    \le
    \sup_{\beta\in\Bhat_{\mathcal C'}}\beta_j
    -\inf_{\beta\in\Bhat_{\mathcal C'}}\beta_j.
    \label{eq:refinement_width}
\end{equation}
\end{proposition}
\noindent\textit{Proof: Appendix~\ref{app:proof_refinement}.}

This confirms the tradeoff between complexity and exactness and reinforces
the motivation for grouping the missing data carefully.

For a block $B_g$ of cells, let
$H_g\in\{-1,1\}^{N_g\times|B_g|}$ be a Resolution~V design. All factors outside
$B_g$ are held at zero. Define
\begin{align}
    \tau_{j,C}^{(g)}
      &=\frac2{N_g}\sum_{r=1}^{N_g}h_{r,C}\widehat\beta_j^{(r)},
      \label{eq:groupwise_main_def}\\
    \tau_{j,C,C'}^{(g)}
      &=\frac2{N_g}\sum_{r=1}^{N_g}
        h_{r,C}h_{r,C'}\widehat\beta_j^{(r)}.
      \label{eq:groupwise_interaction_def}
\end{align}

First, we group observations whose missingness patterns appear related. Set
\begin{equation}
 s_{ii'}=
 \begin{cases}
 \displaystyle\frac{\sum_kM_{ik}M_{i'k}}
 {\sum_k\mathbf1\{M_{ik}+M_{i'k}\ge1\}},
 &\displaystyle\sum_k\mathbf1\{M_{ik}+M_{i'k}\ge1\}>0,\\[3mm]
 0,&\text{otherwise}.
 \end{cases}
 \label{eq:observation_similarity}
\end{equation}
An observation graph is constructed by connecting patient $i$ and $i'$ whenever
$s_{ii'}\ge\lambda_{\mathrm{obs}}$. The connected components of this graph
define the observation groups $G_1,\ldots,G_H$. For each group $G_h$ and
covariate $k$, the imputation cell is defined as
$C_{h,k}=\{(i,k):i\in G_h,M_{ik}=1\}$. Within a group, the symmetric cell
similarity is
\begin{equation}
 a_{(h,k),(h,l)}=\frac12\left\{
 \frac{\sum_{i\in G_h}M_{ik}M_{il}}{\sum_{i\in G_h}M_{ik}}
 +\frac{\sum_{i\in G_h}M_{ik}M_{il}}{\sum_{i\in G_h}M_{il}}
 \right\}.
 \label{eq:cell_similarity}
\end{equation}
We connect cells whenever $a_{(h,k),(h,l)}\ge\lambda_{\mathrm{cell}}$. The
connected components of this cell graph define the cell blocks
$\mathfrak B=\{B_1,\ldots,B_G\}$.

For a tolerance $\eta_{\mathrm{lin}}\in(0,1)$, cell $C\in B_g$ is kept at
the cell level when
\begin{equation}
    2\left|\frac{\partial^2\widehat\beta_j}{\partial u_C^2}(0)\right|
    +\sum_{C'\in B_g,\,C'\ne C}|\tau_{j,C,C'}^{(g)}|
    \le\eta_{\mathrm{lin}}|\tau_{j,C}^{(g)}|;
    \label{eq:cell_linearity_diagnostic}
\end{equation}
otherwise it is refined to its microscopic entries. 

\begin{algorithm2e}[H]
\caption{Adaptive Groupwise Macro--Micro Hadamard Sensitivity Analysis}
\label{alg:adaptive_groupwise_macro_micro}
\footnotesize
\SetKwInOut{Input}{Input}
\SetKwInOut{Output}{Output}
\Input{Target coefficient $j$; thresholds
$\lambda_{\mathrm{obs}},\lambda_{\mathrm{cell}},\eta_{\mathrm{lin}}$;
observed data $(X^{\mathrm{obs}},Y,M)$.}
\Output{Observation groups, cell blocks, and blockwise width approximation
$\widehat W_j$.}

Restrict the observation graph to rows $i$ with $\sum_kM_{ik}>0$; connect
$i,i'$ when $s_{ii'}\ge\lambda_{\mathrm{obs}}$, and take its connected
components $G_1,\ldots,G_H$\;

For every $G_h$ and incomplete covariate $k$, form the non-empty cells
$C_{h,k}$; connect two cells when
$a_{(h,k),(h,l)}\ge\lambda_{\mathrm{cell}}$, and take the resulting blocks
$B_1,\ldots,B_G$\;

\For{each cell block $B_g$}{
 Run a local Resolution~V or full factorial design with all other blocks fixed
 at zero, and compute $\tau_{j,C}^{(g)}$ and
 $\tau_{j,C,C'}^{(g)}$, together with the central diagonal derivatives\;

 \For{each $C\in B_g$}{
  \eIf{$2|\partial^2\widehat\beta_j(0)/\partial u_C^2|
  +\sum_{C'\in B_g,\,C'\ne C}|\tau_{j,C,C'}^{(g)}|
  \le\eta_{\mathrm{lin}}|\tau_{j,C}^{(g)}|$}{
   Set $\widehat W_{j,C}^{(g)}=|\tau_{j,C}^{(g)}|$\;
  }{
   Refine $C$ into microscopic entries, compute their local effects, and set
   $\widehat W_{j,C}^{(g)}=\sum_{(i,k)\in C}|\tau_{j,(i,k)}^{(g)}|$\;
  }
 }
}
Return
$\widehat W_j=\sum_{g=1}^G\sum_{C\in B_g}\widehat W_{j,C}^{(g)}$\;
\end{algorithm2e}

Under the same validity assumptions for this method, for OLS, exact local derivatives can be cheaper. The gain here is relative to
vertex enumeration and for estimators treated as black boxes.

\section{Discussion}
\label{sec:discussion}

Prespecified support intervals turn unrestricted covariate missingness,
including MNAR mechanisms, into a partial-identification problem rather than a
point-estimation problem. For least squares, the Aumann formulation separates
the random geometry of admissible moments from the nonlinear coefficient map
and leads to explicit finite-sample Hausdorff bounds for the exact, possibly
non-convex, identified set. The extension to smooth estimating equations shows
how the same representation can be used more broadly, although the root-$n$
zero-set result obtained here is oracle and relies on Donsker and error-bound
conditions. The Hadamard construction has a complementary role: it provides a
structured diagnostic for approximately zonotopal sensitivity, not a uniformly
controlled substitute for exact optimization. Observable calibration of the
oracle enlargement and error guarantees for adaptive grouping remain important
directions for future work.

\section*{Appendices}
The appendices contain all proofs, more detailed explanations and additional visual illustrations.

\section*{Funding}
This work received no funding.

\section*{Use of artificial intelligence tools}
OpenAI Codex was used as an assistant to check long analytical calculations formally with LEAN 4, to assist with code used to generate the numerical figures; and to check spelling and language.  The author reviewed every AI-assisted output and takes full responsibility for the content and results of the article.

\section*{Data availability}
The numerical experiments use simulated data only. No third-party data are
used. Code and numerical files are available from the author upon reasonable
request.

\section*{Conflict of interest}
The author declares no conflict of interest.

\clearpage
\appendix
\numberwithin{theorem}{subsection}
\section{Proofs for the Population Identified Set}
\label{app:section4}
This appendix contains the proofs for Section~\ref{sec:random_sets}.

\subsection{Proof of Theorem~\ref{thm:identification_equivalence}}
\label{app:proofplan_z_ols}

In the least-squares setting, the estimating function is $\varphi((X,Y,M),\beta) = X(X^T \beta - Y)$ for the complete data $(X, Y, M)$. For any admissible full-data law $P \in \mathcal{P}_{\mathrm{adm}}$, let $\Sigma_P = \mathbb{E}_P[XX^T]$ and $\gamma_P = \mathbb{E}_P[XY]$. Assumption~\ref{ass:ure} gives $\Sigma_P\succeq\kappa I_p$, so the unique root is $\beta(P) = \Sigma_P^{-1} \gamma_P$.
Define the vector-valued map $\phi(X,Y,M) = (XX^T, XY)$. Under Assumptions~\ref{ass:sampling} and~\ref{ass:integrable_envelope}, the random compact set $\mathbf{W}$ whose realization at $\omega\in\Omega$ is $\mathbf{W}(\omega) = \{\phi(z) : z \in \mathbf{Z}(\omega)\}$ (which corresponds to $\{\phi(x, Y(\omega), M(\omega)) : x \in \mathbf{X}(\omega)\}$) is measurable, compact-valued, and integrably bounded. Applying Theorem~\ref{thm:z_sharp_equivalence} to the estimating function $\phi$, we obtain:
\[
    \E_{\mathrm{Aumann}}[\mathbf{W}] 
    = \left\{ (\Sigma_P, \gamma_P) : P \in \mathcal{P}_{\mathrm{adm}} \right\}.
\]
Since the mapping $\Phi(\Sigma, \gamma) = \Sigma^{-1} \gamma$ is continuous on the open subset of positive-definite matrices $\mathbb{S}_{++}^p \times \mathbb{R}^p$, the image of the set of attainable moments under $\Phi$ is:
\[
    \Bstar_{\mathrm{Obs}} = \left\{ \Sigma_P^{-1} \gamma_P : P \in \mathcal{P}_{\mathrm{adm}} \right\} 
    = \left\{ \Sigma^{-1} \gamma : (\Sigma, \gamma) \in \E_{\mathrm{Aumann}}[\mathbf{W}] \right\}.
\]
This precisely matches the Aumann formulation defined in Section~\ref{sec:random_sets}, establishing that the least-squares equivalence is a direct consequence of Theorem~\ref{thm:z_sharp_equivalence}.

\subsection{Proof of Proposition~\ref{prop:Kadm_properties}}
\label{app:proof_aumann_properties}

\begin{proof}
Under Assumption~\ref{ass:sampling}, the random set $\mathbf{X}\subset\mathbb{R}^p$ is non-empty and compact-valued almost surely. For each $\omega$, the map
\[
    z\longmapsto (zz^T,zY(\omega))
\]
is continuous from $\mathbb{R}^p$ to $\mathbb{R}^{p\times p}\times\mathbb{R}^p$. Hence
\[
    \mathbf{W}(\omega)
    =
    \{(zz^T,zY(\omega)):z\in\mathbf{X}(\omega)\}
\]
is compact-valued almost surely as the continuous image of the compact set $\mathbf{X}(\omega)$.

Moreover, $\mathbf{W}$ is integrably bounded. Indeed, for any $z\in\mathbf{X}$,
\[
    \norm{(zz^T,zY)}
    =
    \left(\norm{zz^T}_F^2+\norm{zY}_2^2\right)^{1/2}
    \le
    \norm{z}_2^2+|Y|\norm{z}_2.
\]
Therefore
\[
    \sup_{w\in\mathbf{W}}\norm{w}
    \le
    \sup_{z\in\mathbf{X}}\norm{z}_2^2
    +
    |Y|\sup_{z\in\mathbf{X}}\norm{z}_2,
\]
and the right-hand side is integrable by Assumption~\ref{ass:integrable_envelope}.

Thus $\mathbf{W}$ is an integrably bounded compact random set. Its Aumann expectation is therefore compact in $\mathbb{R}^{p\times p}\times\mathbb{R}^p$ \citep[Theorem~2.1.20]{Molchanov2005}. Since the underlying probability space is non-atomic, the Aumann expectation is also convex, even though $\mathbf{W}$ itself need not be convex-valued. Consequently,
\begin{equation}
    \mathcal{K}^*_{\mathrm{adm}} = \mathbb{E}_{\mathrm{Aumann}}[\mathbf{W}]
\end{equation}
is a compact and convex subset of $\mathbb{R}^{p \times p} \times \mathbb{R}^p$.
For any integrable selection $Z$ of $\mathbf{X}$, the second-moment block satisfies
$\E[ZZ^T]\succeq0$ because $ZZ^T\succeq0$ almost surely.
\end{proof}

\subsection{Proof of Proposition~\ref{prop:l1_K}}
\label{app:proof_l1_K}

\begin{proof}
Let
\[
    \mathcal{S}
    =
    \left\{
        U=(U_\Sigma,U_\gamma)\in\mathbb{S}^p\times\mathbb{R}^p:
        \norm{U_\Sigma}_F^2+\norm{U_\gamma}_2^2=1
    \right\},
\]
where $U_\Sigma$ is symmetric. Throughout the non-asymptotic proofs, we
identify $\mathbb S^p\times\mathbb R^p$ with $\mathbb R^d$, where
$d=p(p+1)/2+p$, and use the Euclidean norm
$\norm{(\Sigma,\gamma)}_2^2=\norm{\Sigma}_F^2+\norm{\gamma}_2^2$.

\begin{lemma}[Support-function representation]\label{lem:support_representation}
For $U\in\mathcal{S}$, set
\[
    W_i(U)=
    \sup_{z_i\in\mathbf{X}_i}
    \left(z_i^TU_\Sigma z_i+z_i^TU_\gamma Y_i\right),
    \qquad
    W(U)=
    \sup_{z\in\mathbf{X}}
    \left(z^TU_\Sigma z+z^TU_\gamma Y\right).
\]
The support-function representation of the empirical and population moment sets gives
\[
    \psi_{\mathrm{conv}(\hat{\mathcal{K}}_{\mathrm{adm}})}(U)
    =
    \frac1n\sum_{i=1}^n W_i(U),
    \qquad
    \psi_{\mathcal{K}^*_{\mathrm{adm}}}(U)
    =
    \E[W(U)].
\]
\end{lemma}
The displayed envelope in Assumption~\ref{ass:integrable_envelope} is finite almost surely and integrable, so the support functions above are finite for every fixed $U\in\mathcal{S}$.
Therefore, by identity~\eqref{eq:hausdorff_support}, it is enough to prove
\[
    \sup_{U\in\mathcal{S}}
    \left|
        \frac1n\sum_{i=1}^n W_i(U)-\E[W(U)]
    \right|
    \longrightarrow0
    \qquad\text{a.s.}
\]

Assumption~\ref{ass:integrable_envelope} implies that $W(U)$ is integrable for every fixed $U$, since
\[
    |W(U)|
    \le
    \sup_{z\in\mathbf{X}}\norm{z}_2^2
    +
    |Y|\sup_{z\in\mathbf{X}}\norm{z}_2.
\]
The same assumption gives, by the strong law,
\[
    \frac1n\sum_{i=1}^n
    \left(
        \sup_{z_i\in\mathbf{X}_i}\norm{z_i}_2^2
        +
        |Y_i|\sup_{z_i\in\mathbf{X}_i}\norm{z_i}_2
    \right)
    \longrightarrow
    \E\!\left[
        \sup_{z\in\mathbf{X}}\norm{z}_2^2
        +
        |Y|\sup_{z\in\mathbf{X}}\norm{z}_2
    \right]
    \qquad\text{a.s.}
\]

Fix $\rho>0$ and choose a finite $\rho$-net $\mathcal{N}_\rho$ of $\mathcal{S}$. For every fixed $U_0\in\mathcal{N}_\rho$, the ordinary strong law yields
\[
    \frac1n\sum_{i=1}^n W_i(U_0)-\E[W(U_0)]
    \longrightarrow0
    \qquad\text{a.s.}
\]
Since $\mathcal{N}_\rho$ is finite, this convergence is uniform over $U_0\in\mathcal{N}_\rho$.

For $U,V\in\mathcal{S}$, the elementary bound
\[
    |W_i(U)-W_i(V)|
    \le
    \sqrt{2}
    \left(
        \sup_{z_i\in\mathbf{X}_i}\norm{z_i}_2^2
        +
        |Y_i|\sup_{z_i\in\mathbf{X}_i}\norm{z_i}_2
    \right)
    \norm{U-V}_2
\]
follows by applying Cauchy--Schwarz to the matrix and vector blocks. The analogous bound holds at the population level after replacing $(\mathbf{X}_i,Y_i)$ by $(\mathbf{X},Y)$. Hence, if $U_0\in\mathcal{N}_\rho$ satisfies $\norm{U-U_0}_2\le\rho$, then
\[
\begin{aligned}
    &\left|
        \frac1n\sum_{i=1}^n W_i(U)-\E[W(U)]
    \right|\\
    &\quad\le
    \left|
        \frac1n\sum_{i=1}^n W_i(U_0)-\E[W(U_0)]
    \right|\\
    &\qquad+
    \sqrt{2}
    \left[
        \frac1n\sum_{i=1}^n
        \left(
            \sup_{z_i\in\mathbf{X}_i}\norm{z_i}_2^2
            +
            |Y_i|\sup_{z_i\in\mathbf{X}_i}\norm{z_i}_2
        \right)
        +
        \E\!\left[
            \sup_{z\in\mathbf{X}}\norm{z}_2^2
            +
            |Y|\sup_{z\in\mathbf{X}}\norm{z}_2
        \right]
    \right]\rho .
\end{aligned}
\]
Taking the supremum over $U$, then the limsup in $n$, gives almost surely
\[
    \limsup_{n\to\infty}
    \sup_{U\in\mathcal{S}}
    \left|
        \frac1n\sum_{i=1}^n W_i(U)-\E[W(U)]
    \right|
    \le
    2\sqrt{2}\,
    \E\!\left[
        \sup_{z\in\mathbf{X}}\norm{z}_2^2
        +
        |Y|\sup_{z\in\mathbf{X}}\norm{z}_2
    \right]\rho .
\]
Since $\rho>0$ is arbitrary, the uniform convergence follows, and therefore~\eqref{eq:l1_K} holds.
\end{proof}

\subsection{Proof of Proposition~\ref{prop:Bstar}}
\label{app:proof_bstar_compactness}

\begin{proof}
By Proposition~\ref{prop:Kadm_properties}, $\mathcal{K}^*_{\mathrm{adm}}$ is compact. By Assumption~\ref{ass:ure}, every $(\Sigma,\gamma)\in\mathcal{K}^*_{\mathrm{adm}}$ satisfies $\lambda_{\min}(\Sigma)\ge\kappa>0$. Hence $\mathcal{K}^*_{\mathrm{adm}}$ is contained in the closed domain
\[
    \mathcal{D}_\kappa
    =
    \{(\Sigma,\gamma): \lambda_{\min}(\Sigma)\ge\kappa\}.
\]
On $\mathcal{D}_\kappa$, all matrices $\Sigma$ are invertible and the inversion map $\Sigma\mapsto\Sigma^{-1}$ is continuous. Therefore the map
\begin{equation}
    \Phi(\Sigma, \gamma) = \Sigma^{-1} \gamma.
\end{equation}
is continuous on $\mathcal{D}_\kappa$, and in particular on $\mathcal{K}^*_{\mathrm{adm}}$. By definition,
\begin{equation}
    \Bstar = \Phi\left( \mathcal{K}^*_{\mathrm{adm}} \right).
\end{equation}
Thus $\Bstar$ is the continuous image of the compact set $\mathcal{K}^*_{\mathrm{adm}}$, and is therefore compact in $\mathbb{R}^p$.
\end{proof}

\subsection{Proof of Proposition~\ref{prop:l1_B}}
\label{app:proof_l1_B}

\begin{lemma}[Shapley--Folkman--Starr bound]\label{lem:sfs_bound}
Let $A_1,\ldots,A_n$ be non-empty compact subsets of $\mathbb R^d$. Then
\[
 \dH\!\left(\sum_{i=1}^n A_i,
             \operatorname{conv}\!\left(\sum_{i=1}^n A_i\right)\right)
 \le \sqrt d\,
     \max_{1\le i\le n}\inf_{c\in\mathbb R^d}
        \sup_{a\in A_i}\norm{a-c}_2
 \le \sqrt d\max_{1\le i\le n}\operatorname{diam}(A_i).
\]
\end{lemma}

\begin{proof}
The first inequality is the Euclidean Shapley--Folkman--Starr bound in
\citet[Appendix~2, pp.~35--37]{Starr1969}. The second follows by choosing
any $c\in A_i$.
\end{proof}

\begin{lemma}[Integrable-envelope convergence of the non-convex moment set]\label{lem:l1_K_nonconv}
Under Assumptions~\ref{ass:sampling} and~\ref{ass:integrable_envelope},
\[
    \dH\bigl(\hat{\mathcal{K}}_{\mathrm{adm}},\mathcal{K}^*_{\mathrm{adm}}\bigr)
    \longrightarrow0
    \qquad\text{a.s.}
\]
\end{lemma}

\begin{proof}
By Proposition~\ref{prop:l1_K},
\[
    \dH\bigl(\mathrm{conv}(\hat{\mathcal{K}}_{\mathrm{adm}}),\mathcal{K}^*_{\mathrm{adm}}\bigr)
    \longrightarrow0
    \qquad\text{a.s.}
\]
It remains to prove that the Shapley--Folkman remainder vanishes. Set
\[
    V=
    \sup_{z\in\mathbf{X}}\norm{z}_2^2
    +
    |Y|\sup_{z\in\mathbf{X}}\norm{z}_2.
\]
By Assumption~\ref{ass:integrable_envelope}, $\E[V]<\infty$. Since
\[
    \E[V]=\int_0^\infty \PP(V>t)\,dt
\]
and $t\mapsto\PP(V>t)$ is non-increasing, for every $\varepsilon>0$,
\[
    \sum_{n=1}^\infty \PP(V>\varepsilon n)
    \le
    \frac{1}{\varepsilon}\E[V]
    <\infty.
\]
Thus Borel--Cantelli gives
\[
    \frac{
        \sup_{z_n\in\mathbf{X}_n}\norm{z_n}_2^2
        +
        |Y_n|\sup_{z_n\in\mathbf{X}_n}\norm{z_n}_2
    }{n}
    \longrightarrow0
    \qquad\text{a.s.}
\]
Consequently,
\[
    \frac1n
    \max_{1\le i\le n}
    \left(
        \sup_{z_i\in\mathbf{X}_i}\norm{z_i}_2^2
        +
        |Y_i|\sup_{z_i\in\mathbf{X}_i}\norm{z_i}_2
    \right)
    \longrightarrow0
    \qquad\text{a.s.}
\]
Lemma~\ref{lem:sfs_bound}, applied in dimension $d=p(p+3)/2$ with center $c=0$, gives
\[
    \dH\bigl(\hat{\mathcal{K}}_{\mathrm{adm}},
    \mathrm{conv}(\hat{\mathcal{K}}_{\mathrm{adm}})\bigr)
    \le
    \frac{\sqrt{d}}{n}
    \max_{1\le i\le n}
    \left(
        \sup_{z_i\in\mathbf{X}_i}\norm{z_i}_2^2
        +
        |Y_i|\sup_{z_i\in\mathbf{X}_i}\norm{z_i}_2
    \right),
\]
and the right-hand side converges to zero almost surely. Since $\mathcal{K}^*_{\mathrm{adm}}$ is convex,
\[
\begin{aligned}
    \dH\bigl(\hat{\mathcal{K}}_{\mathrm{adm}},\mathcal{K}^*_{\mathrm{adm}}\bigr)
    &\le
    \dH\bigl(\hat{\mathcal{K}}_{\mathrm{adm}},
    \mathrm{conv}(\hat{\mathcal{K}}_{\mathrm{adm}})\bigr)\\
    &\quad+
    \dH\bigl(\mathrm{conv}(\hat{\mathcal{K}}_{\mathrm{adm}}),
    \mathcal{K}^*_{\mathrm{adm}}\bigr)
    \longrightarrow0
    \qquad\text{a.s.}
\end{aligned}
\]
\end{proof}

\begin{proof}[Proof of Proposition~\ref{prop:l1_B}]
By Proposition~\ref{prop:Bstar}, the population identified set $\Bstar$ is compact, and the Hausdorff distance in~\eqref{eq:l1_B} is well defined.
By Lemma~\ref{lem:l1_K_nonconv},
\[
    \dH\bigl(\hat{\mathcal{K}}_{\mathrm{adm}},\mathcal{K}^*_{\mathrm{adm}}\bigr)
    \longrightarrow0
    \qquad\text{a.s.}
\]
Let
\[
    G_0=
    1+
    \E\!\left[
        \sup_{z\in\mathbf{X}}\norm{z}_2^2
        +
        |Y|\sup_{z\in\mathbf{X}}\norm{z}_2
    \right].
\]
For every population admissible moment pair, Jensen's inequality gives
\[
    \norm{\gamma}_2
    \le
    \E\!\left[
        |Y|\sup_{z\in\mathbf{X}}\norm{z}_2
    \right]
    \le G_0.
\]
For every empirical admissible moment pair,
\[
    \norm{\gamma}_2
    \le
    \frac1n\sum_{i=1}^n
    |Y_i|\sup_{z_i\in\mathbf{X}_i}\norm{z_i}_2
    \le
    \frac1n\sum_{i=1}^n
    \left(
        \sup_{z_i\in\mathbf{X}_i}\norm{z_i}_2^2
        +
        |Y_i|\sup_{z_i\in\mathbf{X}_i}\norm{z_i}_2
    \right).
\]
The strong law therefore implies that, almost surely, all empirical admissible moment pairs satisfy $\norm{\gamma}_2\le G_0$ for all sufficiently large $n$.

By Assumption~\ref{ass:ure}, every population admissible moment pair in $\mathcal{K}^*_{\mathrm{adm}}$ and every empirical admissible moment pair in $\hat{\mathcal{K}}_{\mathrm{adm}}$ satisfies the same restricted eigenvalue bound $\lambda_{\min}(\Sigma)\ge\kappa$ almost surely. Hence, eventually almost surely, both $\hat{\mathcal{K}}_{\mathrm{adm}}$ and $\mathcal{K}^*_{\mathrm{adm}}$ are contained in
\[
    \mathcal{D}_{G_0}
    =
    \left\{
        (\Sigma,\gamma):
        \lambda_{\min}(\Sigma)\ge\kappa,\,
        \norm{\gamma}_2\le G_0
    \right\}.
\]
On this deterministic domain, the map $\Phi(\Sigma,\gamma)=\Sigma^{-1}\gamma$ is Lipschitz with constant
\[
    L_\Phi(G_0)
    =
    \sqrt{\frac{1}{\kappa^2}+\frac{G_0^2}{\kappa^4}}.
\]
Indeed, for $K_j=(\Sigma_j,\gamma_j)\in\mathcal{D}_{G_0}$,
\[
\begin{aligned}
    \norm{\Phi(K_1)-\Phi(K_2)}_2
    &\le
    \frac{1}{\kappa}\norm{\gamma_1-\gamma_2}_2
    +
    \frac{G_0}{\kappa^2}\norm{\Sigma_1-\Sigma_2}_F\\
    &\le
    L_\Phi(G_0)\norm{K_1-K_2}_2.
\end{aligned}
\]
Consequently, for all sufficiently large $n$,
\[
    \dH(\Bhat,\Bstar)
    \le
    L_\Phi(G_0)
    \dH\bigl(\hat{\mathcal{K}}_{\mathrm{adm}},\mathcal{K}^*_{\mathrm{adm}}\bigr)
    \longrightarrow0
    \qquad\text{a.s.}
\]
This proves~\eqref{eq:l1_B}.
\end{proof}

\subsection{Bounded-envelope moment-set bound}
\label{app:proof_uniform}

\begin{lemma}[Error bounds for $\hat{\mathcal{K}}_{\mathrm{adm}}$ estimation]
\label{lem:uniform}
Under Assumptions~\ref{ass:sampling} and~\ref{ass:bounded_xy}, let $L = \sqrt{M_X^4+M_X^2M_Y^2}$ and $d = p(p+1)/2 + p$. For every $\delta\in(0,1)$, with probability at least $1-\delta$,
\begin{equation}\label{eq:uniform_bound}
  \dH\bigl( \mathrm{conv}(\hat{\mathcal{K}}_{\mathrm{adm}}), \mathcal{K}^*_{\mathrm{adm}} \bigr)
  \le
  L\sqrt{\frac{2\log(1/\delta)}{n}}
  +\frac{(\sqrt{2\pi d}+2)L}{\sqrt n}.
\end{equation}
\end{lemma}

\begin{lemma}[Lipschitz Continuity of the Support Function]
\label{lem:support_lipschitz}
For any direction $U, V \in \mathcal{S}$ and any observation $i$:
\begin{align*}
    |W_i(U)|
    &\le
    \sup_{z_i\in\mathbf{X}_i}\norm{z_i}_2^2
    +
    |Y_i|\sup_{z_i\in\mathbf{X}_i}\norm{z_i}_2, \\
    |W_i(U)-W_i(V)|
    &\le
    \left(
    \sup_{z_i\in\mathbf{X}_i}\norm{z_i}_2^2
    +
    |Y_i|\sup_{z_i\in\mathbf{X}_i}\norm{z_i}_2
    \right)
    \norm{U-V}_2.
\end{align*}
Under Assumption~\ref{ass:bounded_xy}, with $L = \sqrt{M_X^4+M_X^2M_Y^2}$, the support function summand is uniformly bounded and Lipschitz continuous on the unit sphere $\mathcal{S}$:
\begin{equation}
    |W_i(U)| \le L, \qquad |W_i(U) - W_i(V)| \le L \norm{U - V}_2.
\end{equation}
\end{lemma}

\begin{proof}
For any $z_i\in\mathbf{X}_i$,
\[
    |z_i^TU_\Sigma z_i+z_i^TU_\gamma Y_i|
    \le \norm{z_i}_2^2\norm{U_\Sigma}_F+\norm{z_i}_2|Y_i|\norm{U_\gamma}_2
    \le
    \sup_{z_i\in\mathbf{X}_i}\norm{z_i}_2^2
    +
    |Y_i|\sup_{z_i\in\mathbf{X}_i}\norm{z_i}_2,
\]
because $\norm{U_\Sigma}_F,\norm{U_\gamma}_2\le1$. The same argument applied to $U-V$, followed by Cauchy--Schwarz on $\mathbb{R}^2$ with vectors $(\sup_{u\in\mathbf{X}_i}\norm{u}_2^2, |Y_i|\sup_{u\in\mathbf{X}_i}\norm{u}_2)$ and $(\norm{U_\Sigma-V_\Sigma}_F, \norm{U_\gamma-V_\gamma}_2)$, gives the Lipschitz bound. 

Under Assumption~\ref{ass:bounded_xy}, any admissible selection $z_i \in \mathbf{X}_i$ satisfies $\norm{z_i}_2 \le M_X$ and the response variable satisfies $|Y_i| \le M_Y$ almost surely. By Cauchy-Schwarz on matrices and vectors:
\begin{align*}
    |W_i(U)| &\le \sup_{z_i \in \mathbf{X}_i} \left( \abs{z_i^T U_{\Sigma} z_i} + \abs{z_i^T U_{\gamma} Y_i} \right) \\
    &\le M_X^2 \norm{U_{\Sigma}}_F + M_X M_Y \norm{U_{\gamma}}_2.
\end{align*}
Applying Cauchy-Schwarz on $\mathbb{R}^2$ with vectors $(M_X^2, M_X M_Y)$ and $(\norm{U_{\Sigma}}_F, \norm{U_{\gamma}}_2)$ yields:
\begin{align*}
    |W_i(U)| &\le \sqrt{(M_X^2)^2 + (M_X M_Y)^2} \sqrt{\norm{U_{\Sigma}}_F^2 + \norm{U_{\gamma}}_2^2} \\
    &= \sqrt{M_X^4 + M_X^2 M_Y^2}=L.
\end{align*}
Similarly, for any $U, V \in \mathcal{S}$:
\begin{align*}
    \abs{W_i(U) - W_i(V)} &\le \sup_{z_i \in \mathcal{X}_{\mathrm{adm}, i}} \abs{ z_i^T (U_{\Sigma} - V_{\Sigma}) z_i + z_i^T (U_{\gamma} - V_{\gamma}) Y_i } \\
    &\le M_X^2 \norm{U_{\Sigma} - V_{\Sigma}}_F + M_X M_Y \norm{U_{\gamma} - V_{\gamma}}_2 \\
    &\le L \norm{U - V}_2,
\end{align*}
which completes the proof.
\end{proof}

\begin{lemma}[Bounded support-process concentration]\label{lem:bounded_support_chaining}
Suppose that $|W_i(U)|\le L$ and
$|W_i(U)-W_i(V)|\le L\norm{U-V}_2$ almost surely for all
$i$ and $U,V\in\mathcal S$. For every $\delta\in(0,1)$, with probability at least $1-\delta$,
\[
    \sup_{U\in\mathcal{S}}
    \left|\frac1n\sum_{i=1}^nW_i(U)-\E[W(U)]\right|
    \le
    L\sqrt{\frac{2\log(1/\delta)}{n}}
    +\frac{(\sqrt{2\pi d}+2)L}{\sqrt n}.
\]
\end{lemma}

\begin{proof}
Set $F=\sup_{U\in\mathcal S}
|n^{-1}\sum_{i=1}^nW_i(U)-\E[W(U)]|$.
Continuity allows all suprema below to be taken over a countable dense
subset of $\mathcal S$. For independent Rademacher variables
$\varepsilon_1,\ldots,\varepsilon_n$, symmetrization
\citep[Section~2.3]{vanDerVaartWellner1996} gives
\[
 \E F\le2\E\E_\varepsilon
 \sup_{U\in\mathcal S}\left|\frac1n
       \sum_{i=1}^n\varepsilon_iW_i(U)\right|.
\]
Condition on the sample. More generally, let $A_i\ge0$ satisfy
$|W_i(U)|\le A_i$ and
$|W_i(U)-W_i(V)|\le A_i\norm{U-V}_2$; here $A_i=L$.
Put $a_n=n^{-1}(\sum_iA_i^2)^{1/2}$.
Take independent standard Gaussian variables $g_0,g_1,\ldots,g_n$
and an independent $g\sim N(0,I_d)$. On
$\mathcal S\times\{-1,1\}$, compare
\[
 G_{U,s}=\frac{s}{n}\sum_{i=1}^ng_iW_i(U),
 \qquad H_{U,s}=a_n\{\langle g,U\rangle+s g_0\}.
\]
Their increments satisfy
\[
 \E_g(G_{U,s}-G_{V,t})^2
 \le a_n^2
 \begin{cases}
   \norm{U-V}_2^2,&s=t,\\
   4,&s\ne t
 \end{cases}
 \le a_n^2\{\norm{U-V}_2^2+(s-t)^2\}
 =\E_g(H_{U,s}-H_{V,t})^2.
\]
The Sudakov--Fernique inequality
\citep[Theorem~7.2.11]{Vershynin2018} therefore yields
\[
 \E_g\sup_{U,s}G_{U,s}
 \le \E_g\sup_{U,s}H_{U,s}
 =a_n\{\E\norm{g}_2+\E|g_0|\}
 \le a_n\{\sqrt d+\sqrt{2/\pi}\}.
\]
Writing each $g_i$ as its independent sign and absolute value, Jensen's
inequality and $\E|g_i|=\sqrt{2/\pi}$ give
\[
 \E_\varepsilon\sup_{U\in\mathcal S}
 \left|\frac1n\sum_{i=1}^n\varepsilon_iW_i(U)\right|
 \le\sqrt{\frac{\pi}{2}}\,\E_g\sup_{U,s}G_{U,s}
 \le\frac{\sqrt{\pi d/2}+1}{n}
       \left(\sum_{i=1}^nA_i^2\right)^{1/2}.
\]
With $A_i=L$, symmetrization gives
$\E F\le(\sqrt{2\pi d}+2)L/\sqrt n$.

Replacing one observation changes $F$ by at most $2L/n$.
McDiarmid's inequality \citep{McDiarmid1989} consequently gives, for every
$x\ge0$,
\[
 \PP\left(F>\E F+L\sqrt{\frac{2x}{n}}\right)\le e^{-x}.
\]
Taking $x=\log(1/\delta)$ proves the result.
\end{proof}

\begin{proof}[Proof of Lemma~\ref{lem:uniform}]
By Lemma~\ref{lem:support_representation} and identity~\eqref{eq:hausdorff_support},
\[
    \dH\bigl(\mathrm{conv}(\hat{\mathcal K}_{\mathrm{adm}}),\mathcal K^*_{\mathrm{adm}}\bigr)
    =
    \sup_{U\in\mathcal S}
    \left|\frac1n\sum_{i=1}^nW_i(U)-\E[W(U)]\right|.
\]
Under Assumption~\ref{ass:bounded_xy}, Lemma~\ref{lem:support_lipschitz} gives
$|W_i(U)|\le L$ and
$|W_i(U)-W_i(V)|\le L\norm{U-V}_2$. Lemma~\ref{lem:bounded_support_chaining}
proves~\eqref{eq:uniform_bound}.
\end{proof}

\subsection{Proof of Theorem~\ref{thm:main_bounded}}
\label{app:proof_main_asymptotic}

The following lemma records the Lipschitz continuity of the parameter
identification mapping.

\begin{lemma}[Lipschitz Continuity of the Coefficient Mapping]
\label{lem:phi_lipschitz}
Let $a>0$ and $G<\infty$, and define
\[
    \mathcal{D}_{a,G}
    =
    \{(\Sigma,\gamma):\lambda_{\min}(\Sigma)\ge a,\ \norm{\gamma}_2\le G\}.
\]
The coefficient mapping $\Phi(\Sigma,\gamma)=\Sigma^{-1}\gamma$ is Lipschitz continuous on $\mathcal{D}_{a,G}$ with constant
\[
    L_\Phi(a,G)=\sqrt{\frac1{a^2}+\frac{G^2}{a^4}}.
\]
That is, for all $K_1,K_2\in\mathcal{D}_{a,G}$,
\begin{equation}
    \norm{\Phi(K_1) - \Phi(K_2)}_2 \le L_\Phi(a,G) \norm{K_1 - K_2}_2.
\end{equation}
\end{lemma}

\begin{proof}
For any $K_1 = (\Sigma_1, \gamma_1)$ and $K_2 = (\Sigma_2, \gamma_2)$ in $\mathcal{D}_{a,G}$, we have $\norm{\Sigma_1^{-1}}_2\le1/a$, $\norm{\Sigma_2^{-1}}_2\le1/a$, and $\norm{\gamma_2}_2\le G$. Using the matrix identity $\Sigma_1^{-1} - \Sigma_2^{-1} = \Sigma_1^{-1}(\Sigma_2 - \Sigma_1)\Sigma_2^{-1}$, we evaluate the difference:
\begin{align*}
    \norm{\Phi(K_1) - \Phi(K_2)}_2 &= \norm{\Sigma_1^{-1} \gamma_1 - \Sigma_2^{-1} \gamma_2}_2 \\
    &= \norm{\Sigma_1^{-1} \gamma_1 - \Sigma_1^{-1} \gamma_2 + \Sigma_1^{-1} \gamma_2 - \Sigma_2^{-1} \gamma_2}_2 \\
    &= \norm{\Sigma_1^{-1} (\gamma_1 - \gamma_2) + (\Sigma_1^{-1} - \Sigma_2^{-1}) \gamma_2}_2 \\
    &= \norm{\Sigma_1^{-1} (\gamma_1 - \gamma_2) + \Sigma_1^{-1} (\Sigma_2 - \Sigma_1) \Sigma_2^{-1} \gamma_2}_2 \\
    &\le \norm{\Sigma_1^{-1}}_2 \norm{\gamma_1 - \gamma_2}_2 + \norm{\Sigma_1^{-1}}_2 \norm{\Sigma_2 - \Sigma_1}_2 \norm{\Sigma_2^{-1}}_2 \norm{\gamma_2}_2 \\
    &\le \frac{1}{a} \norm{\gamma_1 - \gamma_2}_2 + \frac{G}{a^2} \norm{\Sigma_1 - \Sigma_2}_F.
\end{align*}
Cauchy--Schwarz and the definition of $L_\Phi(a,G)$ give:
\begin{align*}
    \norm{\Phi(K_1) - \Phi(K_2)}_2
    &\le L_\Phi(a,G)\norm{K_1-K_2}_2,
\end{align*}
which completes the proof.
\end{proof}

\begin{lemma}[Shapley--Folkman remainder]\label{lem:sf_remainder}
Let $\mathbf{W}_i=\{(z_iz_i^T,z_iY_i):z_i\in\mathbf{X}_i\}$. Then
\begin{equation}\label{eq:sf_general}
    \dH\bigl(\hat{\mathcal{K}}_{\mathrm{adm}},\mathrm{conv}(\hat{\mathcal{K}}_{\mathrm{adm}})\bigr)
    \le
    \frac{\sqrt{d}}{n}\max_{1\le i\le n}
        \sup_{w\in\mathbf{W}_i}\norm{w}_2
    \le
    \frac{\sqrt{d}}{n}
    \max_{1\le i\le n}
    \left(
        \sup_{z_i\in\mathbf{X}_i}\norm{z_i}_2^2
        +
        |Y_i|\sup_{z_i\in\mathbf{X}_i}\norm{z_i}_2
    \right).
\end{equation}
\end{lemma}

\begin{proof}
The empirical moment set is the Minkowski average $n^{-1}\sum_{i=1}^n\mathbf{W}_i$
in a Euclidean space of dimension $d=p(p+3)/2$.
Lemma~\ref{lem:sfs_bound}, with center $c=0$ and scaling by $1/n$,
gives the first inequality. For the second, every $z\in\mathbf X_i$ satisfies
\[
 \norm{(zz^T,zY_i)}_2
 =\sqrt{\norm{z}_2^4+Y_i^2\norm{z}_2^2}
 \le \sup_{u\in\mathbf X_i}\norm{u}_2^2
       +|Y_i|\sup_{u\in\mathbf X_i}\norm{u}_2.
\]
\end{proof}

\begin{lemma}[Deterministic transport from moments to coefficients]\label{lem:transport_B}
Let
\[
    G_n=\max\left\{
    \sup_{(\Sigma,\gamma)\in\mathcal{K}^*_{\mathrm{adm}}}\norm{\gamma}_2,\,
    \sup_{(\Sigma,\gamma)\in\hat{\mathcal{K}}_{\mathrm{adm}}}\norm{\gamma}_2
    \right\}.
\]
On any event where $G_n<\infty$ and the empirical and population second-moment blocks satisfy the URE lower bound $\kappa$,
\begin{equation}\label{eq:transport_B}
\begin{aligned}
    &\dH\bigl(\Bhat,\Bstar\bigr)\\
    &\quad\le
    \sqrt{\frac1{\kappa^2}+\frac{G_n^2}{\kappa^4}}
    \Bigg[
        \dH\bigl(\mathrm{conv}(\hat{\mathcal{K}}_{\mathrm{adm}}),\mathcal{K}^*_{\mathrm{adm}}\bigr)\\
    &\hspace{6.3cm}
        +\frac{\sqrt{d}}{n}
        \max_{1\le i\le n}
        \left(
            \sup_{z_i\in\mathbf{X}_i}\norm{z_i}_2^2
            +|Y_i|\sup_{z_i\in\mathbf{X}_i}\norm{z_i}_2
        \right)
    \Bigg].
\end{aligned}
\end{equation}
\end{lemma}

\begin{proof}
Let $K=\hat{\mathcal{K}}_{\mathrm{adm}}$ and $K^*=\mathcal{K}^*_{\mathrm{adm}}$. Lemma~\ref{lem:phi_lipschitz}, applied on the union of $K$ and $K^*$ with $G=G_n$, yields
\[
    \dH(\Phi(K),\Phi(K^*))
    \le
    \sqrt{\frac1{\kappa^2}+\frac{G_n^2}{\kappa^4}}\dH(K,K^*).
\]
Since $\Bhat=\Phi(K)$ and $\Bstar=\Phi(K^*)$, this directly bounds
$\dH(\Bhat,\Bstar)$.
Since $K^*$ is convex, the triangle inequality gives
\[
    \dH(K,K^*)
    \le
    \dH\bigl(K,\mathrm{conv}(K)\bigr)
    +\dH\bigl(\mathrm{conv}(K),K^*\bigr).
\]
Lemma~\ref{lem:sf_remainder} completes the proof.
\end{proof}

\begin{proof}[Proof of Theorem~\ref{thm:main_bounded}]
For every empirical admissible moment vector,
\[
    \norm{\gamma}_2
    \le
    \frac1n\sum_{i=1}^n|Y_i|\norm{z_i}_2
    \le M_XM_Y.
\]
The same inequality holds at the population level. Thus $G_n\le M_XM_Y$, and Lemma~\ref{lem:phi_lipschitz} gives the constant $\bar L_\Phi^*$ without a factor $\sqrt p$.

Moreover, every element of $\mathbf W_i$ has norm at most $L$.
The first inequality of Lemma~\ref{lem:sf_remainder} gives
\[
    \dH\bigl(\hat{\mathcal K}_{\mathrm{adm}},\mathrm{conv}(\hat{\mathcal K}_{\mathrm{adm}})\bigr)
    \le
    \frac{L\sqrt d}{n}.
\]
Combining this inequality, Lemma~\ref{lem:uniform}, and Lemma~\ref{lem:transport_B} yields
\[
    \dH\bigl(\Bhat,\Bstar\bigr)
    \le
    L\bar L_\Phi^*
    \left[
        \sqrt{\frac{2\log(1/\delta)}{n}}
        +\frac{\sqrt{2\pi d}+2}{\sqrt n}
        +\frac{\sqrt d}{n}
    \right]
\]
with probability at least $1-\delta$.
\end{proof}

\subsection{Sub-exponential moment-set bound}
\label{app:proof_subexp_K}

\begin{proposition}[Sub-exponential bounds]\label{prop:subexp_K}
Under Assumptions~\ref{ass:sampling} and~\ref{ass:subexp_envelope}, let $d=p(p+1)/2+p$. For every $\delta\in(0,1)$, with probability at least $1-\delta$,
\begin{equation}\label{eq:subexp_bound}
\begin{aligned}
  &\dH\bigl(\mathrm{conv}(\hat{\mathcal{K}}_{\mathrm{adm}}),\mathcal{K}^*_{\mathrm{adm}}\bigr)\\
  &\quad\le K_W\left[
  \begin{aligned}
      &2\sqrt{\frac{\log(1/\delta)+3\log2+d\log3}{n}}\\
      &+4\sqrt{\frac{\log(1/\delta)+5\log2+d\log12}{n}}\\
      &+\frac{3\log(1/\delta)+13\log2+d(\log3+2\log12)}{n}
  \end{aligned}
  \right].
\end{aligned}
\end{equation}
\end{proposition}

\begin{lemma}[Sub-exponential support-process concentration]\label{lem:subexp_support_chaining}
Let
\[
    A_i
    =
    \sup_{z_i\in\mathbf X_i}\norm{z_i}_2^2
    +|Y_i|\sup_{z_i\in\mathbf X_i}\norm{z_i}_2,
\]
and let $A$ be the corresponding population quantity. If $\norm{A}_{\psi_1}\le K_W$,
then, for every $\delta\in(0,1)$, with probability at least $1-\delta$,
\[
\begin{aligned}
    &\sup_{U\in\mathcal S}
    \left|\frac1n\sum_{i=1}^nW_i(U)-\E[W(U)]\right|\\
    &\quad\le K_W\left[
    \begin{aligned}
        &2\sqrt{\frac{\log(1/\delta)+3\log2+d\log3}{n}}\\
        &+4\sqrt{\frac{\log(1/\delta)+5\log2+d\log12}{n}}\\
        &+\frac{3\log(1/\delta)+13\log2+d(\log3+2\log12)}{n}
    \end{aligned}
    \right].
\end{aligned}
\]
\end{lemma}

\begin{proof}
We use $\norm{Z}_{\psi_1}=\inf\{K>0:\E e^{|Z|/K}\le2\}$.
If $K_W=0$, the result is immediate. If $\norm{X}_{\psi_1}\le b$ with
$b>0$, then $e^u\ge1+u^m/m!$ for $u\ge0$ gives
$\E|X|^m\le m!b^m$. Thus, for $|t|b<1$,
\[
 \log\E e^{t(X-\E X)}
 \le \E e^{tX}-1-t\E X
 \le\sum_{m\ge2}|t|^m b^m
 =\frac{t^2b^2}{1-|t|b}.
\]
This is the Bernstein moment bound with variance proxy $2b^2$ and
scale $b$, obtained without first enlarging the norm by centering.

Let $\mathcal T=\mathcal S\times\{-1,1\}$ and
\[
 \rho((U,s),(V,t))=
 \begin{cases}\norm{U-V}_2,&s=t,\\2,&s\ne t.\end{cases}
\]
By Lemma~\ref{lem:support_lipschitz},
$\norm{sW(U)-tW(V)}_{\psi_1}\le K_W\rho((U,s),(V,t))$.
Independence and the preceding moment bound, followed by Bernstein's
Chernoff argument \citep[Section~2.8]{Vershynin2018}, yield
\[
 \PP\left(|Z_v-Z_w|>
 K_W\rho(v,w)\left\{2\sqrt{\frac{x}{n}}+\frac{x}{n}\right\}\right)
 \le2e^{-x},
\]
where $Z_{(U,s)}=s\{n^{-1}\sum_iW_i(U)-\E W(U)\}$.
The same bound with $\rho(v,w)$ replaced by $1$ applies to $|Z_v|$.

A maximal $\varepsilon$-separated subset of $\mathcal S$ is an
$\varepsilon$-net. The disjoint balls of radius $\varepsilon/2$ about
its points lie in the ball of radius $1+\varepsilon/2$. Comparing
volumes gives
\[
 N(\mathcal S,\norm{\cdot}_2,\varepsilon)
 \le(1+2/\varepsilon)^d\le(3/\varepsilon)^d,
 \qquad 0<\varepsilon\le1.
\]
For $j\ge0$, take a $2^{-j}$-net $\mathcal T_j$, with
$|\mathcal T_j|\le2(3\,2^j)^d$. For each point of $\mathcal T_j$,
$j\ge1$, choose one parent in $\mathcal T_{j-1}$ within distance
$2^{-(j-1)}$. There are at most $|\mathcal T_j|$ such edges.
Assign level $j\ge1$ the exponent
\[
 x_j=x+(j+2)\log2+\log\{2(3\,2^j)^d\}.
\]
The edge failure probabilities sum to at most
$e^{-x}\sum_{j\ge1}2^{-j-1}=e^{-x}/2$.
The coarsest net, with exponent $x+\log(8\,3^d)$, has failure
probability at most $e^{-x}/2$.
Since $\sum_{j\ge1}2^{-j}=1$ and $\sum_{j\ge1}j2^{-j}=2$,
\[
 \sum_{j\ge1}2^{-j}x_j=x+5\log2+d\log12,
 \qquad
 \sum_{j\ge1}2^{-j}\sqrt{x_j}
 \le\sqrt{x+5\log2+d\log12}.
\]
On the complementary event, telescope from each net point through
its parents to $\mathcal T_0$. Using the edge lengths $2\,2^{-j}$ gives
\[
\begin{aligned}
 &\sup_{U\in\mathcal S}
 \left|\frac1n\sum_{i=1}^nW_i(U)-\E[W(U)]\right|\\
 &\quad\le K_W\left[
 \begin{aligned}
    &2\sqrt{\frac{x+3\log2+d\log3}{n}}\\
    &+4\sqrt{\frac{x+5\log2+d\log12}{n}}\\
    &+\frac{3x+13\log2+d(\log3+2\log12)}{n}
 \end{aligned}
 \right].
\end{aligned}
\]
Continuity completes the passage from the dense union of the nets to
$\mathcal S$; take $x=\log(1/\delta)$.
\end{proof}

\begin{proof}[Proof of Proposition~\ref{prop:subexp_K}]
By Lemma~\ref{lem:support_representation} and identity~\eqref{eq:hausdorff_support},
\[
    \dH\bigl(\mathrm{conv}(\hat{\mathcal K}_{\mathrm{adm}}),\mathcal K^*_{\mathrm{adm}}\bigr)
    =
    \sup_{U\in\mathcal S}
    \left|\frac1n\sum_{i=1}^nW_i(U)-\E[W(U)]\right|.
\]
Assumption~\ref{ass:subexp_envelope} and
Lemma~\ref{lem:subexp_support_chaining} prove~\eqref{eq:subexp_bound}.
\end{proof}

\subsection{Proof of Theorem~\ref{thm:main_subexp}}
\label{app:proof_subexp_B}

For any $i = 1,\dots,n$, we denote by
\begin{equation}\label{eq:envelope_def}
    A_i = \sup_{z_i\in\mathbf{X}_i}\norm{z_i}_2^2 + |Y_i|\sup_{z_i\in\mathbf{X}_i}\norm{z_i}_2
\end{equation}
the envelope of the $i$-th observation, and by $A$ the corresponding population quantity.

\begin{lemma}[Sub-exponential Shapley--Folkman Remainder]\label{lem:subexp_sf}
Suppose Assumption~\ref{ass:subexp_envelope} holds. For any $\delta \in (0,1)$, with probability at least $1-\delta$,
\begin{equation}\label{eq:subexp_event2}
    \dH\bigl(\hat{\mathcal{K}}_{\mathrm{adm}},\mathrm{conv}(\hat{\mathcal{K}}_{\mathrm{adm}})\bigr)
    \le
    \frac{\sqrt d\,K_W\log(2n/\delta)}{n}.
\end{equation}
\end{lemma}

\begin{proof}
The definition of the $\psi_1$ norm and a union bound give
\[
    \PP\left(\max_{1\le i\le n}A_i>K_W\log(2n/\delta)\right)
    \le
    2n\exp\left\{-\log(2n/\delta)\right\}
    =\delta.
\]
Apply Lemma~\ref{lem:sf_remainder} on the complementary event.
\end{proof}

\begin{lemma}[Sub-exponential Lipschitz Envelope]\label{lem:subexp_lipschitz_envelope}
Suppose Assumption~\ref{ass:subexp_envelope} holds. For every $\delta\in(0,1)$, with probability at least $1-\delta$,
\begin{equation}\label{eq:subexp_lipschitz_envelope_bound}
    \frac1n\sum_{i=1}^n A_i
    \le K_W\left(\log 2+\frac{\log(1/\delta)}{n}\right).
\end{equation}
Consequently, $G_n$ is bounded by the right-hand side on the same event.
\end{lemma}

\begin{proof}
If $K_W=0$, the result is immediate. Otherwise, independence and the
$\psi_1$ normalization give
\[
    \E\exp\left(\frac1{K_W}\sum_{i=1}^nA_i\right)
    =\prod_{i=1}^n\E e^{A_i/K_W}\le2^n.
\]
Exponential Markov inequality proves~\eqref{eq:subexp_lipschitz_envelope_bound}.
Moreover, Jensen's inequality gives $\E[A]\le K_W\log2$. Every empirical
admissible $\gamma$ satisfies $\norm{\gamma}_2\le n^{-1}\sum A_i$, whereas
every population admissible $\gamma$ satisfies $\norm{\gamma}_2\le\E[A]$.
\end{proof}

\begin{proof}[Proof of Theorem~\ref{thm:main_subexp}]
Apply Proposition~\ref{prop:subexp_K}, Lemma~\ref{lem:subexp_sf}, and Lemma~\ref{lem:subexp_lipschitz_envelope} with failure probability $\delta/3$ each. Their intersection has probability at least $1-\delta$, and on it
$G_n\le K_W\{\log2+\log(3/\delta)/n\}$. Lemma~\ref{lem:transport_B} then gives
\[
\begin{aligned}
    &\dH\bigl(\Bhat,\Bstar\bigr)\\
    &\quad\le K_W L_{\Phi,\mathrm{sub}}(n,\delta)
    \left[
    \begin{aligned}
        &2\sqrt{\frac{\log(3/\delta)+3\log2+d\log3}{n}}\\
        &+4\sqrt{\frac{\log(3/\delta)+5\log2+d\log12}{n}}\\
        &+\frac{3\log(3/\delta)+13\log2+d(\log3+2\log12)}{n}\\
        &+\frac{\sqrt d\log(6n/\delta)}{n}
    \end{aligned}
    \right],
\end{aligned}
\]
which is~\eqref{eq:subexp_beta_bound}.
\end{proof}

\subsection{Polynomial-envelope moment-set bound}
\label{app:proof_poly_K}

\begin{proposition}[Polynomial-envelope bounds]\label{prop:poly_K}
Under Assumptions~\ref{ass:sampling} and~\ref{ass:polynomial_envelope}, let $d=p(p+1)/2+p$. For every $\delta\in(0,1)$, with probability at least $1-\delta$,
\begin{equation}\label{eq:poly_K}
\begin{aligned}
  &\dH\bigl(\mathrm{conv}(\hat{\mathcal{K}}_{\mathrm{adm}}),\mathcal{K}^*_{\mathrm{adm}}\bigr)\\
  &\quad\le
  \frac{(\sqrt{2\pi d}+2)K_2}{\sqrt n}
  +5\sqrt2K_2\sqrt{\frac{\log(1/\delta)}{n}}\\
  &\qquad+3\mu_q(K_q+K_2)\delta^{-1/q}n^{-1+1/q}.
\end{aligned}
\end{equation}
\end{proposition}

\begin{lemma}[Polynomial-envelope support-process concentration]\label{lem:poly_support_chaining}
Let
\[
    A_i
    =
    \sup_{z_i\in\mathbf X_i}\norm{z_i}_2^2
    +|Y_i|\sup_{z_i\in\mathbf X_i}\norm{z_i}_2,
\]
and let $A$ be the corresponding population quantity. If $\norm{A}_2\le K_2$ and
$\norm{A}_q\le K_q$ for some $q>2$, then, for every $\delta\in(0,1)$, with probability at
least $1-\delta$,
\[
\begin{aligned}
    &\sup_{U\in\mathcal S}
    \left|\frac1n\sum_{i=1}^nW_i(U)-\E[W(U)]\right|\\
    &\quad\le
    \frac{(\sqrt{2\pi d}+2)K_2}{\sqrt n}
    +5\sqrt2K_2\sqrt{\frac{\log(1/\delta)}{n}}\\
    &\qquad+3\mu_q(K_q+K_2)\delta^{-1/q}n^{-1+1/q}.
\end{aligned}
\]
\end{lemma}

\begin{proof}
By Lemma~\ref{lem:support_lipschitz}, conditionally on the sample,
$|W_i(U)|\le A_i$ and
$|W_i(U)-W_i(V)|\le A_i\norm{U-V}_2$.
The conditional Gaussian comparison in the proof of
Lemma~\ref{lem:bounded_support_chaining} gives
\[
 \E_\varepsilon\sup_{U\in\mathcal S}
 \left|\frac1n\sum_{i=1}^n\varepsilon_iW_i(U)\right|
 \le\frac{\sqrt{\pi d/2}+1}{n}
       \left(\sum_{i=1}^nA_i^2\right)^{1/2}.
\]
Symmetrization and Jensen's inequality consequently give
\begin{align*}
    \E\sup_{U\in\mathcal S}
        \left|\frac1n\sum_{i=1}^nW_i(U)-\E[W(U)]\right|
    &\le\frac{\sqrt{2\pi d}+2}{n}
          \E\left(\sum_{i=1}^nA_i^2\right)^{1/2}\\
    &\le\frac{(\sqrt{2\pi d}+2)K_2}{\sqrt n}.
\end{align*}

To control deviations from this mean, apply the Fuk--Nagaev inequality of
\citet[Theorem~2.1]{Marchina2021} to the countable process indexed by
$(U,s)\in\mathcal S\times\{-1,1\}$, with coordinates
\[
    s\bigl(W_i(U)-\E[W(U)]\bigr).
\]
Its supremum is the absolute supremum in the statement. Its envelope satisfies
\[
    \sup_{U\in\mathcal S}
    |W_i(U)-\E[W(U)]|
    \le A_i+\E[A],
\]
and hence
\[
    \norm{A_i+\E[A]}_2\le2K_2,
    \qquad
    \norm{A_i+\E[A]}_q\le K_q+K_2.
\]
In the notation of that theorem, the wimpy standard deviation is at most
$K_2$, since
$\sup_{U\in\mathcal S}\operatorname{Var}(W(U))
\le\sup_{U\in\mathcal S}\E[W(U)^2]\le K_2^2$.
Moreover, each variable entering the definition of its variance correction is
obtained from twice the envelope by replacing a lower part of its distribution by
zero. Its second moment is therefore at most
$4\E[(A_i+\E[A])^2]$, and averaging these second moments gives
\[
    v_n
    \le4\E\bigl[(A_i+\E[A])^2\bigr]
    \le16K_2^2.
\]
Theorem~2.1 and the bound $\Lambda_q^+(A_i+\E[A])\le
\norm{A_i+\E[A]}_q$ now imply, with probability at least $1-\delta$,
\begin{align*}
    &\sup_{U\in\mathcal S}
    \left|\sum_{i=1}^nW_i(U)-n\E[W(U)]\right|\\
    &\quad\le
    \E\!\left[
        \sup_{U\in\mathcal S}
        \left|\sum_{i=1}^nW_i(U)-n\E[W(U)]\right|
    \right]
    +5\sqrt{2}\,K_2\sqrt{n\log(1/\delta)}
    +3\mu_q(K_q+K_2)n^{1/q}\delta^{-1/q}.
\end{align*}
Insert the expectation bound above and divide by $n$. Continuity extends the
result from the countable dense subset to $\mathcal S$, and the conclusion follows.
\end{proof}

\begin{proof}[Proof of Proposition~\ref{prop:poly_K}]
By identity~\eqref{eq:hausdorff_support} and Lemma~\ref{lem:support_representation}, the Hausdorff distance equals the supremum of the centered empirical support function:
\[
    \dH\bigl(\mathrm{conv}(\hat{\mathcal{K}}_{\mathrm{adm}}),\mathcal{K}^*_{\mathrm{adm}}\bigr) = \sup_{U \in \mathcal{S}} \abs{ \frac1n \sum_{i=1}^n W_i(U) - \E[W(U)] }.
\]
Lemma~\ref{lem:poly_support_chaining} proves~\eqref{eq:poly_K}.
\end{proof}

\subsection{Proof of Theorem~\ref{thm:main_polynomial}}
\label{app:proof_poly_B}

For any $i = 1,\dots,n$, we denote by $A_i$ the envelope of the $i$-th observation as defined in~\eqref{eq:envelope_def}, and by $A$ the corresponding population quantity.

\begin{lemma}[Polynomial Lipschitz Envelope]\label{lem:poly_lipschitz_envelope}
Suppose Assumption~\ref{ass:polynomial_envelope} holds. For every $\delta\in(0,1)$, with probability at least $1-\delta$,
\begin{equation}\label{eq:poly_lipschitz_envelope_bound}
    \frac1n\sum_{i=1}^n A_i
    \le K_2\sqrt{1+\frac{1-\delta}{n\delta}}.
\end{equation}
Consequently, $G_n$ is bounded by the right-hand side on the same event.
\end{lemma}

\begin{proof}
If $\operatorname{Var}(A)=0$, the result is immediate. Otherwise,
Cantelli's one-sided variance inequality gives
\[
 \PP\left(\frac1n\sum_{i=1}^nA_i>
 \E[A]+\sqrt{\frac{1-\delta}{n\delta}}\,
            \sqrt{\operatorname{Var}(A)}\right)\le\delta.
\]
Indeed, for centered $T$ with variance $v$ and $t>0$, Markov's inequality
applied to $(T+v/t)^2$ gives
$\PP(T\ge t)\le v/(v+t^2)$.
Cauchy--Schwarz and $(\E[A])^2+\operatorname{Var}(A)=K_2^2$ give
\[
 \E[A]+\sqrt{\frac{1-\delta}{n\delta}}\,
            \sqrt{\operatorname{Var}(A)}
 \le K_2\sqrt{1+\frac{1-\delta}{n\delta}}.
\]
Every empirical admissible $\gamma$ satisfies
$\norm{\gamma}_2\le n^{-1}\sum A_i$, and every population admissible
$\gamma$ satisfies $\norm{\gamma}_2\le\E[A]\le K_2$.
This proves the last assertion.
\end{proof}

\begin{lemma}[Polynomial Shapley--Folkman Remainder]\label{lem:poly_sf}
Suppose Assumption~\ref{ass:polynomial_envelope} holds. For every $\delta\in(0,1)$, with probability at least $1-\delta$,
\begin{equation}\label{eq:poly_sf_envelope_bound}
    \max_{1\le i\le n} A_i \le K_q \left(\frac{n}{\delta}\right)^{1/q},
\end{equation}
and
\begin{equation}\label{eq:poly_sf_bound}
    \dH\bigl(\hat{\mathcal{K}}_{\mathrm{adm}},\mathrm{conv}(\hat{\mathcal{K}}_{\mathrm{adm}})\bigr)
    \le
    \sqrt{d}\,K_q\,\delta^{-1/q}\,n^{-1+1/q}.
\end{equation}
\end{lemma}

\begin{proof}
By a union bound and Markov's inequality,
\[
    \PP\left( \max_{1\le i\le n} A_i > K_q\left(\frac{n}{\delta}\right)^{1/q} \right)
    \le
    n\frac{\E[A^q]}{K_q^q n/\delta}
    \le\delta.
\]
Lemma~\ref{lem:sf_remainder} then gives~\eqref{eq:poly_sf_bound}.
\end{proof}

\begin{proof}[Proof of Theorem~\ref{thm:main_polynomial}]
Apply Proposition~\ref{prop:poly_K}, Lemma~\ref{lem:poly_sf}, and Lemma~\ref{lem:poly_lipschitz_envelope} with failure probability $\delta/3$ each. Their intersection has probability at least $1-\delta$, and on it
$G_n\le K_2\sqrt{1+(3-\delta)/(n\delta)}$. Lemma~\ref{lem:transport_B} gives
\begin{align*}
    \dH\bigl(\Bhat,\Bstar\bigr)
    &\le
    L_{\Phi,q}(n,\delta)
    \left[
    \begin{aligned}
        &\frac{(\sqrt{2\pi d}+2)K_2}{\sqrt n}
         +5\sqrt2K_2\sqrt{\frac{\log(3/\delta)}{n}}\\
        &+\{3\mu_q(K_q+K_2)+\sqrt d\,K_q\}(3/\delta)^{1/q}n^{-1+1/q}
    \end{aligned}
    \right].
\end{align*}
This is~\eqref{eq:polynomial_beta_bound}.
\end{proof}

\section{Proofs for Set-Valued Z-Estimation}
\label{app:section5}
This appendix contains the proofs for Section~\ref{sec:set_valued_z}.
Write $F=\sup_{\beta\in\Theta}\sup_{z\in\mathbf Z(O)}\norm{\varphi(z,\beta)}_2$
for the integrable envelope in Assumption~\ref{ass:z_completion}.

\subsection{Measurability and Compactness of the Score Correspondence}
\label{app:proofplan_z_measurability}

Fix $\beta \in \Theta$. We verify the regularity of
$\Phi_\beta(o) = \{\varphi(z,\beta) : z \in \mathbf{Z}(o)\}$ and of
$\psi_{\beta,u}(o) = \sup_{z \in \mathbf{Z}(o)} u^T \varphi(z,\beta)$.

Let $o$ lie outside the $P_{\mathrm{obs}}$-null exceptional set in
Assumption~\ref{ass:z_completion}. Under the same assumption,
$\mathbf{X}(o)$ is non-empty and compact. Hence
$\mathbf{Z}(o) = \mathbf{X}(o) \times \{y\} \times \{m\}$ is also non-empty
and compact. The continuity of $z \mapsto \varphi(z,\beta)$ on
$\mathbf{Z}(o)$ then makes $\Phi_\beta(o)$ non-empty and compact.

For measurability, Assumption~\ref{ass:z_completion} implies that $\mathbf{Z}$
is a measurable compact-valued correspondence and that
$\mathrm{Gr}(\mathbf{Z})$ is measurable. Under
Assumption~\ref{ass:z_completion}, $(o,z) \mapsto \varphi(z,\beta)$ is a
Carath\'eodory map and is therefore jointly measurable on this graph. Moreover,
\[
    \mathrm{Gr}(\Phi_\beta) = \{(o,v) \in \mathcal{O} \times \mathbb{R}^r : \exists z \in \mathbf{Z}(o) \text{ such that } \varphi(z,\beta) = v\}.
\]
Since the spaces $\mathcal{O}$ and $\mathcal{Z}$ are standard Borel and $\mathbf{Z}(o)$ is compact-valued, the compact-section projection theorem guarantees that for any closed set $C \subseteq \mathbb{R}^r$, the inverse image:
\[
    \{o \in \mathcal{O} : \Phi_\beta(o) \cap C \neq \varnothing\} = \mathrm{Proj}_{\mathcal{O}}\left( \{(o,z) \in \mathrm{Gr}(\mathbf{Z}) : \varphi(z,\beta) \in C\} \right)
\]
is a measurable set in $\mathcal{O}$. This proves that $\Phi_\beta$ is a measurable correspondence.

For fixed $(\beta,u)$, the map $(o,z) \mapsto u^T\varphi(z,\beta)$ is
measurable in $o$ and continuous in $z$. The measurable maximum theorem shows
that
\[
    \psi_{\beta,u}(o) = \sup_{z \in \mathbf{Z}(o)} u^T \varphi(z,\beta)
\]
is measurable and provides a selector $z_{\beta,u} : \mathcal{O} \to \mathcal{Z}$
satisfying $z_{\beta,u}(o) \in \mathbf{Z}(o)$ and
$\psi_{\beta,u}(o) = u^T \varphi(z_{\beta,u}(o), \beta)$ almost surely.

The envelope in Assumption~\ref{ass:z_completion} gives
\[
    |\psi_{\beta,u}| \le \sup_{z\in\mathbf{Z}}\norm{\varphi(z,\beta)}_2 \le F.
\]
so $\psi_{\beta,u}$ is integrable.

\subsection{Proof of Theorem~\ref{thm:z_sharp_equivalence}}
\label{app:proofplan_z_sharp_equivalence}

Fix $\beta \in \Theta$ and define the attainable moment set
\[
    \mathcal{M}(\beta) := \left\{ \mathbb{E}_P\bigl[\varphi((X,Y,M),\beta)\bigr] : P \in \mathcal{P}_{\mathrm{adm}} \right\} \subseteq \mathbb{R}^r,
\]
and the Aumann integral
\[
    \mathcal{G}(\beta) := \E_{\mathrm{Aumann}}\bigl[\Phi_\beta\bigr] = \left\{ \mathbb{E}[V(O)] : V(o) \in \Phi_\beta(o) \text{ } P_{\mathrm{obs}}\text{-a.s., } V \text{ is integrable} \right\} \subseteq \mathbb{R}^r.
\]
To prove $\mathcal{M}(\beta) \subseteq \mathcal{G}(\beta)$, let
$v \in \mathcal{M}(\beta)$. There exists $P \in \mathcal{P}_{\mathrm{adm}}$
such that $v = \mathbb{E}_P[\varphi((X,Y,M),\beta)]$.
Under $P$, the pair $(O, (X,Y,M))$ has a joint law on $\mathcal{O} \times \mathcal{Z}$. Since both spaces are standard Borel, the disintegration theorem \citep{Kallenberg2002} guarantees the existence of a probability kernel $K_P(dz \mid o)$ on $\mathcal{Z}$ given $\mathcal{O}$ such that:
\[
    P(do, dz) = P_{\mathrm{obs}}(do)\, K_P(dz \mid o).
\]
The marginal on $\mathcal{O}$ is $P_{\mathrm{obs}}$: indeed, for any measurable $A \subseteq \mathcal{O}$,
\[
    P(O \in A) = P\bigl(\mathsf{Obs}^{-1}(A)\bigr) = \bigl(P \circ \mathsf{Obs}^{-1}\bigr)(A) = P_{\mathrm{obs}}(A),
\]
where the last equality is the admissibility condition. Moreover, the condition $P((X,Y,M) \in \mathbf{Z}) = 1$ translates into a support constraint on the kernel: since
\[
    1 = P\bigl((X,Y,M) \in \mathbf{Z}\bigr) = \int_{\mathcal{O}} K_P\bigl(\mathbf{Z}(o) \mid o\bigr)\, P_{\mathrm{obs}}(do),
\]
and $K_P(\mathbf{Z}(o) \mid o) \le 1$ for all $o$, we conclude that $K_P(\mathbf{Z}(o) \mid o) = 1$ for $P_{\mathrm{obs}}$-almost every $o \in \mathcal{O}$. In other words, conditionally on the observed data $o$, the complete-data vector $(X,Y,M)$ is almost surely compatible with $o$.
Define the conditional score expectation function on $\mathcal O$:
\[
    V_P(o) = \int_{\mathcal{Z}} \varphi(z, \beta) K_P(dz \mid o).
\]
Since $\mathbf{Z}(o)$ is compact, the support of $K_P(\cdot \mid o)$ is contained in the compact set $\mathbf{Z}(o)$ almost surely. Thus, since $z \mapsto \varphi(z, \beta)$ is continuous, $\Phi_\beta(o) = \{\varphi(z, \beta) : z \in \mathbf{Z}(o)\}$ is compact and the point $V_P(o)$ lies in the closed convex hull of $\Phi_\beta(o)$ almost surely:
\[
    V_P(o) \in \mathrm{conv}\bigl(\Phi_\beta(o)\bigr) \quad P_{\mathrm{obs}}\text{-a.s.}
\]
Since $F$ is integrable, $V_P$ (when composed with $O$) defines an integrable selection of the convexified score correspondence $\mathrm{conv}(\Phi_\beta)$. Since the underlying probability space is non-atomic, Aumann's convexity theorem ensures that the Aumann expectation of $\Phi_\beta$ is convex and coextensive with that of its convex hull:
\[
    \E_{\mathrm{Aumann}}\bigl[\Phi_\beta\bigr] = \E_{\mathrm{Aumann}}\bigl[\mathrm{conv}\bigl(\Phi_\beta\bigr)\bigr].
\]
Therefore,
\[
    v = \mathbb{E}_P[\varphi((X,Y,M),\beta)] = \mathbb{E}[V_P(O)] \in \E_{\mathrm{Aumann}}\bigl[\mathrm{conv}\bigl(\Phi_\beta\bigr)\bigr] = \mathcal{G}(\beta).
\]
Conversely, let $v \in \mathcal{G}(\beta)$. There exists an integrable
selection $V(o) \in \Phi_\beta(o)$ such that
$v = \mathbb{E}[V(O)]$. Consider the inverse-image correspondence
\[
    \Lambda(o) = \{z \in \mathbf{Z}(o) : \varphi(z,\beta) = V(o)\} \subseteq \mathcal{Z}.
\]
By definition of $\Phi_\beta(o)$, the set $\Lambda(o)$ is non-empty for almost every $o$. It is also compact: since $z \mapsto \varphi(z,\beta)$ is continuous, $\Lambda(o)$ is a closed subset of the compact set $\mathbf{Z}(o)$. Furthermore, its graph is measurable because it is the intersection of the measurable graph of $\mathbf{Z}$ with the zero set of the jointly measurable function $(o,z) \mapsto \varphi(z,\beta) - V(o)$. By the Kuratowski--Ryll-Nardzewski measurable selection theorem \citep[Theorem~1.2.13]{Molchanov2005}, there exists a Borel measurable function $T : \mathcal{O} \to \mathcal{Z}$ such that $T(o) \in \Lambda(o)$ for almost every $o$.
Let $P_T$ be the law of $T(O)$ when $O$ has distribution $P_{\mathrm{obs}}$.
Since $T(o)\in\mathbf Z(o)$, we have $\mathsf{Obs}(T(o))=o$ almost surely;
hence $P_T\in\mathcal P_{\mathrm{adm}}$, and:
\[
    \mathbb{E}_{P_T}[\varphi((X,Y,M),\beta)] 
    = \int_{\mathcal{O}} \varphi(T(o), \beta) P_{\mathrm{obs}}(do) 
    = \int_{\mathcal{O}} V(o) P_{\mathrm{obs}}(do) = \mathbb{E}[V(O)] = v.
\]
Consequently, $\mathcal{G}(\beta) \subseteq \mathcal{M}(\beta)$, and
\eqref{eq:z_attainable_moments_equal_aumann} follows:
\[
    \mathcal{G}(\beta) = \left\{ \mathbb{E}_P[\varphi((X,Y,M),\beta)] : P \in \mathcal{P}_{\mathrm{adm}} \right\}.
\]
The equivalence of the zero sets $\Bstar_{\mathrm{Obs}} = \Bstar_{\mathrm{Aumann}}$ follows immediately:
\[
    \beta \in \Bstar_{\mathrm{Obs}} 
    \iff 0 \in \left\{ \mathbb{E}_P[\varphi((X,Y,M),\beta)] : P \in \mathcal{P}_{\mathrm{adm}} \right\}
    \iff 0 \in \mathcal{G}(\beta)
    \iff \beta \in \Bstar_{\mathrm{Aumann}}.
\]

Under Assumption~\ref{ass:z_completion}, $\Phi_\beta$ is compact-valued
and integrably bounded by $F$, so its Aumann expectation is compact. Since the
underlying probability space is non-atomic, Aumann's convexity theorem also
implies that $\mathcal{G}(\beta)$ is convex, even when the values
$\Phi_\beta(o)$ are not.

\subsection{Proof of the Support-Function Identity}
\label{app:proofplan_z_support}
For $u\in\mathbb S^{r-1}$, define
\begin{equation}
\psi_{\beta,u}(o)=\sup_{z\in\mathbf Z(o)}u^T\varphi(z,\beta).
\label{eq:z_support_class_element}
\end{equation}

For $u \in \mathbb{S}^{r-1}$, let $V(O)$ be any integrable selection of
$\Phi_\beta$. By \eqref{eq:z_support_class_element},
\[
    u^T V(O) \le \sup_{v \in \Phi_\beta} u^T v = \psi_{\beta,u} \quad \text{a.s.}
\]
Taking expectations on both sides yields:
\[
    u^T \mathbb{E}[V(O)] \le \mathbb{E}[\psi_{\beta,u}].
\]
Since this holds for every selection $V(O)$, taking the supremum over all selection integrals yields $\psi_{\mathcal{G}(\beta)}(u) \le \mathbb{E}[\psi_{\beta,u}]$.

For the reverse inequality, the measurable maximum theorem provides a Borel
measurable function $z_{\beta,u} : \mathcal{O} \to \mathcal{Z}$ such that
$z_{\beta,u}(o) \in \mathbf{Z}(o)$ and
\[
    u^T \varphi(z_{\beta,u}(o), \beta) = \psi_{\beta,u}(\omega) \quad \text{a.s. where } o = O(\omega).
\]
Define the selection $V_{\beta,u}(o) = \varphi(z_{\beta,u}(o), \beta)$. This is a measurable and integrably bounded selection of $\Phi_\beta(o)$. Its expectation satisfies:
\[
    \mathbb{E}[V_{\beta,u}(O)] = \int_{\mathcal{O}} \varphi(z_{\beta,u}(o), \beta) P_{\mathrm{obs}}(do).
\]
Taking the inner product with $u$ gives:
\[
    u^T \mathbb{E}[V_{\beta,u}(O)] = \int_{\mathcal{O}} u^T \varphi(z_{\beta,u}(o), \beta) P_{\mathrm{obs}}(do) = \mathbb{E}[\psi_{\beta,u}].
\]
Since $\mathbb{E}[V_{\beta,u}(O)] \in \mathcal{G}(\beta)$, we must have:
\[
    \psi_{\mathcal{G}(\beta)}(u) \ge u^T \mathbb{E}[V_{\beta,u}(O)] = \mathbb{E}[\psi_{\beta,u}].
\]
Combining the two inequalities yields the support function identity $\psi_{\mathcal{G}(\beta)}(u) = \mathbb{E}[\psi_{\beta,u}]$.

\subsection{Proof of the Empirical Convexification Identity}
\label{app:proofplan_z_convexification}

Let $A_1, \dots, A_n \subseteq \mathbb{R}^r$ be non-empty. Since
$A_i \subseteq \mathrm{conv}(A_i)$, monotonicity of the Minkowski sum gives
\[
    \sum_{i=1}^n A_i \subseteq \sum_{i=1}^n \mathrm{conv}(A_i).
\]
Since the Minkowski sum of convex sets is convex, the set $\sum_{i=1}^n \mathrm{conv}(A_i)$ is convex. Since the convex hull operator $\mathrm{conv}(\cdot)$ is the smallest convex set containing its argument, we obtain:
\[
    \mathrm{conv}\left( \sum_{i=1}^n A_i \right) \subseteq \sum_{i=1}^n \mathrm{conv}(A_i).
\]

For the reverse inclusion, let
$y \in \sum_{i=1}^n \mathrm{conv}(A_i)$, so that
$y = \sum_{i=1}^n y_i$ with $y_i \in \mathrm{conv}(A_i)$. By
Carath\'eodory's theorem, each $y_i$ admits a finite convex representation
\[
    y_i = \sum_{k=1}^{r+1} \lambda_{i, k} a_{i, k}, \quad \text{where } a_{i, k} \in A_i, \ \lambda_{i, k} \ge 0, \ \sum_{k=1}^{r+1} \lambda_{i, k} = 1.
\]
Substituting these expansions into the sum for $y$ gives:
\[
    y = \sum_{i=1}^n \sum_{k=1}^{r+1} \lambda_{i, k} a_{i, k}.
\]
We can expand this sum as a convex combination of points in $\sum_{i=1}^n A_i$. Let $K = \{1, \dots, r+1\}^n$. For each multi-index $\mathbf{k} = (k_1, \dots, k_n) \in K$, define the joint weight $w_{\mathbf{k}} = \prod_{i=1}^n \lambda_{i, k_i}$ and the sum vector $s_{\mathbf{k}} = \sum_{i=1}^n a_{i, k_i} \in \sum_{i=1}^n A_i$. Note that $w_{\mathbf{k}} \ge 0$ and:
\[
    \sum_{\mathbf{k} \in K} w_{\mathbf{k}} = \sum_{k_1} \dots \sum_{k_n} \prod_{i=1}^n \lambda_{i, k_i} = \prod_{i=1}^n \left( \sum_{k=1}^{r+1} \lambda_{i, k} \right) = 1.
\]
By expanding, we verify that:
\[
    \sum_{\mathbf{k} \in K} w_{\mathbf{k}} s_{\mathbf{k}} 
    = \sum_{k_1} \dots \sum_{k_n} \left(\prod_{j=1}^n \lambda_{j, k_j}\right) \left( \sum_{i=1}^n a_{i, k_i} \right)
    = \sum_{i=1}^n \sum_{k_i=1}^{r+1} \lambda_{i, k_i} a_{i, k_i} = y.
\]
Thus, $y$ is a convex combination of elements $s_{\mathbf{k}} \in \sum_{i=1}^n A_i$, which implies that $y \in \mathrm{conv}(\sum_{i=1}^n A_i)$.
This establishes that $\sum_{i=1}^n \mathrm{conv}(A_i) \subseteq \mathrm{conv}(\sum_{i=1}^n A_i)$.

Scaling by $1/n$ now yields
\[
    \widehat{\mathcal{G}}_n(\beta) = \frac{1}{n} \sum_{i=1}^n \mathrm{conv}(\Phi_\beta(O_i)) = \mathrm{conv}\left( \frac{1}{n} \sum_{i=1}^n \Phi_\beta(O_i) \right) = \mathrm{conv}\left( \widehat{\mathcal{G}}^{\mathrm{d}}_n(\beta) \right).
\]

\subsection{Proof of Corollary~\ref{cor:z_pointwise_clt}}
\label{app:proof_z_pointwise_clt}
\begin{proof}
Apply the random-set central limit theorem to the i.i.d.\ random compact sets
$\mathrm{conv}(\Phi_\beta(O_i))$.
\end{proof}

\subsection{Proof of Theorem~\ref{thm:z_functional_clt}}
\label{app:proof_z_functional_clt}

\begin{proof}
Using the support-function identities and
identity~\eqref{eq:hausdorff_support},
\[
 \sqrt n\,\Delta_n
 =\sup_{\substack{\beta\in\Theta\\u\in\mathbb S^{r-1}}}
 \left|\frac1{\sqrt n}\sum_{i=1}^n
 \{\psi_{\beta,u}(O_i)-\E[\psi_{\beta,u}(O)]\}\right|.
\]
Under the usual measurable or separable version of the Donsker assumption
\citep[Chapter~2]{vanDerVaartWellner1996},
the process on the right converges weakly in
$\ell^\infty(\Theta\times\mathbb S^{r-1})$ to a tight centered Gaussian
process. Its supremum norm is therefore $\mathcal O_P(1)$, which proves
$\Delta_n=\mathcal O_P(n^{-1/2})$.
\end{proof}

\subsection{Proof of Corollary~\ref{cor:z_oracle_zero_root_n}}
\label{app:proofplan_z_oracle_zero_root_n}

Work on the event
$\sup_{\beta\in\Theta}d_{\mathrm{H}}(\widehat{\mathcal{G}}_n(\beta),
\mathcal{G}(\beta))=\Delta_n$. If $\beta \in \Bstar$, then
$d(0,\mathcal{G}(\beta))=0$. The Lipschitz property of the distance to a set
implies
\[
    d(0, \widehat{\mathcal{G}}_n(\beta)) \le d(0, \mathcal{G}(\beta)) + \Delta_n = \Delta_n.
\]
By definition of $\widehat{\mathcal{B}}_n^{+}$, we have $\beta \in \widehat{\mathcal{B}}_n^{+}$, which proves the containment $\Bstar \subseteq \widehat{\mathcal{B}}_n^{+}$.

Conversely, if $\beta \in \widehat{\mathcal{B}}_n^{+}$, then
$d(0,\widehat{\mathcal{G}}_n(\beta))\le\Delta_n$, and the same Lipschitz
property gives
\[
    d(0, \mathcal{G}(\beta)) \le d(0, \widehat{\mathcal{G}}_n(\beta)) + \Delta_n \le 2\Delta_n.
\]
Under Assumption~\ref{ass:z_error_bound}, the uniform zero-set error bound guarantees that $d(0, \mathcal{G}(\beta)) \ge \kappa_0 d(\beta, \Bstar)$. Combining these inequalities gives:
\[
    d(\beta, \Bstar) \le \frac{d(0, \mathcal{G}(\beta))}{\kappa_0} \le \frac{2\Delta_n}{\kappa_0}.
\]

Since $\Bstar \subseteq \widehat{\mathcal{B}}_n^{+}$, the directed Hausdorff
distance from $\Bstar$ to $\widehat{\mathcal{B}}_n^{+}$ is zero, whereas
\[
    \sup_{\beta \in \widehat{\mathcal{B}}_n^{+}} d(\beta, \Bstar) \le \frac{2\Delta_n}{\kappa_0}.
\]
Therefore, the symmetric Hausdorff distance is:
\[
    d_{\mathrm{H}}\left( \widehat{\mathcal{B}}_n^{+}, \Bstar \right) = \max\left( \sup_{\beta \in \widehat{\mathcal{B}}_n^{+}} d(\beta, \Bstar), \ \sup_{\beta \in \Bstar} d(\beta, \widehat{\mathcal{B}}_n^{+}) \right) \le \frac{2\Delta_n}{\kappa_0}.
\]

Theorem~\ref{thm:z_functional_clt} gives
$\Delta_n = \mathcal{O}_P(n^{-1/2})$; since $\kappa_0>0$ is fixed,
\[
    d_{\mathrm{H}}\left( \widehat{\mathcal{B}}_n^{+}, \Bstar \right) = \mathcal{O}_P(n^{-1/2}).
\]

\section{Proofs for the Hadamard Sensitivity Analysis}
\label{app:section3}
This appendix contains the proofs and heuristic derivations for
Section~\ref{sec:hadamard}.

\subsection{Proof of Theorem~\ref{thm:hadamard_width_bound}}
\label{app:proof_hadamard_width_bound}

\begin{proof}
Let $f(x)=\widehat\beta_j(x)$ and
$\ell(x)=f(0)+\sum_{r=1}^m\partial_r f(0)x_r$. The integral form of
Taylor's theorem gives
\[
 f(x)-\ell(x)
 =\int_0^1(1-t)\sum_{r,s=1}^m
   \partial_{rs}^2f(tx)x_rx_s\,dt.
\]
Since $|x_r|\le1$,
\[
 \norm{f-\ell}_\infty
 \le\frac12\sup_{x\in[-1,1]^m}
       \sum_{r,s=1}^m|\partial_{rs}^2f(x)|.
\]
For two real-valued functions, the difference between their widths is at most
twice their uniform distance. Moreover,
$\max_x\ell(x)-\min_x\ell(x)=2\sum_r|\partial_r f(0)|$. Hence
\begin{align*}
 \left|\operatorname{wid}_j(\{\widehat\beta(x):x\in[-1,1]^m\})
       -\sum_r|\tau_{j,r}|\right|
 &\le \sup_x\sum_{r,s}|\partial_{rs}^2f(x)|
   +\left|2\sum_r|\partial_r f(0)|-\sum_r|\tau_{j,r}|\right|\\
 &\le \sup_x\sum_{r,s}|\partial_{rs}^2f(x)|
   +\sum_r|\tau_{j,r}-2\partial_r f(0)|,
\end{align*}
which is~\eqref{eq:hadamard_width_bound}.

For the vector conclusion, set
\[
 g(x)=\widehat\beta(0)+\sum_{r=1}^m
       \partial_r\widehat\beta(0)x_r,
 \qquad
 \mathcal Z_0=\{g(x):x\in[-1,1]^m\}.
\]
Pairing $\widehat\beta(x)$ with $g(x)$ for the same $x$, in both directions,
the same Taylor formula, now in Euclidean norm, gives
\[
 \dH\!\left(\{\widehat\beta(x):x\in[-1,1]^m\},\mathcal Z_0\right)
 \le\frac12\sup_{x\in[-1,1]^m}
       \sum_{r,s=1}^m\norm{\partial_{rs}^2\widehat\beta(x)}_2.
\]
Using the same coefficients $t_r$ in the two zonotopes gives
\[
 \dH(\mathcal Z_0,\widehat{\mathcal Z}_{\mathrm H})
 \le\frac12\sum_{r=1}^m
       \norm{\tau_r-2\partial_r\widehat\beta(0)}_2.
\]
The triangle inequality proves~\eqref{eq:hadamard_zonotope_bound}.
\end{proof}

\begin{proof}[Proof of Corollary~\ref{cor:total_width_error}]
On the event $\dH(\Bhat,\Bstar)\le r_n(\delta)$, the coordinatewise extrema
therefore give
\[
 |\operatorname{wid}_j(\Bhat)-\operatorname{wid}_j(\Bstar)|
 \le2r_n(\delta).
\]
The triangle inequality and Theorem~\ref{thm:hadamard_width_bound} complete
the proof.
\end{proof}

\subsection{Smooth sensitivity path}
Let $x=(x_{ik})\in[-1,1]^m$ index the missing entries and let
$X_i(x)=(z_i(x),Y_i,M_i)$.  A smooth just-identified estimator satisfies
\begin{equation}
    F_n(\beta,x)
    :=
    \frac1n\sum_{i=1}^n\varphi\bigl(X_i(x),\beta\bigr)=0.
    \label{eq:z_sample_estimating_map_new}
\end{equation}
Write $\widehat\beta(x)$ for its solution.

\begin{assumption}[Smooth sensitivity path]
\label{ass:z_ift_new}
For every $x\in[-1,1]^m$, the equation $F_n(\beta,x)=0$ has a unique solution
$\widehat\beta(x)$ in the interior of $\Theta$.  The map $F_n$ is twice
continuously differentiable on an open neighborhood of the solution graph,
its first and second derivatives are uniformly bounded there, and
\[
    J_n(x):=D_\beta F_n(\widehat\beta(x),x)
\]
is invertible with
\[
    \sup_{x\in[-1,1]^m}\norm{J_n(x)^{-1}}_{\mathrm{op}}<\infty.
\]
\end{assumption}

Under Assumption~\ref{ass:z_ift_new}, the implicit function theorem makes
$x\mapsto\widehat\beta(x)$ twice continuously differentiable.  We can
therefore carry out the same sensitivity analysis as in
Section~\ref{sec:hadamard}: compute
the first and second derivatives by implicit differentiation, use the latter
to control the Taylor correction, and compare the accumulated curvature with
the first-order Hadamard effect.  The general derivative formulas are stated
and proved in
Lemmas~\ref{lem:z_ift_micro_derivatives}--\ref{lem:z_second_derivative_glm}.
In what follows, we illustrate this construction for canonical GLMs.

\subsection{Exact Sensitivity Formulas}
\label{app:z_exact_sensitivity}

\begin{lemma}[Exact microscopic sensitivity of general smooth Z-estimators]
\label{lem:z_ift_micro_derivatives}
Under Assumption~\ref{ass:z_ift_new}, for
$(i,k)\in\mathcal I_{\mathrm{miss}}$,
\begin{equation}
    \frac{\partial\widehat\beta}{\partial x_{ik}}
    =
    -\frac{c_k}{n}J_n(x)^{-1}
    D_z\varphi\bigl(X_i(x),\widehat\beta(x)\bigr)e_k,
    \label{eq:z_first_implicit_derivative_micro}
\end{equation}
where $D_z\varphi$ is the derivative with respect to the completed covariate.
\end{lemma}

\begin{proof}
Under Assumption~\ref{ass:z_ift_new}, the map
$(\beta,x)\mapsto F_n(\beta,x)$ is $C^1$ near the solution path and
$J_n(x)=D_\beta F_n(\widehat\beta(x),x)$ is invertible. The implicit function
theorem therefore makes $x\mapsto\widehat\beta(x)$ continuously differentiable
on $[-1,1]^m$.

Differentiating $F_n(\widehat\beta(x),x)=0$ with respect to $x_{ik}$ gives
\begin{equation}
    D_\beta F_n(\widehat\beta(x), x) \frac{\partial \widehat\beta}{\partial x_{ik}} + \frac{\partial F_n}{\partial x_{ik}}(\widehat\beta(x), x) = 0.
\end{equation}
By definition, $D_\beta F_n(\widehat\beta(x), x) = J_n(x)$. The partial derivative of the sample estimating function $F_n$ with respect to the microscopic parameter $x_{ik}$ is:
\begin{equation}
    \frac{\partial F_n}{\partial x_{ik}}(\beta, x)
    =
    \frac1n \sum_{j=1}^n \frac{\partial}{\partial x_{ik}} \varphi\bigl(X_j(x), \beta\bigr).
\end{equation}
Since the complete-data covariate vector $z_j(x)$ depends on the coordinate $x_{ik}$ only for observation $j = i$, we have $\frac{\partial z_j(x)}{\partial x_{ik}} = \mathbf{0}$ for all $j \neq i$. For the $i$-th observation, the derivative is:
\begin{equation}
    \frac{\partial z_i(x)}{\partial x_{ik}} = c_k e_k.
\end{equation}
Applying the chain rule, we obtain:
\begin{equation}
    \frac{\partial F_n}{\partial x_{ik}}(\beta, x)
    =
    \frac1n D_z \varphi\bigl(X_i(x), \beta\bigr) \frac{\partial z_i(x)}{\partial x_{ik}}
    =
    \frac{c_k}{n} D_z \varphi\bigl(X_i(x), \beta\bigr) e_k.
\end{equation}
Substitution into the differentiated identity gives
\begin{equation}
    J_n(x) \frac{\partial \widehat\beta}{\partial x_{ik}} + \frac{c_k}{n} D_z \varphi\bigl(X_i(x), \widehat\beta(x)\bigr) e_k = 0.
\end{equation}
Left-multiplying by $J_n(x)^{-1}$ yields:
\begin{equation}
    \frac{\partial \widehat\beta}{\partial x_{ik}}
    =
    - J_n(x)^{-1} \left( \frac{c_k}{n} D_z \varphi\bigl(X_i(x), \widehat\beta(x)\bigr) e_k \right),
\end{equation}
which is \eqref{eq:z_first_implicit_derivative_micro}.
\end{proof}

Thus each microscopic sensitivity combines the global inverse Jacobian with
one local score derivative.

\begin{lemma}[Exact microscopic sensitivity of canonical GLMs]
\label{lem:z_glm_micro_derivatives}
For
\[
    \varphi(X,\beta)=z\bigl(\mu(z^T\beta)-y\bigr),
\]
define
\begin{align*}
    \widehat\varepsilon_i(x)
    &:=Y_i-\mu\bigl(z_i(x)^T\widehat\beta(x)\bigr),\\
    A_\mu(x)
    &:=\sum_{a=1}^n
      \mu'\bigl(z_a(x)^T\widehat\beta(x)\bigr)z_a(x)z_a(x)^T.
\end{align*}
If $A_\mu(x)\succ0$, then
\begin{equation}
    \frac{\partial\widehat\beta}{\partial x_{ik}}
    =
    c_kA_\mu(x)^{-1}
    \left[
        \widehat\varepsilon_i(x)e_k
        -\mu'\bigl(z_i(x)^T\widehat\beta(x)\bigr)
         \widehat\beta_k(x)z_i(x)
    \right].
    \label{eq:z_glm_micro_derivative}
\end{equation}
For $\mu(t)=t$, this is the OLS formula of
Lemma~\ref{lem:exact_beta_deriv_micro}.
\end{lemma}

\begin{proof}
For the canonical GLM score
$\varphi(X,\beta)=z\bigl(\mu(z^T\beta)-y\bigr)$, the $\beta$-Jacobian is
\begin{equation}
    D_\beta \varphi(X, \beta) = \mu'(z^T\beta) z z^T.
\end{equation}
Summing over the sample, the sample Jacobian of the estimating map is:
\begin{equation}
    J_n(x)
    =
    D_\beta F_n(\widehat\beta(x), x)
    =
    \frac1n \sum_{j=1}^n \mu'\bigl(z_j(x)^T \widehat\beta(x)\bigr) z_j(x) z_j(x)^T
    =
    \frac{1}{n} A_\mu(x).
\end{equation}
Thus, the inverse sample Jacobian is:
\begin{equation}
    J_n(x)^{-1} = n A_\mu(x)^{-1}.
\end{equation}

The derivative with respect to $z$ is
\begin{equation}
    D_z \varphi(X, \beta)
    =
    D_z \left[ z \bigl( \mu(z^T\beta) - y \bigr) \right].
\end{equation}
Using the product rule, the Jacobian is:
\begin{equation}
    D_z \varphi(X, \beta)
    =
    \bigl( \mu(z^T\beta) - y \bigr) I_p + \mu'(z^T\beta) z \beta^T.
\end{equation}
Evaluating this derivative at $X = X_i(x)$ and $\beta = \widehat\beta(x)$, we obtain:
\begin{equation}
    D_z \varphi\bigl(X_i(x), \widehat\beta(x)\bigr) e_k
    =
    \bigl( \mu(z_i(x)^T \widehat\beta(x)) - Y_i \bigr) e_k + \mu'\bigl(z_i(x)^T \widehat\beta(x)\bigr) z_i(x) \widehat\beta_k(x).
\end{equation}
Using the definition of the residual $\widehat\varepsilon_i(x) = Y_i - \mu(z_i(x)^T \widehat\beta(x))$, this simplifies to:
\begin{equation}
    D_z \varphi\bigl(X_i(x), \widehat\beta(x)\bigr) e_k
    =
    -\widehat\varepsilon_i(x) e_k + \mu'\bigl(z_i(x)^T \widehat\beta(x)\bigr) \widehat\beta_k(x) z_i(x).
\end{equation}

Substitution into \eqref{eq:z_first_implicit_derivative_micro} yields
\begin{align}
    \frac{\partial \widehat\beta}{\partial x_{ik}}
    &= - J_n(x)^{-1} \left( \frac{c_k}{n} D_z \varphi\bigl(X_i(x), \widehat\beta(x)\bigr) e_k \right) \notag \\
    &= - n A_\mu(x)^{-1} \left( \frac{c_k}{n} \left[ -\widehat\varepsilon_i(x) e_k + \mu'\bigl(z_i(x)^T \widehat\beta(x)\bigr) \widehat\beta_k(x) z_i(x) \right] \right) \notag \\
    &= A_\mu(x)^{-1} c_k \left[ \widehat\varepsilon_i(x) e_k - \mu'\bigl(z_i(x)^T \widehat\beta(x)\bigr) \widehat\beta_k(x) z_i(x) \right].
\end{align}
This is \eqref{eq:z_glm_micro_derivative}.
\end{proof}

Second derivatives measure the error of the first-order imputation
approximation.

\begin{lemma}[Second-order microscopic sensitivity of general Z-estimators]
\label{lem:z_second_derivative_general}
Under Assumption~\ref{ass:z_ift_new}, for
$(i,k),(j,l)\in\mathcal I_{\mathrm{miss}}$,
\begin{equation}
\begin{split}
    \frac{\partial^2\widehat\beta}{\partial x_{ik}\partial x_{jl}}
    =
    -J_n(x)^{-1}\Bigg[
        \frac{\partial^2F_n}{\partial x_{ik}\partial x_{jl}}
        +D^2_{\beta x_{ik}}F_n\frac{\partial\widehat\beta}{\partial x_{jl}}
        +D^2_{\beta x_{jl}}F_n\frac{\partial\widehat\beta}{\partial x_{ik}}\\
        +D^2_{\beta\beta}F_n
        \left[
            \frac{\partial\widehat\beta}{\partial x_{ik}},
            \frac{\partial\widehat\beta}{\partial x_{jl}}
        \right]
    \Bigg],
\end{split}
\label{eq:z_second_derivative_general}
\end{equation}
where all derivatives of $F_n$ are evaluated at $(\widehat\beta(x),x)$.
\end{lemma}

\begin{proof}
Differentiate
\[
    D_\beta F_n\frac{\partial\widehat\beta}{\partial x_{ik}}
    +\frac{\partial F_n}{\partial x_{ik}}=0
\]
with respect to $x_{jl}$.  The chain and product rules give
\begin{align*}
    0={}&J_n(x)\frac{\partial^2\widehat\beta}
        {\partial x_{ik}\partial x_{jl}}
        +\frac{\partial^2F_n}{\partial x_{ik}\partial x_{jl}}
        +D^2_{\beta x_{ik}}F_n
          \frac{\partial\widehat\beta}{\partial x_{jl}}\\
       &+D^2_{\beta x_{jl}}F_n
          \frac{\partial\widehat\beta}{\partial x_{ik}}
        +D^2_{\beta\beta}F_n
          \left[
              \frac{\partial\widehat\beta}{\partial x_{ik}},
              \frac{\partial\widehat\beta}{\partial x_{jl}}
          \right].
\end{align*}
Multiplication by $-J_n(x)^{-1}$ proves
\eqref{eq:z_second_derivative_general}.
\end{proof}

\begin{lemma}[Second-order microscopic sensitivity of canonical GLMs]
\label{lem:z_second_derivative_glm}
For a canonical GLM, let
\[
    \Delta_{ik}(x)
    :=c_k\left[
        \widehat\varepsilon_i(x)e_k
        -\mu'\bigl(z_i(x)^T\widehat\beta(x)\bigr)
         \widehat\beta_k(x)z_i(x)
    \right].
\]
Then
\begin{equation}
    \frac{\partial^2\widehat\beta}{\partial x_{ik}\partial x_{jl}}
    =
    A_\mu(x)^{-1}
    \left(
        \frac{\partial\Delta_{ik}(x)}{\partial x_{jl}}
        -\frac{\partial A_\mu(x)}{\partial x_{jl}}
         \frac{\partial\widehat\beta}{\partial x_{ik}}
    \right),
    \label{eq:z_second_derivative_glm}
\end{equation}
with total derivatives on the right-hand side.
\end{lemma}

\begin{proof}
For canonical GLMs, the Z-estimator satisfies $A_\mu(x) \frac{\partial \widehat\beta}{\partial x_{ik}} = \Delta_{ik}$. Differentiating both sides with respect to $x_{jl}$ gives:
\begin{equation}
    A_\mu(x) \frac{\partial^2 \widehat\beta}{\partial x_{ik} \partial x_{jl}} 
    + \frac{\partial A_\mu(x)}{\partial x_{jl}} \frac{\partial \widehat\beta}{\partial x_{ik}} 
    = \frac{\partial \Delta_{ik}}{\partial x_{jl}},
\end{equation}
which immediately yields \eqref{eq:z_second_derivative_glm} upon left-multiplying by $A_\mu(x)^{-1}$.
\end{proof}

The linear approximation is reliable when the accumulated
curvature is small relative to the first derivative.

\subsection{Analytical GLM linearity condition}
\label{app:proof_z_analytical_linearity_glm}

For a canonical GLM, define
\begin{align*}
    \widehat\varepsilon_i(x)
    &:=Y_i-\mu\bigl(z_i(x)^T\widehat\beta(x)\bigr),\\
    A_\mu(x)
    &:=\sum_{a=1}^n
      \mu'\bigl(z_a(x)^T\widehat\beta(x)\bigr)z_a(x)z_a(x)^T,\\
    \Delta_{ik}(x)
    &:=c_k\left[
        \widehat\varepsilon_i(x)e_k
        -\mu'\bigl(z_i(x)^T\widehat\beta(x)\bigr)
         \widehat\beta_k(x)z_i(x)
    \right].
\end{align*}

\begin{theorem}[Analytical GLM linearity condition]
\label{thm:z_analytical_linearity_glm}
Suppose that $A_\mu(x)\succ0$.  For fixed $j$ and $k$, the first-order
sensitivity dominates the accumulated second-order sensitivity if, uniformly
in $x\in[-1,1]^m$,
\begin{equation}
\begin{aligned}
&\sum_{i\in\mathcal M_k}
 \sum_{(a,l)\in\mathcal I_{\mathrm{miss}}}
 \left|
 e_j^TA_\mu(x)^{-1}
 \left(
     \frac{\partial\Delta_{ik}(x)}{\partial x_{al}}
     -\frac{\partial A_\mu(x)}{\partial x_{al}}
      A_\mu(x)^{-1}\Delta_{ik}(x)
 \right)
 \right|\\
&\hspace{35mm}<
\sum_{i\in\mathcal M_k}
\left|e_j^TA_\mu(x)^{-1}\Delta_{ik}(x)\right|.
\end{aligned}
\label{cond:z_analytical_linearity_glm}
\end{equation}
Under this condition, the first-order Hadamard effect associated with
covariate $k$ dominates its Taylor correction.
\end{theorem}

\begin{proof}
By Lemmas~\ref{lem:z_glm_micro_derivatives} and
\ref{lem:z_second_derivative_glm},
\[
    \frac{\partial\widehat\beta_j}{\partial x_{ik}}
    =
    e_j^TA_\mu(x)^{-1}\Delta_{ik}(x)
\]
and
\[
    \frac{\partial^2\widehat\beta_j}{\partial x_{ik}\partial x_{al}}
    =
    e_j^TA_\mu(x)^{-1}
    \left(
        \frac{\partial\Delta_{ik}(x)}{\partial x_{al}}
        -\frac{\partial A_\mu(x)}{\partial x_{al}}
         A_\mu(x)^{-1}\Delta_{ik}(x)
    \right).
\]
Hence the two sides of
\eqref{cond:z_analytical_linearity_glm} are respectively the accumulated
absolute second- and first-order sensitivities associated with covariate $k$.
Taylor's theorem, uniformly on $[-1,1]^m$, then gives the stated domination.
\end{proof}

\subsection{Regularity for the probabilistic approximations}
\begin{assumption}[MNAR Regularity]
\label{ass:mnar_regularity}
Let $n\to\infty$ while the expected missingness fractions $\PP(M_k=1)$ remain
bounded away from zero and one. Fix an admissible trajectory of design matrices
$(Z_n)_{n\ge1}$, with $Z_n\in\Xcal_{\mathrm{adm},n}$, where
$\Xcal_{\mathrm{adm},n}$ is the admissible design space at sample size $n$.
Suppose that, for some non-singular matrix $\Sigma$,
\[
    \frac1n Z_n^T Z_n \xrightarrow{p} \Sigma,
    \qquad \widehat\beta(Z_n)\xrightarrow{p}\beta^*.
\]
Write $\Omega=\Sigma^{-1}$, let $\omega_j^T$ be its $j$-th row and set
$\varepsilon=Y-X^T\beta^*$. The conditional laws of large numbers used below
for the missingness-weighted first- and second-order terms are assumed to hold
along this trajectory.
\end{assumption}

\subsection{Derivation of Remark~\ref{rem:z_probabilistic_domination_glm}}
\label{app:proofplan_z_probabilistic_glm}

\begin{proof}
Following the same logic as the OLS probabilistic proof in Appendix~\ref{app:proof_probabilistic_domination}, the condition for the linear approximation to dominate is that the sum of the absolute curvatures over the missing coordinates is dominated by the first-order gradient.
Summing the exact microscopic monotonicity requirement over $i \in \mathcal{M}_k$:
\begin{equation} \label{eq:glm_sum_monotonicity}
    \sum_{i \in \mathcal{M}_k} \left| \frac{\partial \widehat\beta_j}{\partial x_{ik}} \right| 
    \;>\; 
    \sum_{i \in \mathcal{M}_k} \sum_l \sum_{a \in \mathcal{M}_l} \left| \frac{\partial^2 \widehat\beta_j}{\partial x_{ik} \partial x_{al}} \right|.
\end{equation}
For the left-hand side, since $\frac{\partial \widehat\beta}{\partial x_{ik}} = A_\mu(x)^{-1} \Delta_{ik}$, and $n A_\mu(x)^{-1} \xrightarrow{p} \Omega_\mu$, we have:
\begin{equation}
    \sum_{i \in \mathcal{M}_k} \left| \frac{\partial \widehat\beta_j}{\partial x_{ik}} \right|
    =
    \left( \frac{|\mathcal{M}_k|}{n} \right) \frac{1}{|\mathcal{M}_k|} \sum_{i \in \mathcal{M}_k} \left| e_j^T (n A_\mu(x)^{-1}) \frac{\Delta_{ik}}{c_k} \right| c_k.
\end{equation}
By the Law of Large Numbers, this converges in probability to:
\begin{equation}
    \PP(M_k=1) \, c_k \, \E_P\left[ \left| \Omega_{\mu, jk} \varepsilon^*(z) - \mu'\bigl(z^T \beta^*\bigr) \beta^*_k (\omega_{\mu, j}^T z) \right| \;\middle|\; M_k = 1 \right].
\end{equation}

For the right-hand side, work in the direct-curvature regime stated in
Remark~\ref{rem:z_probabilistic_domination_glm}, where the aggregated global
coupling terms are negligible relative to the first-order sensitivity. The
retained term corresponds to coordinates on the same observation $a=i$,
including the diagonal:
\begin{equation}
    \sum_l \sum_{i \in \mathcal{M}_k \cap \mathcal{M}_l} \left| \frac{\partial^2 \widehat\beta_j}{\partial x_{ik} \partial x_{il}} \right|.
\end{equation}
From Lemma~\ref{lem:z_second_derivative_glm}, the retained direct second-order derivative term is:
\begin{equation}
    \left( \frac{\partial^2 \widehat\beta}{\partial x_{ik} \partial x_{il}} \right)_{\text{direct}}
    =
    - A_\mu(x)^{-1} \Gamma_{ikl}(x),
\end{equation}
where
\[
\begin{aligned}
    \Gamma_{ikl}(x)
    =c_kc_l\Big[&
        \mu'\bigl(z_i(x)^T\widehat\beta(x)\bigr)
        \bigl(\widehat\beta_l(x)e_k+\widehat\beta_k(x)e_l\bigr)\\
        &+\mu''\bigl(z_i(x)^T\widehat\beta(x)\bigr)
        \widehat\beta_k(x)\widehat\beta_l(x)z_i(x)
    \Big].
\end{aligned}
\]
Summing over $i \in \mathcal{M}_k \cap \mathcal{M}_l$ and scaling by $n/n$ yields the limit in probability:
\begin{equation}
    \sum_l \PP(M_k=1, M_l=1) \, c_k c_l \, \mathcal{C}_{j,kl}^*.
\end{equation}
Combining these limits into \eqref{eq:glm_sum_monotonicity} and dividing by $\PP(M_k=1) c_k$ directly yields the probabilistic GLM linearity condition \eqref{cond:z_probabilistic_glm}.
\end{proof}

\subsection{Simplified GLM probabilistic condition}
\label{app:derivation_z_glm_simplifications}

\begin{remark}[Simplified probabilistic condition]
\label{rem:z_glm_simplifications}
Suppose heuristically that the signal dominates the residual, that
$c_l\asymp c$ and $|\beta_l^*|\asymp|\beta_k^*|$, that $\Omega_\mu$ is
diagonal, and that the covariates are standardized. Then condition
\eqref{cond:z_probabilistic_glm} reduces, up to multiplicative constants, to:
\begin{equation}
\label{cond:z_glm_simplifications}
\begin{cases}
    \text{For } j=k: \\[1ex]
    \displaystyle\quad c\sum_l\PP(M_l=1\mid M_k=1) \mathbb E\left[ |\mu'(z^T\beta^*)| + |\beta_k^*||\mu''(z^T\beta^*)| |z_k| \,\middle|\, M_k=1,M_l=1 \right] \\[1ex]
    \displaystyle\quad\quad \ll \mathbb E\left[ |\mu'(z^T\beta^*)| |z_k| \,\middle|\, M_k=1 \right], \\[3ex]
    \text{For } j\ne k: \\[1ex]
    \displaystyle\quad c\,\PP(M_j=1\mid M_k=1) \mathbb E\left[ |\mu'(z^T\beta^*)| + |\beta_k^*||\mu''(z^T\beta^*)| |z_j| \,\middle|\, M_k=1,M_j=1 \right] \\[1ex]
    \displaystyle\quad\quad + c|\beta_k^*| \sum_{l\ne j}\PP(M_l=1\mid M_k=1) \mathbb E\left[ |\mu''(z^T\beta^*)| |z_j| \,\middle|\, M_k=1,M_l=1 \right] \\[1ex]
    \displaystyle\quad\quad \ll \mathbb E\left[ |\mu'(z^T\beta^*)| |z_j| \,\middle|\, M_k=1 \right].
\end{cases}
\end{equation}
The bounded curvature of the logistic mean function makes this condition
potentially more favorable than for Poisson regression, whose curvature may
grow rapidly.
\end{remark}

In the high signal-to-noise regime,
\[
    \E_P\!\left[
        |\Omega_{\mu,jk}\varepsilon^*(z)|
        \,\middle|\,M_k=1
    \right]
    \ll
    |\beta_k^*|
    \mathbb E\!\left[
        |\mu'(z^T\beta^*)|\,|\omega_{\mu,j}^Tz|
        \,\middle|\,M_k=1
    \right].
\]
The right-hand side of \eqref{cond:z_probabilistic_glm} is therefore of order
\[
    |\beta_k^*|
    \mathbb E\!\left[
        |\mu'(z^T\beta^*)|\,|\omega_{\mu,j}^Tz|
        \,\middle|\,M_k=1
    \right].
\]
If $c_l\asymp c$ and $|\beta_l^*|\asymp|\beta_k^*|$, the triangle
inequality bounds the left-hand side, up to a constant, by
\[
\begin{aligned}
    c|\beta_k^*|\sum_l\PP(M_l=1\mid M_k=1)
    \Bigg\{&
        \bigl(|\Omega_{\mu,jk}|+|\Omega_{\mu,jl}|\bigr)
        \mathbb E\!\left[
            |\mu'(z^T\beta^*)|
            \,\middle|\,M_k=1,M_l=1
        \right]\\
        &+
        |\beta_k^*|
        \mathbb E\!\left[
            |\mu''(z^T\beta^*)|\,|\omega_{\mu,j}^Tz|
            \,\middle|\,M_k=1,M_l=1
        \right]
    \Bigg\}.
\end{aligned}
\]
When $\Omega_\mu$ is diagonal, all $l$ contribute to the first term for
$j=k$, whereas only $l=j$
contributes for $j\ne k$.  Cancelling the common factors
$|\beta_k^*|$ and the corresponding diagonal entry of $\Omega_\mu$ gives the simplified conditions:
\begin{equation}
\begin{cases}
    \text{For } j=k: \\[1ex]
    \displaystyle\quad c\sum_l\PP(M_l=1\mid M_k=1) \mathbb E\left[ |\mu'(z^T\beta^*)| + |\beta_k^*||\mu''(z^T\beta^*)| |z_k| \,\middle|\, M_k=1,M_l=1 \right] \\[1ex]
    \displaystyle\quad\quad \ll \mathbb E\left[ |\mu'(z^T\beta^*)| |z_k| \,\middle|\, M_k=1 \right], \\[3ex]
    \text{For } j\ne k: \\[1ex]
    \displaystyle\quad c\,\PP(M_j=1\mid M_k=1) \mathbb E\left[ |\mu'(z^T\beta^*)| + |\beta_k^*||\mu''(z^T\beta^*)| |z_j| \,\middle|\, M_k=1,M_j=1 \right] \\[1ex]
    \displaystyle\quad\quad + c|\beta_k^*| \sum_{l\ne j}\PP(M_l=1\mid M_k=1) \mathbb E\left[ |\mu''(z^T\beta^*)| |z_j| \,\middle|\, M_k=1,M_l=1 \right] \\[1ex]
    \displaystyle\quad\quad \ll \mathbb E\left[ |\mu'(z^T\beta^*)| |z_j| \,\middle|\, M_k=1 \right].
\end{cases}
\end{equation}

\subsection{Exact OLS derivatives}
\label{app:exact_bounds_micro}

The dependence of $\hat\beta(x)$ on a missing scalar entry $x_{ik}$ is obtained
by differentiating the Gram matrix $A(x)=Z(x)^TZ(x)$ and the normal equations.
A second differentiation then permits comparison of the coordinate gradient
with the accumulated curvatures.

\begin{lemma}[Exact Microscopic Gram Derivatives]
\label{lem:exact_bounds_micro}
Let $Q_Z = \sup_{Z \in \Xcal_{\mathrm{adm}}} \max_i \norm{z_i(Z)}_2$ be the deterministic spatial bound on the covariate vectors. For any missingness coordinate $(i,k) \in \mathcal{I}_{\mathrm{miss}}$, the first derivative of the information matrix $A(x)$ satisfies
\begin{equation}
    \frac{\partial A}{\partial x_{ik}} = c_k \bigl( e_k z_i(x)^T + z_i(x) e_k^T \bigr),
    \qquad
    \norm{\frac{\partial A}{\partial x_{ik}}}_{\mathrm{op}} \le 2 c_k \norm{z_i}_2 \le 2 (\max_m c_m) Q_Z.
\end{equation}
For any $(i,k),(a,l)\in\mathcal I_{\mathrm{miss}}$, including equal pairs,
\begin{equation}
    \frac{\partial^2 A}{\partial x_{ik}\partial x_{al}}
    =\mathbf 1\{i=a\}c_kc_l(e_ke_l^T+e_le_k^T).
    \label{eq:exact_gram_hessian_micro}
\end{equation}
In particular,
$\partial^2 A/\partial x_{ik}^2=2c_k^2e_ke_k^T$.
\end{lemma}

\begin{proof}
The design matrix $Z(x) \in \R^{n \times p}$ is composed of row vectors $z_a(x)^T$ for $a \in \{1,\dots,n\}$. Thus, the Gram matrix can be decomposed as the sum of outer products:
\begin{equation}
    A(x) = Z(x)^T Z(x) = \sum_{a=1}^n z_a(x) z_a(x)^T.
\end{equation}
By construction, only the $i$-th observation vector $z_i(x)$ depends on the microscopic parameter $x_{ik}$. For any $a \neq i$, we have $\frac{\partial z_a(x)}{\partial x_{ik}} = \mathbf{0}$. For $a = i$, the derivative is given by the coordinate projection scaled by the half-width: $\frac{\partial z_i(x)}{\partial x_{ik}} = c_k e_k. $

Applying the product rule to the sum of outer products yields:
\begin{align}
    \frac{\partial A(x)}{\partial x_{ik}} &= \sum_{a=1}^n \left( \frac{\partial z_a(x)}{\partial x_{ik}} z_a(x)^T + z_a(x) \left( \frac{\partial z_a(x)}{\partial x_{ik}} \right)^T \right) \notag \\
    &= \frac{\partial z_i(x)}{\partial x_{ik}} z_i(x)^T + z_i(x) \left( \frac{\partial z_i(x)}{\partial x_{ik}} \right)^T \notag \\
    &= c_k e_k z_i(x)^T + z_i(x) \left( c_k e_k \right)^T \notag \\
    &= c_k \bigl( e_k z_i(x)^T + z_i(x) e_k^T \bigr).
\end{align}

To prove the operator norm bound, we analyze the general symmetric rank-two matrix $M = u v^T + v u^T$ for arbitrary vectors $u, v \in \R^p$. By definition, the operator norm under the $\ell_2$-norm is:
\begin{equation}
    \norm{u v^T + v u^T}_{\mathrm{op}} = \sup_{\norm{y}_2 = 1} \norm{(u v^T + v u^T)y}_2.
\end{equation}
Applying the triangle inequality and the Cauchy-Schwarz inequality, we obtain:
\begin{align}
    \norm{(u v^T + v u^T)y}_2 &= \norm{u (v^T y) + v (u^T y)}_2 \notag \\
    &\le \norm{u}_2 \left| v^T y \right| + \norm{v}_2 \left| u^T y \right| \notag \\
    &\le \norm{u}_2 \left( \norm{v}_2 \norm{y}_2 \right) + \norm{v}_2 \left( \norm{u}_2 \norm{y}_2 \right) \notag \\
    &= 2 \norm{u}_2 \norm{v}_2 \norm{y}_2.
\end{align}
Taking the supremum over $\norm{y}_2=1$ proves that $\norm{u v^T + v u^T}_{\mathrm{op}} \le 2 \norm{u}_2 \norm{v}_2$. Setting $u = e_k$ (noting that $\norm{e_k}_2 = 1$) and $v = z_i(x)$ yields:
\begin{equation}
    \norm{\frac{\partial A}{\partial x_{ik}}}_{\mathrm{op}} \le c_k \norm{e_k z_i^T + z_i e_k^T}_{\mathrm{op}} \le 2 c_k \norm{e_k}_2 \norm{z_i}_2 = 2 c_k \norm{z_i}_2.
\end{equation}

For the second-order cross-derivative, we differentiate $\frac{\partial A(x)}{\partial x_{ik}}$ with respect to $x_{jl}$:
\begin{align}
    \frac{\partial^2 A(x)}{\partial x_{ik} \partial x_{jl}} &= \frac{\partial}{\partial x_{jl}} \left[ c_k \bigl( e_k z_i(x)^T + z_i(x) e_k^T \bigr) \right] \notag \\
    &= c_k \left( e_k \left( \frac{\partial z_i(x)}{\partial x_{jl}} \right)^T + \frac{\partial z_i(x)}{\partial x_{jl}} e_k^T \right).
\end{align}
If $i \neq j$, $z_i(x)$ does not depend on $x_{jl}$, so $\frac{\partial z_i(x)}{\partial x_{jl}} = \mathbf{0}$. If $i = j$, including when $k=l$, we have $\frac{\partial z_i(x)}{\partial x_{il}} = c_l e_l$, yielding:
\begin{equation}
    \frac{\partial^2 A(x)}{\partial x_{ik} \partial x_{il}} = c_k \left( e_k \left( c_l e_l \right)^T + \left( c_l e_l \right) e_k^T \right) = c_k c_l \bigl( e_k e_l^T + e_l e_k^T \bigr).
\end{equation}
Combining the two cases gives~\eqref{eq:exact_gram_hessian_micro}.
\end{proof}

The preceding matrix derivative gives the corresponding derivative of the
least-squares coefficient.

\begin{lemma}[Exact Microscopic Differential Representation of $\hat\beta$]
\label{lem:exact_beta_deriv_micro}
Let $\hat{\varepsilon}_i(x) = Y_i - z_i(x)^T \hat\beta(x)$ denote the OLS residual for observation $i$. The exact first derivative of the estimator with respect to the single missing entry $x_{ik}$ is given by:
\begin{equation} \label{eq:exact_gradient_micro}
    g_{ik} = \frac{\partial \hat\beta}{\partial x_{ik}} = A^{-1} \Delta_{ik},
\end{equation}
where $\Delta_{ik} = c_k \bigl( e_k \hat{\varepsilon}_i(x) - z_i(x) \hat\beta_k(x) \bigr) \in \R^p$.
\end{lemma}

\begin{proof}
The OLS estimator satisfies the normal equations $A(x)\hat\beta(x) = Z(x)^T Y$. Differentiating both sides of this identity with respect to the microscopic coordinate $x_{ik}$ using the product rule yields:
\begin{equation} \label{eq:normal_diff}
    \frac{\partial A(x)}{\partial x_{ik}} \hat\beta(x) + A(x) \frac{\partial \hat\beta(x)}{\partial x_{ik}} = \frac{\partial \bigl(Z(x)^T Y\bigr)}{\partial x_{ik}}.
\end{equation}
Since $Z(x)^T Y = \sum_{a=1}^n z_a(x) Y_a$, and only observation $i$ depends on $x_{ik}$, we have:
\begin{equation}
    \frac{\partial \bigl(Z(x)^T Y\bigr)}{\partial x_{ik}} = \sum_{a=1}^n \frac{\partial z_a(x)}{\partial x_{ik}} Y_a = \frac{\partial z_i(x)}{\partial x_{ik}} Y_i = c_k e_k Y_i.
\end{equation}
Substituting this back into \eqref{eq:normal_diff} and isolating the derivative of the estimator gives:
\begin{equation}
    A(x) \frac{\partial \hat\beta(x)}{\partial x_{ik}} = c_k e_k Y_i - \frac{\partial A(x)}{\partial x_{ik}} \hat\beta(x).
\end{equation}
Substitution of the expression for $\frac{\partial A(x)}{\partial x_{ik}}$ from
Lemma~\ref{lem:exact_bounds_micro} gives
\begin{align}
    A(x) \frac{\partial \hat\beta(x)}{\partial x_{ik}} &= c_k e_k Y_i - c_k \bigl( e_k z_i(x)^T + z_i(x) e_k^T \bigr) \hat\beta(x) \notag \\
    &= c_k \left[ e_k Y_i - e_k \bigl( z_i(x)^T \hat\beta(x) \bigr) - z_i(x) \bigl( e_k^T \hat\beta(x) \bigr) \right].
\end{align}
Since the term $z_i(x)^T \hat\beta(x)$ is a scalar, we can factor out $e_k$ from the first two terms. Furthermore, by definition of the coordinate projection, $e_k^T \hat\beta(x)$ is exactly the $k$-th scalar coefficient $\hat\beta_k(x)$. Thus:
\begin{equation}
    A(x) \frac{\partial \hat\beta(x)}{\partial x_{ik}} = c_k \left[ e_k \bigl( Y_i - z_i(x)^T \hat\beta(x) \bigr) - z_i(x) \hat\beta_k(x) \right].
\end{equation}
Using the OLS residual notation $\hat{\varepsilon}_i(x) = Y_i - z_i(x)^T \hat\beta(x)$, this simplifies to:
\begin{equation}
    A(x) \frac{\partial \hat\beta(x)}{\partial x_{ik}} = c_k \bigl( e_k \hat{\varepsilon}_i(x) - z_i(x) \hat\beta_k(x) \bigr) = \Delta_{ik}.
\end{equation}
Multiplying both sides by $A(x)^{-1}$ from the left yields the desired result:
\begin{equation}
    \frac{\partial \hat\beta(x)}{\partial x_{ik}} = A(x)^{-1} \Delta_{ik}.
\end{equation}
\end{proof}

\begin{lemma}[Exact Microscopic OLS Hessian]
\label{lem:exact_beta_hessian_micro}
For $r=(i,k)$ and $s=(a,l)$ in $\mathcal I_{\mathrm{miss}}$, write
$A_r=\partial_rA$ and $A_{rs}=\partial_{rs}^2A$. Then
\begin{equation}
 \partial_{rs}^2\hat\beta
 =-A^{-1}A_sA^{-1}\Delta_r
  -A^{-1}A_rA^{-1}\Delta_s
  -A^{-1}A_{rs}\hat\beta.
 \label{eq:exact_beta_hessian_micro}
\end{equation}
In particular,
\begin{equation}
 \frac{\partial^2\hat\beta}{\partial x_{ik}^2}
 =-2A^{-1}A_{ik}A^{-1}\Delta_{ik}
  -2c_k^2A^{-1}e_ke_k^T\hat\beta.
 \label{eq:exact_beta_diagonal_micro}
\end{equation}
If $\Omega=(A/n)^{-1}$, the accumulated direct term satisfies the exact
identity
\begin{equation}
 n\sum_{(a,l)\in\mathcal I_{\mathrm{miss}}}
 |e_j^TA^{-1}A_{(i,k),(a,l)}\hat\beta|
 =c_k\sum_{l:(i,l)\in\mathcal I_{\mathrm{miss}}}c_l
   |\Omega_{jk}\hat\beta_l+\Omega_{jl}\hat\beta_k|.
 \label{eq:exact_direct_curvature_sum}
\end{equation}
The term $l=k$ on the right is
$2c_k^2|\Omega_{jk}\hat\beta_k|$.
\end{lemma}

\begin{proof}
Differentiate $\partial_r\hat\beta=A^{-1}\Delta_r$. The product rule and
$\partial_sA^{-1}=-A^{-1}A_sA^{-1}$ give
\[
 \partial_{rs}^2\hat\beta
 =-A^{-1}A_sA^{-1}\Delta_r+A^{-1}\partial_s\Delta_r.
\]
Differentiating
$\Delta_r=c_k(e_kY_i-e_kz_i^T\hat\beta-z_ie_k^T\hat\beta)$ gives
$\partial_s\Delta_r=-A_rA^{-1}\Delta_s-A_{rs}\hat\beta$, proving
\eqref{eq:exact_beta_hessian_micro}. Equation
\eqref{eq:exact_beta_diagonal_micro} follows from
$A_{rr}=2c_k^2e_ke_k^T$. Finally, substitute
\eqref{eq:exact_gram_hessian_micro} into the left-hand side of
\eqref{eq:exact_direct_curvature_sum} and use $nA^{-1}=\Omega$.
\end{proof}

\subsection{Microscopic OLS probabilistic condition}
\label{app:proof_probabilistic_domination}

\begin{remark}[Heuristic probabilistic condition]
\label{thm:probabilistic_domination}
Under Assumption~\ref{ass:mnar_regularity}, in a regime where the aggregated
global coupling terms are negligible relative to the first-order sensitivity,
the direct-curvature heuristic is:
\begin{equation} \label{cond:probabilistic}
    \sum_l \PP(M_l=1 \mid M_k=1) \, c_l \big| \Omega_{jk} \beta^*_l + \Omega_{jl} \beta^*_k \big| \;\ll\; \E\bigl[ \big| \omega_j^T (e_k \varepsilon - X\beta^*_k) \big| \;\big|\; M_k = 1 \bigr].
\end{equation}
Here $\PP(M_k=1\mid M_k=1)=1$; its contribution is the diagonal
curvature.
\end{remark}

\begin{proof}
The working comparison for a fixed variable $k$ is
\begin{equation} \label{eq:sum_monotonicity}
    \sum_{i \in \Mcal_k} \left| \frac{\partial \hat\beta_j}{\partial x_{ik}} \right| \;>\; \sum_{i \in \Mcal_k} \sum_{(a,l) \in \mathcal{I}_{\mathrm{miss}}} \left| \frac{\partial^2 \hat\beta_j}{\partial x_{ik} \partial x_{al}} \right|.
\end{equation}
To determine the asymptotic behavior of the left-hand side of \eqref{eq:sum_monotonicity}, we express the summation as:
\begin{equation}
    \sum_{i \in \Mcal_k} \left| \frac{\partial \hat\beta_j}{\partial x_{ik}} \right| = \left(\frac{|\Mcal_k|}{n}\right) \frac{1}{|\Mcal_k|} \sum_{i \in \Mcal_k} \big| \omega_j^T \Delta_{ik} \big|.
\end{equation}
By the Law of Large Numbers, $\frac{|\Mcal_k|}{n} \xrightarrow{p} \PP(M_k=1)$, and the empirical average of $|\omega_j^T \Delta_{ik}|$ converges to its conditional expectation:
\begin{equation} \label{eq:limit_grad}
    \sum_{i \in \Mcal_k} \left| \frac{\partial \hat\beta_j}{\partial x_{ik}} \right| \xrightarrow{p} \PP(M_k=1) \, c_k \E\bigl[ \big| \omega_j^T (e_k \varepsilon - X\beta^*_k) \big| \;\big|\; M_k = 1 \bigr],
\end{equation}
where we factored out the deterministic scale $c_k > 0$.
For each pair of missing entries, Lemma~\ref{lem:exact_beta_hessian_micro}
gives the decomposition
\[
    \frac{\partial^2 \hat\beta_j}{\partial x_{ik} \partial x_{al}}
    = e_j^T T_1 + e_j^T T_2 + e_j^T T_3,
\]
where $T_1$ and $T_2$ are the two global coupling terms and
$T_3=-A^{-1}(\partial^2_{ik,al}A)\hat\beta$ is the direct term. In the
direct-curvature regime stated in Remark~\ref{thm:probabilistic_domination},
the aggregated contributions of $T_1+T_2$ are negligible relative to the
aggregated first-order sensitivity. Thus the calculation below keeps $T_3$;
this is a regime assumption, not a cancellation argument.

Since $\partial^2_{ik,al}A=\mathbf{0}$ whenever $a\neq i$, this local term only
survives for coordinates in the same observation, including the diagonal
pair. Therefore,
the leading curvature contribution retained in this regime is
\begin{equation}
    \sum_l \sum_{i \in \Mcal_k \cap \Mcal_l} \left| e_j^T A^{-1} \left(\frac{\partial^2 A}{\partial x_{ik} \partial x_{il}}\right) \hat\beta \right|.
\end{equation}
Substituting $\frac{\partial^2 A}{\partial x_{ik} \partial x_{il}} = c_k c_l (e_k e_l^T + e_l e_k^T)$ and scaling the sum by $n/n$ yields:
\begin{align}
    \sum_l \sum_{i \in \Mcal_k \cap \Mcal_l}
    \left| e_j^T A^{-1} \left(\frac{\partial^2 A}{\partial x_{ik} \partial x_{il}}\right) \hat\beta \right|
    &= \sum_l \left(\frac{|\Mcal_k \cap \Mcal_l|}{n}\right)
    \frac{1}{|\Mcal_k \cap \Mcal_l|} \notag\\
    &\quad\times \sum_{i \in \Mcal_k \cap \Mcal_l}
    \big| e_j^T \Omega c_k c_l (e_k e_l^T + e_l e_k^T) \hat\beta \big| \notag \\
    &\text{ }\xrightarrow{p} \sum_l \PP(M_k=1, M_l=1) \, c_k c_l \big| \Omega_{jk} \beta^*_l + \Omega_{jl} \beta^*_k \big|. \label{eq:limit_hess}
\end{align}
Substituting the limits \eqref{eq:limit_grad} and \eqref{eq:limit_hess} back into the aggregated monotonicity requirement \eqref{eq:sum_monotonicity} gives:
\begin{equation}
    \sum_l \PP(M_k=1, M_l=1) \, c_k c_l \big| \Omega_{jk} \beta^*_l + \Omega_{jl} \beta^*_k \big| \;\ll\; \PP(M_k=1) \, c_k \E\bigl[ \big| \omega_j^T (e_k \varepsilon - X\beta^*_k) \big| \;\big|\; M_k = 1 \bigr].
\end{equation}
Dividing both sides by $\PP(M_k=1) \, c_k$ (since $c_k > 0$) and applying Bayes' theorem yields the simplified condition \eqref{cond:probabilistic}.
\end{proof}

\subsection{Macroscopic OLS probabilistic condition}
\label{app:proof_analytical_condition}\label{app:exact_A_derivatives}

\begin{remark}[Heuristic probabilistic condition]
\label{prop:analytical_condition}
Under the same working approximations, a covariate-level heuristic is:
\begin{equation} \label{cond:asymptotic_scalar}
    \sum_l \PP(M_l=1 \mid M_k=1) \, c_l \big| \Omega_{jk} \beta^*_l + \Omega_{jl} \beta^*_k \big| \ll \big| \omega_j^T \E\bigl[ e_k \varepsilon - X \beta^*_k \;\big|\; M_k = 1 \bigr] \big|.
\end{equation}
\end{remark}

Consider the covariate-level parametrization $x\in[-1,1]^m$, where a single
coordinate $x_k$ moves all missing entries in covariate $k$ simultaneously.
Differentiating this common perturbation gives the Gram-matrix derivatives and
the corresponding derivative of $\hat\beta(x)$; their population limits yield
Remark~\ref{prop:analytical_condition}.

\begin{lemma}[Exact Macroscopic Gram Derivatives]
\label{lem:exact_A_derivatives}
The matrix $\frac{\partial A}{\partial x_k}$ is symmetric, non-zero only on
its $k$-th row and column, and has operator norm
$\mathcal{O}(|\Mcal_k|)$. For all $k,l$, including $k=l$,
\begin{equation}
 \frac{\partial^2A}{\partial x_k\partial x_l}
 =m_{kl}c_kc_l(e_ke_l^T+e_le_k^T),
 \qquad m_{kl}=|\Mcal_k\cap\Mcal_l|.
 \label{eq:exact_macro_gram_hessian}
\end{equation}
In particular, $\partial^2A/\partial x_k^2=2|\Mcal_k|c_k^2e_ke_k^T$.
\end{lemma}

\begin{proof}
By the linearity of the derivative and the product rule, the first derivative of the information matrix $A(x)$ is:
\begin{equation}
    \frac{\partial A(x)}{\partial x_k} = \sum_{i=1}^n \left( \frac{\partial z_i(x)}{\partial x_k} z_i(x)^T + z_i(x) \left(\frac{\partial z_i(x)}{\partial x_k}\right)^T \right).
\end{equation}
Since $\frac{\partial z_i(x)}{\partial x_k} = \mathbf{0}$ for $i \notin \Mcal_k$ and $\frac{\partial z_i(x)}{\partial x_k} = c_k e_k$ for $i \in \Mcal_k$, the sum collapses to the missingness stratum:
\begin{equation}\label{eq:first_deriv_A}
    \frac{\partial A(x)}{\partial x_k} = \sum_{i \in \Mcal_k} c_k \Bigl( e_k z_i(x)^T + z_i(x) e_k^T \Bigr).
\end{equation}
The triangle inequality and sub-multiplicativity of the operator norm, with
$Q_Z = \sup_{Z, i} \norm{z_i(Z)}_2$, give
\begin{equation}
    \norm{\frac{\partial A}{\partial x_k}}_{\mathrm{op}} \le \sum_{i \in \Mcal_k} 2 c_k \norm{z_i(x)}_2 \le 2 m_k c_k Q_Z = \mathcal{O}(m_k).
\end{equation}
To compute the second-order derivative, we differentiate \eqref{eq:first_deriv_A} with respect to $x_l$:
\begin{equation}
    \frac{\partial^2 A(x)}{\partial x_k \partial x_l} = \sum_{i \in \Mcal_k} c_k \left( e_k \left( \frac{\partial z_i(x)}{\partial x_l} \right)^T + \frac{\partial z_i(x)}{\partial x_l} e_k^T \right).
\end{equation}
Since $\frac{\partial z_i(x)}{\partial x_l} = c_l e_l$ for $i \in \Mcal_l$ and $\mathbf{0}$ otherwise, the summation survives only on the intersection $\Mcal_k \cap \Mcal_l$, giving:
\begin{equation}
    \frac{\partial^2 A(x)}{\partial x_k \partial x_l} = \sum_{i \in \Mcal_k \cap \Mcal_l} c_k c_l \Bigl( e_k e_l^T + e_l e_k^T \Bigr) = m_{kl} c_k c_l \Bigl( e_k e_l^T + e_l e_k^T \Bigr),
\end{equation}
where $m_{kl} = |\Mcal_k \cap \Mcal_l|$. Taking the operator norm yields:
\begin{equation}
    \norm{\frac{\partial^2 A(x)}{\partial x_k \partial x_l}}_{\mathrm{op}} \le 2m_{kl} c_k c_l = \mathcal{O}(m_{kl}),
\end{equation}
which matches the asserted bound.
\end{proof}

The derivative of $A(x)$ identifies the curvature terms. We also need the exact
first-order driving force along the common perturbation $x_k$.

\begin{lemma}[Exact Macroscopic Differential Representation of $\hat\beta$]
\label{lem:exact_beta_deriv}
Let $x \in [-1, 1]^m$ be the normalized missingness policy, mapping to the interval $I_k$ via the deterministic scale $c_{ik}=(u_k-\ell_k)/2$. Let $\hat{\varepsilon}_i(x) = Y_i - z_i(x)^T \hat\beta(x)$ denote the OLS residual. The exact first derivative of the estimator with respect to $x_k$ is:
\begin{equation} \label{eq:exact_gradient}
    G_k = \frac{\partial \hat\beta}{\partial x_k} = (Z^TZ)^{-1} \sum_{i \in \Mcal_k} c_{ik} \Bigl( e_k \hat{\varepsilon}_i(x) - z_i(x) \hat\beta_k(x) \Bigr) = (Z^TZ)^{-1} \Delta_k(x),
\end{equation}
where $\Delta_k \in \R^p$ is the exact empirical driving force of the gradient along axis $k$.
\end{lemma}

\begin{proof}
Let $A = Z^T Z$ and $B = Z^T Y$. By the matrix derivative identity $\frac{\partial A^{-1}}{\partial x_k} = -A^{-1} \frac{\partial A}{\partial x_k} A^{-1}$, the product rule applied to $\hat\beta = A^{-1}B$ yields:
\begin{equation} \label{eq:grad_base}
    G_k = A^{-1} \left( \frac{\partial B}{\partial x_k} - \frac{\partial A}{\partial x_k} \hat\beta \right).
\end{equation}
Because the derivative of $z_i(x)$ with respect to $x_k$ is $\mathbf{0}$ for any patient $i \notin \Mcal_k$, the sum collapses to the $m_k = |\Mcal_k|$ missing observations. For $i \in \Mcal_k$, we have $\frac{\partial z_i}{\partial x_k} = c_{ik} e_k$. Thus:
\begin{align*}
    G_k &= A^{-1} \sum_{i \in \Mcal_k} c_{ik} \Bigl( e_k Y_i - (e_k z_i^T + z_i e_k^T) \hat\beta \Bigr) \\
    &= A^{-1} \sum_{i \in \Mcal_k} c_{ik} \Bigl( e_k \underbrace{\bigl(Y_i - z_i^T \hat\beta\bigr)}_{= \, \hat{\varepsilon}_i} - z_i \underbrace{\bigl(e_k^T \hat\beta\bigr)}_{= \, \hat\beta_k} \Bigr) = A^{-1} \Delta_k.
\end{align*}
\end{proof}

The preceding two lemmas now allow us to compare the limiting gradient with the
limiting curvature terms.
\begin{proof}
By the Law of Large Numbers, $\frac{1}{n} Z^T Z \xrightarrow{p} \Sigma$ and $\hat\beta(\mathbf{0}) \xrightarrow{p} \beta^*$. The empirical driving force of the gradient limits to
\[
    \frac{1}{n} \Delta_k
    \xrightarrow{p}
    \gamma_k c_k \E\bigl[e_k \varepsilon - X\beta_k^* \mid M_k=1\bigr],
\]
where $\gamma_k = \PP(M_k=1)$.

The working condition for strict monotonicity of the scalar function
$\hat\beta_j(x)$ along axis $x_k$ compares the scalar gradient with the sum of
scalar curvatures over the hypercube:
\begin{equation}
    \sum_{l=1}^m \left| \frac{\partial^2 \hat\beta_j}{\partial x_k \partial x_l} \right| \ll \left| \frac{\partial \hat\beta_j}{\partial x_k} \right|.
\end{equation}

Projecting the exact differential representations (Lemma~\ref{lem:exact_beta_deriv}) onto the $j$-th coordinate, the gradient converges in probability to:
\begin{equation}
    \frac{\partial \hat\beta_j}{\partial x_k} = e_j^T G_k \xrightarrow{p} \gamma_k c_k \omega_j^T \E\bigl[e_k \varepsilon - X\beta_k^* \mid M_k=1\bigr].
\end{equation}

To compute the second derivative, differentiate the exact identity
$G_k=A^{-1}\Delta_k$. If $A_l=\partial A/\partial x_l$ and
$A_{kl}=\partial^2 A/\partial x_k\partial x_l$, then
\begin{equation}
    \frac{\partial^2 \hat\beta}{\partial x_k \partial x_l}
    = -A^{-1}A_lA^{-1}\Delta_k
      -A^{-1}A_kA^{-1}\Delta_l
      -A^{-1}A_{kl}\hat\beta .
\end{equation}
The first two terms are global coupling terms: they propagate the first-order
gradient along $x_k$ through the global perturbation of the Gram matrix. We
focus on the direct-curvature regime where their aggregated contribution is
negligible relative to the direct term
$-A^{-1}A_{kl}\hat\beta$. This is a regime assumption, not a cancellation
argument.

Under this regime, using Lemma~\ref{lem:exact_A_derivatives}, the scalar
curvature converges to:
\begin{align}
    \frac{\partial^2 \hat\beta_j}{\partial x_k \partial x_l} = e_j^T H_{kl} &\xrightarrow{p} -\gamma_{kl} e_j^T \Sigma^{-1} c_k c_l (e_k e_l^T + e_l e_k^T) \beta^* \notag \\
    &= -\gamma_{kl} c_k c_l \big( \Omega_{jk} \beta^*_l + \Omega_{jl} \beta^*_k \big),
\end{align}
where $\gamma_{kl} = \PP(M_k=1, M_l=1)$.

Taking the ratio of the asymptotic sum of absolute curvatures, including the
diagonal, to the asymptotic absolute gradient yields:
\begin{equation}
    \frac{ \sum_l \gamma_{kl} c_k c_l \big| \Omega_{jk} \beta^*_l + \Omega_{jl} \beta^*_k \big| }{ \gamma_k c_k \big| \omega_j^T \E[e_k \varepsilon - X\beta_k^* \mid M_k=1] \big| }
    =
    \frac{ \sum_l \gamma_{kl} c_l \big| \Omega_{jk} \beta^*_l + \Omega_{jl} \beta^*_k \big| }{ \gamma_k \big| \omega_j^T \E[e_k \varepsilon - X\beta_k^* \mid M_k=1] \big| },
\end{equation}
where we factored out and canceled the positive scale $c_k$ from both the numerator and the denominator. By the definition of conditional probability, $\frac{\gamma_{kl}}{\gamma_k} = \PP(M_l=1 \mid M_k=1)$, with $\gamma_{kk}=\gamma_k$. Requiring this ratio to be small yields condition \eqref{cond:asymptotic_scalar}.

\end{proof}

\subsubsection{Standardized covariates}
\begin{remark}[Standardized Covariates Simplification]
\label{rem:standardized_covariates_macro}
Under the additional simplification $\E[\varepsilon \mid M_k=1]=0$, and after centering and scaling the covariates, take $c_l \asymp c$ and $\abs{\beta_l^*}\asymp\abs{\beta_k^*}$ as order-of-magnitude approximations. The condition \eqref{cond:asymptotic_scalar} then becomes:
\begin{equation} \label{cond:standardized_macro}
    \sum_l \PP(M_l=1 \mid M_k=1) \, c \abs{\beta_k^*} \big| \Omega_{jk} + \Omega_{jl} \big| \ll \big| \beta^*_k \omega_j^T \mathbb{E}[X \mid M_k = 1] \big|.
\end{equation}
\end{remark}

\subsection{Interpretation of the two-way interaction contrast}
\label{app:interaction_proof}

\begin{remark}[The Interpretation of the Two-Way Interaction Effect $\tau_{j, kl}$]
\label{rem:interaction_extraction}
In the conditional approach of \citet{BatteyCox2023}, the two-way interaction contrast between factors $k$ and $l$ on the target coefficient $j$ is defined using the orthogonal design as:
\begin{equation} \label{eq:interaction_contrast_def}
    \tau_{j, kl} = \frac{1}{N/2} \sum_{r=1}^N h_{r, k} h_{r, l} \hat\beta_j(h^{(r)}).
\end{equation}
By substituting the second-order Taylor expansion of $\hat\beta_j(x)$ around
$\mathbf{0}$ and using a Resolution~V design, with at least
$N\ge1+m+\binom m2$ runs for simultaneous estimation of the corresponding
main and two-factor effects, we get:
\begin{equation}
    \tau_{j, kl} \approx 2 \frac{\partial^2 \hat\beta_j}{\partial x_k \partial x_l}(\mathbf{0}).
\end{equation}
The interaction terms measure the off-diagonal part of the second-order
correction. Pure quadratic terms are not detected by a two-level design and
must be included separately.

From a practical standpoint, the main effects, central diagonal derivatives
and two-way interaction contrasts give the diagnostic:
\begin{equation} \label{cond:practical_heuristic}
    2\left|\frac{\partial^2\hat\beta_j}{\partial x_k^2}(0)\right|
    +\sum_{l \neq k} \big| \tau_{j, kl} \big|
    \;\ll\;\big| \tau_{j, k} \big|.
\end{equation}
Without the first term, this is only a cross-interaction diagnostic. The
diagonal derivative is available from the exact OLS Hessian in
Lemma~\ref{lem:exact_beta_hessian_micro}.
\end{remark}

To establish the relation between the empirical interaction contrast $\tau_{j, kl}$ and the second-order cross-derivative $\frac{\partial^2 \hat\beta_j}{\partial x_k \partial x_l}(\mathbf{0})$, we expand the estimator function $\hat\beta_j(x)$ around the mean imputation $\mathbf{0}$ to the second order:
\begin{equation}\label{eq:taylor_order2}
    \hat\beta_j(x) \approx \hat\beta_j(\mathbf{0}) + \sum_{u=1}^m \frac{\partial \hat\beta_j}{\partial x_u}(\mathbf{0}) x_u + \frac{1}{2} \sum_{u=1}^m \frac{\partial^2 \hat\beta_j}{\partial x_u^2}(\mathbf{0}) x_u^2 + \sum_{u < v} \frac{\partial^2 \hat\beta_j}{\partial x_u \partial x_v}(\mathbf{0}) x_u x_v.
\end{equation}
Substituting \eqref{eq:taylor_order2} into the two-way interaction contrast definition \eqref{eq:interaction_contrast} yields:
\begin{align}
    \tau_{j, kl} &\approx \frac{2}{N} \sum_{r=1}^N h_{r, k} h_{r, l} \left( \hat\beta_j(\mathbf{0}) + \sum_{u=1}^m \frac{\partial \hat\beta_j}{\partial x_u}(\mathbf{0}) h_{r,u} + \frac{1}{2} \sum_{u=1}^m \frac{\partial^2 \hat\beta_j}{\partial x_u^2}(\mathbf{0}) h_{r,u}^2 + \sum_{u < v} \frac{\partial^2 \hat\beta_j}{\partial x_u \partial x_v}(\mathbf{0}) h_{r,u} h_{r,v} \right) \notag \\
    &= I_1 + I_2 + I_3 + I_4, \label{eq:contrast_decomp}
\end{align}
where the four terms correspond respectively to the intercept, the main effects,
the pure quadratic terms, and the cross-curvature terms. The intercept
contribution vanishes by orthogonality of the columns of the design matrix:
\begin{equation}
    I_1 = \frac{2}{N} \hat\beta_j(\mathbf{0}) \sum_{r=1}^N h_{r, k} h_{r, l} = 0.
\end{equation}
The main-effect contribution is
\begin{equation}
    I_2 = \frac{2}{N} \sum_{u=1}^m \frac{\partial \hat\beta_j}{\partial x_u}(\mathbf{0}) \sum_{r=1}^N h_{r, k} h_{r, l} h_{r, u}.
\end{equation}
For a given $u$, the inner sum can be non-zero only when the main effect of
variable $u$ is aliased with the two-factor interaction between $k$ and $l$,
which occurs in Resolution III designs. In a design of Resolution IV or higher,
no main effect is aliased with a two-factor interaction, and hence $I_2=0$.

The pure quadratic contribution also vanishes. Since the design points are
vertices, $h_{r,u}^2=1$ for every $r$ and $u$, and therefore
\begin{equation}
    I_3 = \frac{1}{N} \sum_{u=1}^m \frac{\partial^2 \hat\beta_j}{\partial x_u^2}(\mathbf{0}) \sum_{r=1}^N h_{r, k} h_{r, l} h_{r, u}^2
    = \frac{1}{N} \sum_{u=1}^m \frac{\partial^2 \hat\beta_j}{\partial x_u^2}(\mathbf{0}) \sum_{r=1}^N h_{r, k} h_{r, l}
    = 0.
\end{equation}
It remains to consider the cross-curvature contribution
\begin{equation}
    I_4 = \frac{2}{N} \sum_{u < v} \frac{\partial^2 \hat\beta_j}{\partial x_u \partial x_v}(\mathbf{0}) \sum_{r=1}^N h_{r, k} h_{r, l} h_{r, u} h_{r, v}.
\end{equation}
For $(u,v)=(k,l)$, the inner sum is
$\sum_{r=1}^N h_{r,k}^2 h_{r,l}^2=N$. For $(u,v)\neq(k,l)$, it is the inner
product between the interaction columns $h_kh_l$ and $h_uh_v$. A Resolution IV
design may still alias two-factor interactions with one another; to make this
inner product vanish for every $(u,v)\neq(k,l)$, one needs a Resolution V design
or higher \citep{HedayatEtAl1999}.

Under a Resolution V design, \eqref{eq:contrast_decomp} therefore reduces to
\begin{equation}
    \tau_{j, kl} \approx \frac{2}{N} \frac{\partial^2 \hat\beta_j}{\partial x_k \partial x_l}(\mathbf{0}) \sum_{r=1}^N h_{r, k}^2 h_{r, l}^2 = 2 \frac{\partial^2 \hat\beta_j}{\partial x_k \partial x_l}(\mathbf{0}).
\end{equation}
Thus $\tau_{j,kl}$ estimates twice the cross-derivative at the central
imputation.  Lemma~\ref{lem:exact_A_derivatives} shows that its direct
co-missingness component scales with
$m_{kl}=|\Mcal_k\cap\Mcal_l|$.  Rare joint missingness therefore makes a small
interaction plausible, although cancellation with the remaining factors may
also occur.

\subsection{Proof of Proposition~\ref{prop:mixed_refinement_monotonicity}}
\label{app:proof_refinement}
\begin{proof}
Let \(Z\in\mathcal X_{\mathcal C}\). Then there exists \(u=(u_a)\) such that for all \((i,k)\in C_a\), \(Z_{ik}=z_{ik}^*+c_ku_a\). Since \(\mathcal C'\) refines \(\mathcal C\), every cell \(C_b'\in\mathcal C'\) is contained in some cell \(C_a\in\mathcal C\). Assigning to \(C_b'\) the same value \(u_a\) defines \(u'\) such that \(Z^{\mathcal C'}(u')=Z\). Hence \(\mathcal X_{\mathcal C}\subseteq\mathcal X_{\mathcal C'}\). The inclusion of coefficient sets follows directly from the definition of \(\widehat{\mathcal B}_{\mathcal C}\). The coordinatewise width inequality follows from set inclusion.
\end{proof}

\numberwithin{figure}{section}
\numberwithin{table}{section}
\section{Additional Illustrations and Numerical Details}
\label{app:illustrations}
\subsection{Microscopic and macroscopic cancellation}
\begin{remark}[Intuitive explanation of the difference between the Macroscopic and the Microscopic conditions]
\label{rem:micro_vs_macro_cancellation}
The main difference here is that we end up with a condition on $ \big| \E_P[X_j \mid M_k = 1] \big|$ whereas in the microscopic case we had a condition on $ \E_P[|X_j| \mid M_k = 1]$ and with some missingness distributions this could change a lot of things. \\

For example, let us say we have a measuring instrument that tends to fail to measure extreme values in the first case (very low and very high values), while in the second case we have an instrument that only fails on extremely high values. \\

We could therefore have this distribution for the missingness and some observed values as below:
\end{remark}
\begin{figure}[H]
    \centering
    \begin{tikzpicture}[scale=0.95, >=stealth]
        % Left: Bimodal case (Symmetric tails failure)
        \begin{scope}[xshift=0cm]
            \draw[->] (-3,0) -- (3,0) node[right] {$x$};
            \draw[->] (0,0) -- (0,1.8) node[above] {$f(x)$};
            \draw[domain=-2.8:-0.2, samples=50, thick, blue] plot (\x, {1.2*exp(-(\x+1.5)*(\x+1.5)/0.2)});
            \draw[domain=0.2:2.8, samples=50, thick, blue] plot (\x, {1.2*exp(-(\x-1.5)*(\x-1.5)/0.2)});
            \node[blue, above right] at (1.5,1.2) {$X_k \mid M_k=1$};
            \draw[dashed, red, very thick] (0,0) -- (0,1.5);
            \node[black, below] at (0,-0.1) {$\left\{ \begin{aligned} \E[X_k \mid M_k=1] &= 0 \\ \E[|X_k| \mid M_k=1] &= \mu > 0 \end{aligned} \right.$};
            
        \end{scope}

        % Right: Unimodal case (One-sided high tail failure)
        \begin{scope}[xshift=6.5cm]
            \draw[->] (-1,0) -- (3.5,0) node[right] {$x$};
            \draw[->] (0,0) -- (0,1.8) node[above] {$f(x)$};
            \draw[domain=0.2:2.8, samples=50, thick, blue] plot (\x, {1.2*exp(-(\x-1.5)*(\x-1.5)/0.2)});
            \node[blue, above right] at (1.5,1.2) {$X_k \mid M_k=1$};
            \draw[dashed, red, very thick] (1.5,0) -- (1.5,1.5);
            \node[black, below] at (1.5,-0.1) {$\left\{ \begin{aligned} \E[X_k \mid M_k=1] &= \mu \\ \E[|X_k| \mid M_k=1] &= \mu > 0 \end{aligned} \right.$};
            
        \end{scope}
    \end{tikzpicture}
    \caption{Probability density functions of a covariate $X_k$ conditionally on being missing ($M_k=1$) under two sensor failure modes.}
    \label{fig:sensor_failure_distributions}
    \alttext{Symmetric bimodal and one-sided conditional densities.}
\end{figure}
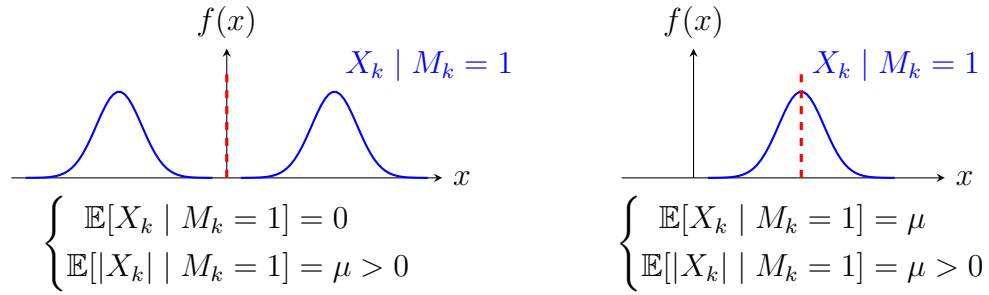

\begin{figure}[H]
    \centering
    \begin{tikzpicture}[scale=0.89, >=stealth, x={(1cm,0cm)}, y={(0.5cm,-0.3cm)}, z={(0cm,1cm)}]
        % --- Left: Alternating signs (5, -7, 6) ---
        \begin{scope}[xshift=0cm]
            % 3D coordinate system axes through the origin
            \draw[->, gray!80, thick] (-2,0,0) -- (3,0,0) node[right] {\small $i=1$};
            \draw[->, gray!80, thick] (0,-2,0) -- (0,3,0) node[below right] {\small $i=2$};
            \draw[->, gray!80, thick] (0,0,-2) -- (0,0,3) node[above] {\small $i=3$};
            
            % Draw microscopic box [-1.5, 1.5]^3 (dashed green)
            \draw[green!60!black, dashed, thick] (-1.5,-1.5,1.5) -- (1.5,-1.5,1.5) -- (1.5,1.5,1.5) -- (-1.5,1.5,1.5) -- cycle;
            \draw[green!60!black, dashed, thick] (-1.5,-1.5,-1.5) -- (1.5,-1.5,-1.5) -- (1.5,1.5,-1.5) -- (-1.5,1.5,-1.5) -- cycle;
            \draw[green!60!black, dashed, thick] (-1.5,-1.5,-1.5) -- (-1.5,-1.5,1.5);
            \draw[green!60!black, dashed, thick] (1.5,-1.5,-1.5) -- (1.5,-1.5,1.5);
            \draw[green!60!black, dashed, thick] (1.5,1.5,-1.5) -- (1.5,1.5,1.5);
            \draw[green!60!black, dashed, thick] (-1.5,1.5,-1.5) -- (-1.5,1.5,1.5);
            \node[green!60!black, right] at (1.5,-1.5,1.5) {\small Micro Space $[-c_k, c_k]^3$};
            
            % Draw macroscopic line (solid red diagonal)
            \draw[red, very thick] (-2.2,-2.2,-2.2) -- (2.5,2.5,2.5) node[above right] {\small Macro Line};
            
            % Plot true imputation point (5, -7, 6) -> scaled: (1.0, -1.4, 1.2)
            \coordinate (Z1) at (1.0, -1.4, 1.2);
            \fill[blue] (Z1) circle (2.5pt);
            \node[blue, above left] at (Z1) {\small $\mathbf{z}^* = (5, \color{red}{-7}\color{blue}, 6)$};
            
            % Distance helper line to macro line (projection at 0.3)
            \draw[dashed, blue!60, <->] (Z1) -- (0.3, 0.3, 0.3);
            
            \node[below] at (0.5,0,-3.0) {\small (a) Alternating Signs (Macro fails)};
        \end{scope}

        % --- Right: All positive signs (5, 7, 6) ---
        \begin{scope}[xshift=7.5cm]
            % 3D coordinate system axes through the origin
            \draw[->, gray!80, thick] (-2,0,0) -- (3,0,0) node[right] {\small $i=1$};
            \draw[->, gray!80, thick] (0,-2,0) -- (0,3,0) node[below right] {\small $i=2$};
            \draw[->, gray!80, thick] (0,0,-2) -- (0,0,3) node[above] {\small $i=3$};
            
            % Draw microscopic box [-1.5, 1.5]^3
            \draw[green!60!black, dashed, thick] (-1.5,-1.5,1.5) -- (1.5,-1.5,1.5) -- (1.5,1.5,1.5) -- (-1.5,1.5,1.5) -- cycle;
            \draw[green!60!black, dashed, thick] (-1.5,-1.5,-1.5) -- (1.5,-1.5,-1.5) -- (1.5,1.5,-1.5) -- (-1.5,1.5,-1.5) -- cycle;
            \draw[green!60!black, dashed, thick] (-1.5,-1.5,-1.5) -- (-1.5,-1.5,1.5);
            \draw[green!60!black, dashed, thick] (1.5,-1.5,-1.5) -- (1.5,-1.5,1.5);
            \draw[green!60!black, dashed, thick] (1.5,1.5,-1.5) -- (1.5,1.5,1.5);
            \draw[green!60!black, dashed, thick] (-1.5,1.5,-1.5) -- (-1.5,1.5,1.5);
            \node[green!60!black, right] at (1.5,-1.5,1.5) {\small Micro Space $[-c_k, c_k]^3$};

            % Draw macroscopic line (solid red diagonal)
            \draw[red, very thick] (-2.2,-2.2,-2.2) -- (2.5,2.5,2.5) node[above right] {\small Macro Line};
            
            % Plot true imputation point (5, 7, 6) -> scaled: (1.0, 1.4, 1.2)
            \coordinate (Z2) at (1.0, 1.4, 1.2);
            \fill[blue] (Z2) circle (2.5pt);
            \node[blue, above left] at (Z2) {\small $\mathbf{z}^* = (5, 7, 6)$};
            
            % Distance helper line to macro line (projection at 1.2)
            \draw[dashed, blue!60, <->] (Z2) -- (1.2, 1.2, 1.2);

            \node[below] at (0.5,0,-3.0) {\small (b) Uniform Signs (Macro succeeds)};
        \end{scope}
    \end{tikzpicture}
    \caption{Geometric comparison of macroscopic vs. microscopic imputation spaces in $\mathbb{R}^3$ for three missing observations.}
    \label{fig:macro_vs_micro_cancellation}
    \alttext{Microscopic cube and macroscopic diagonal.}
\end{figure}
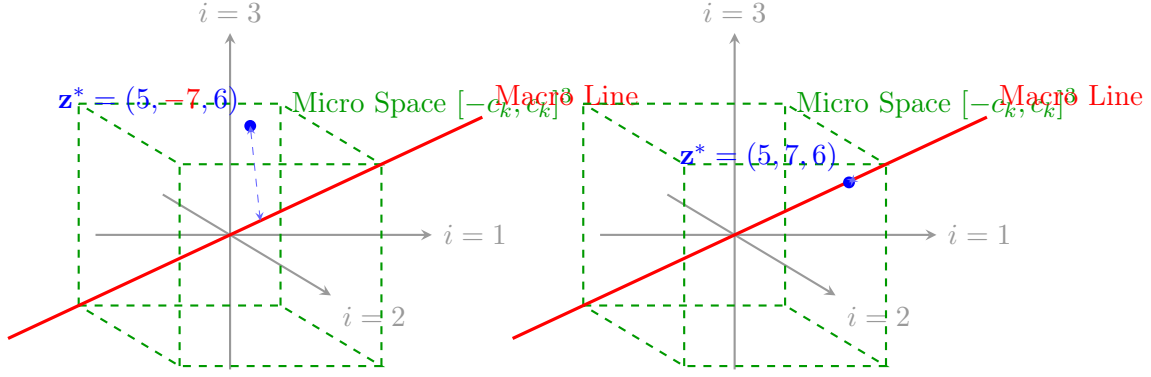

We can see that in the first case the covariate imputation will miss the true distribution of the missingness whereas the microscopic imputation will contain it. Therefore the projection on the covariate imputation space will lead to some variance loss and therefore an underestimation of the sensitivity analysis.

\subsection{Numerical reproducibility}
\label{app:reproducibility}

We set $n=1000$, $p=10$, $\beta^*=(2.5,-2,1.5,1,1,\ldots,1)^T$ and
$Y=X^T\beta^*+\varepsilon$, with $\varepsilon\sim\mathcal N(0,0.5^2)$.
Independent standard-normal covariates are clipped to $[-2,2]$ and standardized.
For $k=1,2,3$, let
\[
 U_{ik}=\Phi^{-1}(0.1)\sqrt{1+\gamma^2}+\gamma X_{ik}
       +\sqrt{1-\rho_{\mathrm{miss}}}\,\eta_{ik}
       +\sqrt{\rho_{\mathrm{miss}}}\,\xi_i,
 \qquad M_{ik}=\mathbf1\{U_{ik}>0\},
\]
where $\gamma=0.5$ and $\eta_{ik},\xi_i$ are independent standard normals.
The $8\times8$ grid uses $\rho_{\mathrm{miss}}\in[0.2,0.9]$ and central
empirical quantile mass $1-\alpha\in[0.50,0.99]$, with 40 replications and
seed 42. An order-eight full factorial design gives the Hadamard contrasts.
For each extremum over $[-1,1]^3$, all vertices, the center and 64 Sobol points
are evaluated; the five best candidates are polished by L-BFGS-B. The plotted
quantities are given in~\eqref{eq:numerical_metrics}.
Code and numerical files are available from the author upon reasonable request.

Both quantities tend to increase with co-missingness
$(\rho_{\mathrm{miss}})$ and quantile size $(1-\alpha)$. This is consistent
with Remark~\ref{rem:uncorrelated_rules} and
Remark~\ref{rem:interaction_extraction}.

The following separate experiment retains the vertex metric across missingness
distributions.

\begin{figure}[H]
    \centering
    \includegraphics[width=0.95\textwidth]{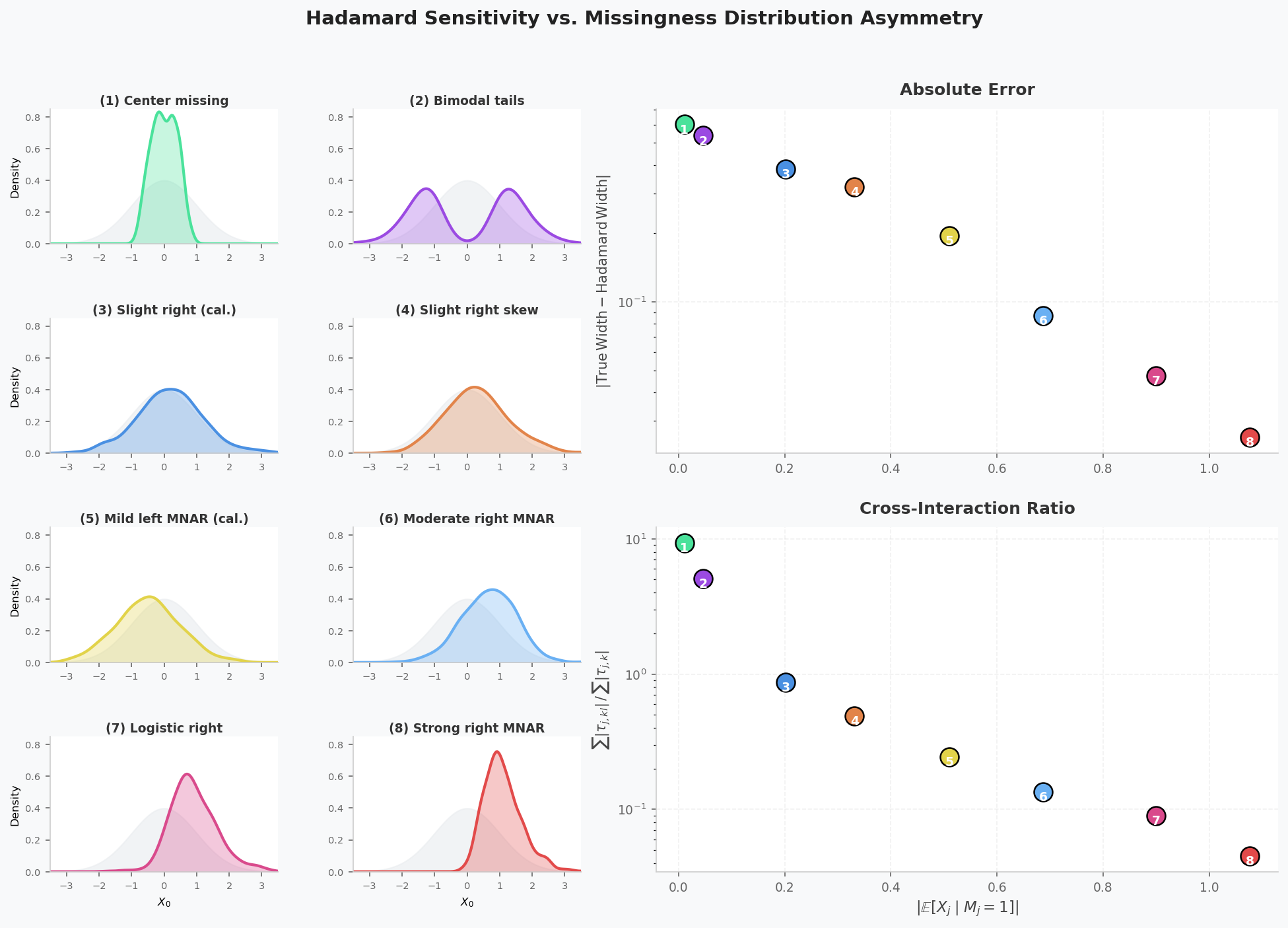}
    \caption{Vertex Hadamard Error and Cross-Interaction Ratio for different missingness Distributions.}
    \label{fig:hadamard_grid2}
    \alttext{Results for eight missing-value distributions.}
\end{figure}

Moreover, Figure~\ref{fig:hadamard_grid2} is consistent with
Remark~\ref{rem:uncorrelated_rules}: the variation across missingness
distributions reflects the influence of $\E_P[X_j\mid M_k=1]$ and remains
coherent with the interaction diagnostic in
Remark~\ref{rem:interaction_extraction}.

\subsection{Comparison of imputation partitions}
\begin{table}[htbp]
\centering
\caption{Comparison of Microscopic and Macroscopic Imputation Partitions.}
\label{tab:micro_macro_partitions}
\resizebox{\textwidth}{!}{%
\begin{tabular}{p{2.5cm} p{6.0cm} p{3.0cm} p{2.0cm}}
\toprule
Partition & Definition of cells $C_a$ & Blocks & Dimension $d$ \\
\midrule
Microscopic & $\{\{(i,k)\}:(i,k)\in\mathcal I_{\mathrm{miss}}\}$ & $\{\{(i,k)\}:(i,k)\in\mathcal I_{\mathrm{miss}}\}$ & $N_{\mathrm{miss}}$ \\[8pt]
Macroscopic & $\{C_k:k\in\mathcal J\}$ where $C_k = \{(i,k):i\in\Mcal_k\}$ & $I_{\mathrm{miss}}$ & $m$ \\
\bottomrule
\end{tabular}
}
\end{table}
\paragraph{Complexity comparison.}
We compare the computational complexity of the different designs in Table~\ref{tab:hadamard_complexity_extended}. $d_{\mathrm{group}} = \sum_{g=1}^G d_g$ is the number of factors in the groupwise adaptive method, where $d_g$ is the number of final factors in block $B_g$.
For OLS, the central diagonal derivatives use the exact Hessian formula and do
not change these orders in numbers of regression fits.

\begin{table}[htbp]
\centering
\caption{Computational complexity of the Hadamard sensitivity methods.}
\label{tab:hadamard_complexity_extended}
\resizebox{\textwidth}{!}{%
\begin{tabular}{lcccc}
\toprule
Method
& Factors
& Main effects only
& Main + 2-way interactions
& Exact enumeration \\
\midrule

Microscopic
&
\(N_{\mathrm{miss}}\)
&
\(\mathcal O(N_{\mathrm{miss}}T_{\mathrm{OLS}})\)
&
\(\mathcal O(N_{\mathrm{miss}}^2T_{\mathrm{OLS}})\)
&
\(\mathcal O(2^{N_{\mathrm{miss}}}T_{\mathrm{OLS}})\)
\\[6pt]

Macroscopic
&
\(m\)
&
\(\mathcal O(mT_{\mathrm{OLS}})\)
&
\(\mathcal O(m^2T_{\mathrm{OLS}})\)
&
\(\mathcal O(2^mT_{\mathrm{OLS}})\)
\\[6pt]

Groupwise adaptive 
&
\(\sum_g d_g\)
&
\(\mathcal O(d_{\mathrm{group}}T_{\mathrm{OLS}})\)
&
\(\mathcal O(d_{\mathrm{group}}^2T_{\mathrm{OLS}})\)
&
\(\mathcal O(2^{d_{\mathrm{group}}}T_{\mathrm{OLS}})\)
\\
\bottomrule
\end{tabular}
}
\end{table}

\subsection{Illustrative Example: Patient data Across Multiple Hospitals}
\label{sec:illustrative_example}

This synthetic example illustrates the construction of the observation groups
and cell blocks. The same pattern is then scaled to give an end-to-end
computational assessment of the procedure.

\definecolor{listA}{HTML}{D6EAF8}     % Soft blue (Hospital A similarity block)
\definecolor{listB}{HTML}{FADBD8}     % Soft red (Hospital B similarity block)
\definecolor{listC}{HTML}{D5F5E3}     % Soft green (Hospital C similarity block)
\definecolor{cellB35}{HTML}{FCF3CF}   % Soft yellow (Cell block B3-B5)
\definecolor{cellB6}{HTML}{EBDEF0}    % Soft purple (Cell block B6)
\definecolor{cellC47}{HTML}{F5CBA7}   % Soft orange (Cell block C4-C7)
\definecolor{cellC2}{HTML}{E8F8F5}    % Soft teal (Cell C2)
\definecolor{cellA1}{HTML}{AED6F1}    % Soft blue for C_{1,1}
\definecolor{cellA7}{HTML}{A9DFBF}    % Soft green for C_{1,7}
\definecolor{cellB3}{HTML}{FAD0C4}    % Soft pink/red for C_{2,3}
\definecolor{cellB5}{HTML}{FDEBD0}    % Soft peach for C_{2,5}
\definecolor{cellB6m1}{HTML}{E8DAEF}  % Soft lavender for C_{2,6} micro 1
\definecolor{cellB6m2}{HTML}{F5EEF8}  % Lighter lavender for C_{2,6} micro 2
\definecolor{cellC4}{HTML}{FCF3CF}    % Soft yellow for C_{3,4}
\definecolor{cellC7}{HTML}{FADBD8}    % Soft pink for C_{3,7}

To illustrate the concrete application of the observation grouping and cell block clustering steps in Algorithm~\ref{alg:adaptive_groupwise_macro_micro}, we study a synthetic database of $n = 13$ patients and $p = 7$ covariates. As shown in the previous section, the success of the linear approximation depends fundamentally on the co-missingness structure. We thus group the data to calculate interactions among observations and variables that are mutually related (co-missing). These patients are drawn from three different hospitals ($A$, $B$, and $C$), leading to hospital-specific systematic missingness:
\begin{itemize}
    \item \textbf{Hospital A} (Patients $P_1, P_2, P_3$): Covariates $X_1$ and $X_7$ are systematically missing.
    \item \textbf{Hospital B} (Patients $P_4, P_5, P_6, P_7$): Covariates $X_3$ and $X_5$ are systematically missing. Furthermore, for a subpopulation ($P_4$ and $P_5$), the covariate $X_6$ is also missing.
    \item \textbf{Hospital C} (Patients $P_8, P_9, P_{10}, P_{11}, P_{12}, P_{13}$): Covariate $X_2$ is systematically missing, covariate $X_4$ is missing for exactly half of the patients ($P_8, P_9, P_{10}$), and when $X_4$ is missing, $X_7$ is missing for 2 out of 3 patients ($P_8, P_9$).
\end{itemize}
The resulting missingness matrix $M \in \{0, 1\}^{13 \times 7}$ is shown in Table~\ref{tab:missingness_M}.

\begin{table}[htbp]
\centering
\caption{Missingness indicators $M_{ik}$ for the 13-patient cohort (1: missing, 0: observed).}
\label{tab:missingness_M}
\begin{tabular}{c c ccccccc}
\toprule
Patient & Hospital & $M_1$ & $M_2$ & $M_3$ & $M_4$ & $M_5$ & $M_6$ & $M_7$ \\
\midrule
$P_{1}$ & Hospital A & 1 & 0 & 0 & 0 & 0 & 0 & 1 \\
$P_{2}$ & Hospital A & 1 & 0 & 0 & 0 & 0 & 0 & 1 \\
$P_{3}$ & Hospital A & 1 & 0 & 0 & 0 & 0 & 0 & 1 \\
$P_{4}$ & Hospital B & 0 & 0 & 1 & 0 & 1 & 1 & 0 \\
$P_{5}$ & Hospital B & 0 & 0 & 1 & 0 & 1 & 1 & 0 \\
$P_{6}$ & Hospital B & 0 & 0 & 1 & 0 & 1 & 0 & 0 \\
$P_{7}$ & Hospital B & 0 & 0 & 1 & 0 & 1 & 0 & 0 \\
$P_{8}$ & Hospital C & 0 & 1 & 0 & 1 & 0 & 0 & 1 \\
$P_{9}$ & Hospital C & 0 & 1 & 0 & 1 & 0 & 0 & 1 \\
$P_{10}$ & Hospital C & 0 & 1 & 0 & 1 & 0 & 0 & 0 \\
$P_{11}$ & Hospital C & 0 & 1 & 0 & 0 & 0 & 0 & 0 \\
$P_{12}$ & Hospital C & 0 & 1 & 0 & 0 & 0 & 0 & 0 \\
$P_{13}$ & Hospital C & 0 & 1 & 0 & 0 & 0 & 0 & 0 \\
\bottomrule
\end{tabular}
\end{table}

\textbf{Observation Grouping}
Following the first step of the algorithm, we compute the patient co-missingness similarity matrix $s_{ii'}$ according to:
\begin{equation}
    s_{ii'} = \frac{\sum_k M_{ik}M_{i'k}}{\sum_k \mathbf 1\{M_{ik}+M_{i'k}\ge 1\}}.
\end{equation}
We set $s_{ii'}=0$ when the denominator is zero.
The calculated similarities are detailed in Table~\ref{tab:patient_similarity}.
We then construct the observation graph by connecting patients $i$ and $i'$ whenever $s_{ii'} \ge \lambda_{\mathrm{obs}}$. Under the threshold $\lambda_{\mathrm{obs}} = 0.5$, we obtain exactly three connected components representing the three hospitals:
\begin{itemize}
    \item \textbf{Group 1 (Hospital A)}: Patients $\{P_1, P_2, P_3\}$ have a pairwise similarity of $1.0$, forming a fully connected component.
    \item \textbf{Group 2 (Hospital B)}: Patients $\{P_4, P_5, P_6, P_7\}$ have pairwise similarities of $1.0$ (within the same missingness pattern) or $0.67$ (between the $\{P_4, P_5\}$ and $\{P_6, P_7\}$ subgroups, as they share 2 missing variables out of 3 in their union). All exceed the $0.5$ threshold, forming a single connected component.
    \item \textbf{Group 3 (Hospital C)}: Patients $\{P_8, \dots, P_{13}\}$ have pairwise similarities of $1.0$ (between $\{P_{11}, P_{12}, P_{13}\}$ or $\{P_8, P_9\}$), $0.67$ (between $\{P_8, P_9\}$ and $P_{10}$), and $0.5$ (between $P_{10}$ and $P_{11}, P_{12}, P_{13}$). At $\lambda_{\mathrm{obs}} = 0.5$, all are connected, forming a single connected component.
\end{itemize}

\begin{table}[htbp]
\centering
\caption{Patient co-missingness similarities $s_{ii'}$.}
\label{tab:patient_similarity}
\resizebox{\textwidth}{!}{%
\begin{tabular}{c c c c c c c c c c c c c c}
\toprule
Patient & $P_{1}$ & $P_{2}$ & $P_{3}$ & $P_{4}$ & $P_{5}$ & $P_{6}$ & $P_{7}$ & $P_{8}$ & $P_{9}$ & $P_{10}$ & $P_{11}$ & $P_{12}$ & $P_{13}$ \\
\midrule
$P_{1}$ & 1.00 & 1.00 & 1.00 & 0.00 & 0.00 & 0.00 & 0.00 & 0.25 & 0.25 & 0.00 & 0.00 & 0.00 & 0.00 \\
$P_{2}$ & 1.00 & 1.00 & 1.00 & 0.00 & 0.00 & 0.00 & 0.00 & 0.25 & 0.25 & 0.00 & 0.00 & 0.00 & 0.00 \\
$P_{3}$ & 1.00 & 1.00 & 1.00 & 0.00 & 0.00 & 0.00 & 0.00 & 0.25 & 0.25 & 0.00 & 0.00 & 0.00 & 0.00 \\
$P_{4}$ & 0.00 & 0.00 & 0.00 & 1.00 & 1.00 & 0.67 & 0.67 & 0.00 & 0.00 & 0.00 & 0.00 & 0.00 & 0.00 \\
$P_{5}$ & 0.00 & 0.00 & 0.00 & 1.00 & 1.00 & 0.67 & 0.67 & 0.00 & 0.00 & 0.00 & 0.00 & 0.00 & 0.00 \\
$P_{6}$ & 0.00 & 0.00 & 0.00 & 0.67 & 0.67 & 1.00 & 1.00 & 0.00 & 0.00 & 0.00 & 0.00 & 0.00 & 0.00 \\
$P_{7}$ & 0.00 & 0.00 & 0.00 & 0.67 & 0.67 & 1.00 & 1.00 & 0.00 & 0.00 & 0.00 & 0.00 & 0.00 & 0.00 \\
$P_{8}$ & 0.25 & 0.25 & 0.25 & 0.00 & 0.00 & 0.00 & 0.00 & 1.00 & 1.00 & 0.67 & 0.33 & 0.33 & 0.33 \\
$P_{9}$ & 0.25 & 0.25 & 0.25 & 0.00 & 0.00 & 0.00 & 0.00 & 1.00 & 1.00 & 0.67 & 0.33 & 0.33 & 0.33 \\
$P_{10}$ & 0.00 & 0.00 & 0.00 & 0.00 & 0.00 & 0.00 & 0.00 & 0.67 & 0.67 & 1.00 & 0.50 & 0.50 & 0.50 \\
$P_{11}$ & 0.00 & 0.00 & 0.00 & 0.00 & 0.00 & 0.00 & 0.00 & 0.33 & 0.33 & 0.50 & 1.00 & 1.00 & 1.00 \\
$P_{12}$ & 0.00 & 0.00 & 0.00 & 0.00 & 0.00 & 0.00 & 0.00 & 0.33 & 0.33 & 0.50 & 1.00 & 1.00 & 1.00 \\
$P_{13}$ & 0.00 & 0.00 & 0.00 & 0.00 & 0.00 & 0.00 & 0.00 & 0.33 & 0.33 & 0.50 & 1.00 & 1.00 & 1.00 \\
\bottomrule
\end{tabular}%
}
\end{table}

\begin{figure}[H]
    \centering
    \includegraphics[width=0.75\textwidth]{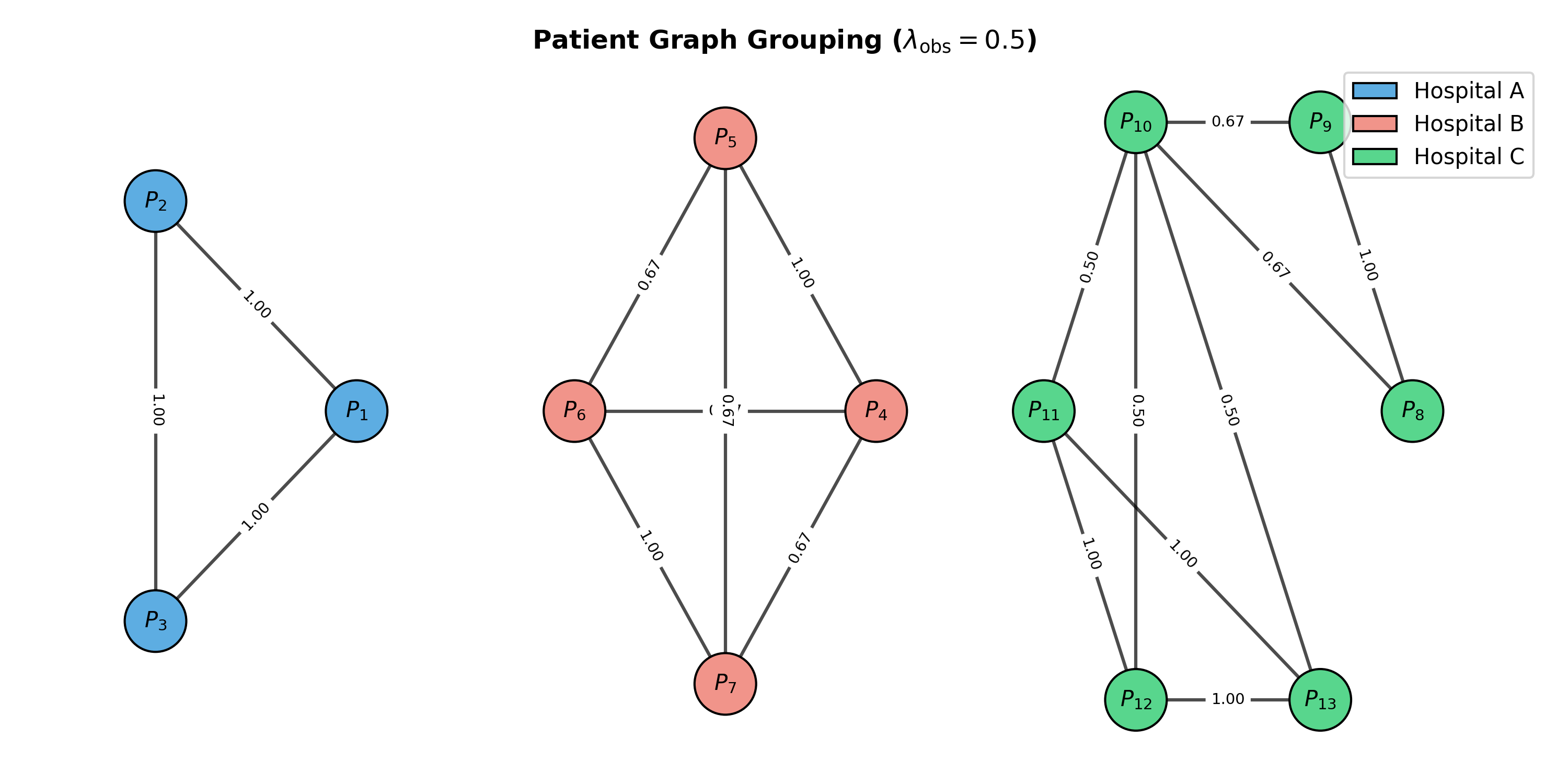}
    \caption{Patient similarity graph under threshold $\lambda_{\mathrm{obs}} = 0.5$. Node colors represent ground-truth hospitals, showing that the three components perfectly recover the hospital groupings.}
    \label{fig:obs_grouping}
    \alttext{Patient graph with three hospital components.}
\end{figure}

\begin{table}[htbp]
\centering
\caption{Patient co-missingness similarities $s_{ii'}$ with hospital blocks color-shaded (Hospital A: Blue, Hospital B: Red, Hospital C: Green).}
\label{tab:patient_similarity_colored}
\resizebox{\textwidth}{!}{%
\begin{tabular}{c c c c c c c c c c c c c c}
\toprule
Patient & $P_{1}$ & $P_{2}$ & $P_{3}$ & $P_{4}$ & $P_{5}$ & $P_{6}$ & $P_{7}$ & $P_{8}$ & $P_{9}$ & $P_{10}$ & $P_{11}$ & $P_{12}$ & $P_{13}$ \\
\midrule
$P_{1}$ & \cellcolor{listA} 1.00 & \cellcolor{listA} 1.00 & \cellcolor{listA} 1.00 & 0.00 & 0.00 & 0.00 & 0.00 & 0.25 & 0.25 & 0.00 & 0.00 & 0.00 & 0.00 \\
$P_{2}$ & \cellcolor{listA} 1.00 & \cellcolor{listA} 1.00 & \cellcolor{listA} 1.00 & 0.00 & 0.00 & 0.00 & 0.00 & 0.25 & 0.25 & 0.00 & 0.00 & 0.00 & 0.00 \\
$P_{3}$ & \cellcolor{listA} 1.00 & \cellcolor{listA} 1.00 & \cellcolor{listA} 1.00 & 0.00 & 0.00 & 0.00 & 0.00 & 0.25 & 0.25 & 0.00 & 0.00 & 0.00 & 0.00 \\
$P_{4}$ & 0.00 & 0.00 & 0.00 & \cellcolor{listB} 1.00 & \cellcolor{listB} 1.00 & \cellcolor{listB} 0.67 & \cellcolor{listB} 0.67 & 0.00 & 0.00 & 0.00 & 0.00 & 0.00 & 0.00 \\
$P_{5}$ & 0.00 & 0.00 & 0.00 & \cellcolor{listB} 1.00 & \cellcolor{listB} 1.00 & \cellcolor{listB} 0.67 & \cellcolor{listB} 0.67 & 0.00 & 0.00 & 0.00 & 0.00 & 0.00 & 0.00 \\
$P_{6}$ & 0.00 & 0.00 & 0.00 & \cellcolor{listB} 0.67 & \cellcolor{listB} 0.67 & \cellcolor{listB} 1.00 & \cellcolor{listB} 1.00 & 0.00 & 0.00 & 0.00 & 0.00 & 0.00 & 0.00 \\
$P_{7}$ & 0.00 & 0.00 & 0.00 & \cellcolor{listB} 0.67 & \cellcolor{listB} 0.67 & \cellcolor{listB} 1.00 & \cellcolor{listB} 1.00 & 0.00 & 0.00 & 0.00 & 0.00 & 0.00 & 0.00 \\
$P_{8}$ & 0.25 & 0.25 & 0.25 & 0.00 & 0.00 & 0.00 & 0.00 & \cellcolor{listC} 1.00 & \cellcolor{listC} 1.00 & \cellcolor{listC} 0.67 & \cellcolor{listC} 0.33 & \cellcolor{listC} 0.33 & \cellcolor{listC} 0.33 \\
$P_{9}$ & 0.25 & 0.25 & 0.25 & 0.00 & 0.00 & 0.00 & 0.00 & \cellcolor{listC} 1.00 & \cellcolor{listC} 1.00 & \cellcolor{listC} 0.67 & \cellcolor{listC} 0.33 & \cellcolor{listC} 0.33 & \cellcolor{listC} 0.33 \\
$P_{10}$ & 0.00 & 0.00 & 0.00 & 0.00 & 0.00 & 0.00 & 0.00 & \cellcolor{listC} 0.67 & \cellcolor{listC} 0.67 & \cellcolor{listC} 1.00 & \cellcolor{listC} 0.50 & \cellcolor{listC} 0.50 & \cellcolor{listC} 0.50 \\
$P_{11}$ & 0.00 & 0.00 & 0.00 & 0.00 & 0.00 & 0.00 & 0.00 & \cellcolor{listC} 0.33 & \cellcolor{listC} 0.33 & \cellcolor{listC} 0.50 & \cellcolor{listC} 1.00 & \cellcolor{listC} 1.00 & \cellcolor{listC} 1.00 \\
$P_{12}$ & 0.00 & 0.00 & 0.00 & 0.00 & 0.00 & 0.00 & 0.00 & \cellcolor{listC} 0.33 & \cellcolor{listC} 0.33 & \cellcolor{listC} 0.50 & \cellcolor{listC} 1.00 & \cellcolor{listC} 1.00 & \cellcolor{listC} 1.00 \\
$P_{13}$ & 0.00 & 0.00 & 0.00 & 0.00 & 0.00 & 0.00 & 0.00 & \cellcolor{listC} 0.33 & \cellcolor{listC} 0.33 & \cellcolor{listC} 0.50 & \cellcolor{listC} 1.00 & \cellcolor{listC} 1.00 & \cellcolor{listC} 1.00 \\
\bottomrule
\end{tabular}%
}
\end{table}

\textbf{Cell Block Clustering}
For each group $G_h$, the algorithm constructs the non-empty cells $C_{h,k} = \{(i,k) : i \in G_h, M_{ik}=1\}$.
For Hospital B ($G_2$), we obtain cells $C_{2,3}, C_{2,5}, C_{2,6}$. Their cell similarity matrix is shown in Table~\ref{tab:cell_similarity_B}.
For Hospital C ($G_3$), we obtain cells $C_{3,2}, C_{3,4}, C_{3,7}$. Their cell similarity matrix is shown in Table~\ref{tab:cell_similarity_C}.

\begin{table}[htbp]
\centering
\caption{Cell co-missingness similarities $a_{ab}$ within Hospital B ($G_2$).}
\label{tab:cell_similarity_B}
\begin{tabular}{c ccc}
\toprule
Cell & $C_{2,3}$ & $C_{2,5}$ & $C_{2,6}$ \\
\midrule
$C_{2,3}$ & 1.00 & 1.00 & 0.75 \\
$C_{2,5}$ & 1.00 & 1.00 & 0.75 \\
$C_{2,6}$ & 0.75 & 0.75 & 1.00 \\
\bottomrule
\end{tabular}
\end{table}

\begin{table}[htbp]
\centering
\caption{Cell co-missingness similarities $a_{ab}$ within Hospital C ($G_3$).}
\label{tab:cell_similarity_C}
\begin{tabular}{c ccc}
\toprule
Cell & $C_{3,2}$ & $C_{3,4}$ & $C_{3,7}$ \\
\midrule
$C_{3,2}$ & 1.00 & 0.75 & 0.67 \\
$C_{3,4}$ & 0.75 & 1.00 & 0.83 \\
$C_{3,7}$ & 0.67 & 0.83 & 1.00 \\
\bottomrule
\end{tabular}
\end{table}

We construct cell graphs by connecting cells with similarities at least $\lambda_{\mathrm{cell}}$, and define cell blocks as connected components. Figure~\ref{fig:cell_clustering} displays the cell graphs for both Hospital B and Hospital C under the threshold $\lambda_{\mathrm{cell}} = 0.8$. This shows that:
\begin{itemize}
    \item In Hospital B ($G_2$), the similarity of $C_{2,6}$ with $C_{2,3}$ and $C_{2,5}$ is $0.75 < 0.8$, leaving it isolated. This splits the hospital's cells into two blocks: $\{C_{2,3}, C_{2,5}\}$ (yellow) and $\{C_{2,6}\}$ (purple).
    \item In Hospital C ($G_3$), the systematically missing cell $C_{3,2}$ is isolated since its similarities with the others are $0.75$ and $0.67$. Only $a_{(3,4),(3,7)} = 0.83 \ge 0.8$ remains connected, yielding blocks: $\{C_{3,2}\}$ (teal) and $\{C_{3,4}, C_{3,7}\}$ (orange). This groups the two dependent cells $C_{3,4}$ and $C_{3,7}$ together, and splits off the independent systematic cell $C_{3,2}$, showing that the grouping captures variables whose missingness is mutually dependent.
\end{itemize}

\begin{figure}[H]
    \centering
    \includegraphics[width=0.85\textwidth]{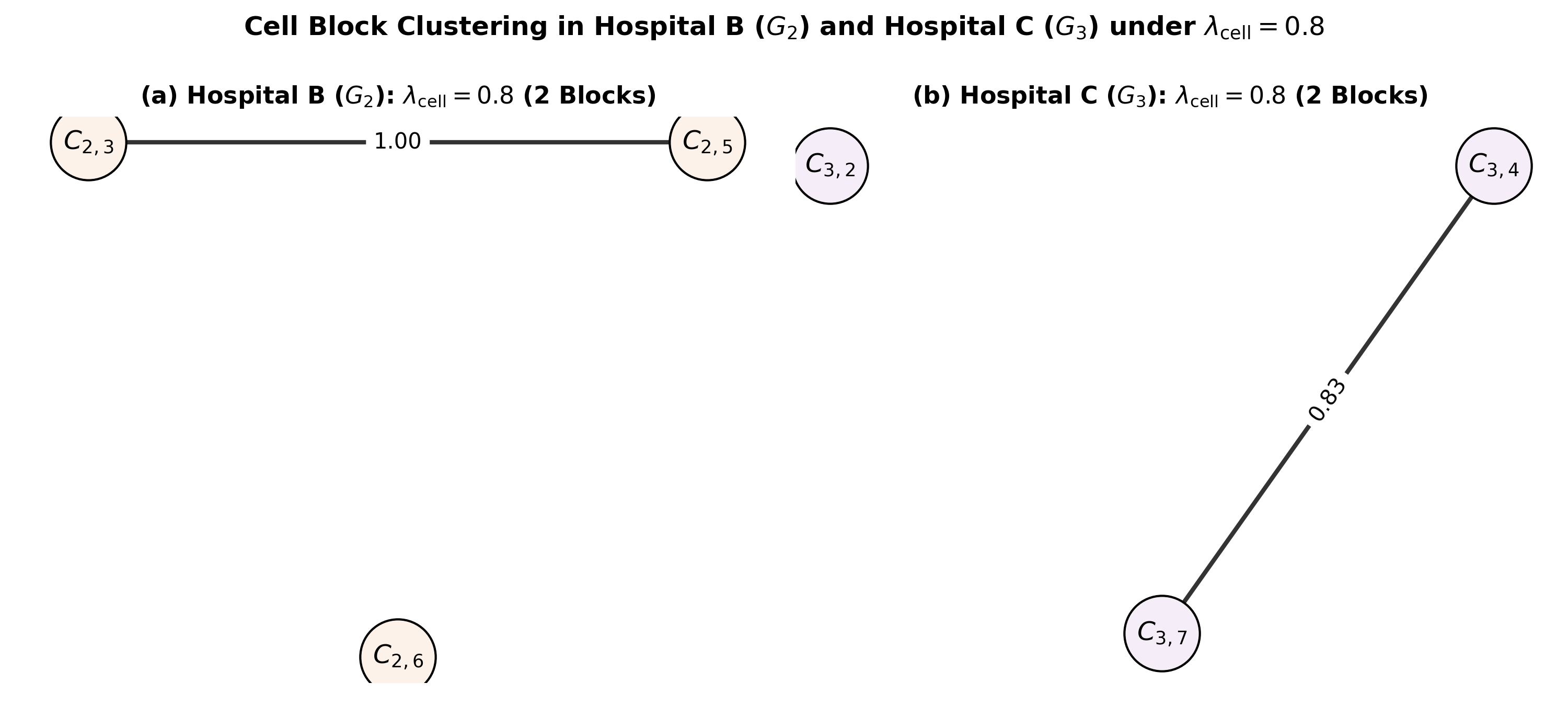}
    \caption{Cell block clustering within Hospital B ($G_2$) and Hospital C ($G_3$) under $\lambda_{\mathrm{cell}} = 0.8$. Node colors represent the cell blocks.}
    \label{fig:cell_clustering}
    \alttext{Thresholded cell graphs for Hospitals B and C.}
\end{figure}

The cell similarities with the resulting cell blocks color-shaded at threshold $\lambda_{\mathrm{cell}} = 0.8$ are shown in Table~\ref{tab:cell_similarity_B_colored} and Table~\ref{tab:cell_similarity_C_colored}.

\begin{table}[htbp]
\centering
\caption{Cell similarities within Hospital B ($G_2$) with cell blocks highlighted at $\lambda_{\mathrm{cell}} = 0.8$.}
\label{tab:cell_similarity_B_colored}
\begin{tabular}{c ccc}
\toprule
Cell & $C_{2,3}$ & $C_{2,5}$ & $C_{2,6}$ \\
\midrule
$C_{2,3}$ & \cellcolor{cellB35} 1.00 & \cellcolor{cellB35} 1.00 & 0.75 \\
$C_{2,5}$ & \cellcolor{cellB35} 1.00 & \cellcolor{cellB35} 1.00 & 0.75 \\
$C_{2,6}$ & 0.75 & 0.75 & \cellcolor{cellB6} 1.00 \\
\bottomrule
\end{tabular}
\end{table}

\begin{table}[htbp]
\centering
\caption{Cell similarities within Hospital C ($G_3$) with cell blocks highlighted at $\lambda_{\mathrm{cell}} = 0.8$.}
\label{tab:cell_similarity_C_colored}
\begin{tabular}{c ccc}
\toprule
Cell & $C_{3,2}$ & $C_{3,4}$ & $C_{3,7}$ \\
\midrule
$C_{3,2}$ & \cellcolor{cellC2} 1.00 & 0.75 & 0.67 \\
$C_{3,4}$ & 0.75 & \cellcolor{cellC47} 1.00 & \cellcolor{cellC47} 0.83 \\
$C_{3,7}$ & 0.67 & \cellcolor{cellC47} 0.83 & \cellcolor{cellC47} 1.00 \\
\bottomrule
\end{tabular}
\end{table}

\textbf{Final Partition of Missing Entries}
Finally, after running the local Hadamard design, the algorithm applies the cell-level curvature criterion to decide whether to refine cells. Suppose that under $\lambda_{\mathrm{cell}} = 0.8$:
\begin{itemize}
    \item Group 1 (Hospital A) has cells $\{C_{1,1}, C_{1,7}\}$ kept at cell level and grouped together.
    \item Group 2 (Hospital B) has blocks $\{C_{2,3}, C_{2,5}\}$ and $\{C_{2,6}\}$. Suppose the block $\{C_{2,6}\}$ is refined to microscopic cells $\{(4,6)\}$ and $\{(5,6)\}$ and the block $\{C_{2,3}, C_{2,5}\}$ is kept at the macro level.
    \item Group 3 (Hospital C) has blocks $\{C_{3,2}\}$ and $\{C_{3,4}, C_{3,7}\}$. Suppose both are kept at the macro block level.
\end{itemize}
The resulting final partition of the missing entries is displayed in Table~\ref{tab:final_partition_grid}. In this table, missing entries (1) are colored based on their partition block. The cells $C_{1,1}$ and $C_{1,7}$ in Hospital A share the same color (\colorbox{cellA1}{1}) since they form a single systematic block. The cells $C_{2,3}$ and $C_{2,5}$ are shaded with the same color (\colorbox{cellB35}{1}), showing they are grouped together at the macro level. Similarly, the dependent cells $C_{3,4}$ and $C_{3,7}$ in Hospital C share the same color (\colorbox{cellC47}{1}), illustrating that they are grouped into a single block because their missingness patterns depend on each other. The distinct colors on $M_6$ for $P_4, P_5$ illustrate cells refined to the microscopic level.

\begin{table}[htbp]
\centering
\caption{Final partition grid of the missing entries. Cells with the same color are grouped together. Red/Orange/Yellow represent cells kept at the macro level, whereas the individual patterns on $M_6$ represent cells refined to the microscopic level.}
\label{tab:final_partition_grid}
\begin{tabular}{c c ccccccc}
\toprule
Patient & Hospital & $M_1$ & $M_2$ & $M_3$ & $M_4$ & $M_5$ & $M_6$ & $M_7$ \\
\midrule
$P_{1}$ & Hospital A & \cellcolor{cellA1} 1 & 0 & 0 & 0 & 0 & 0 & \cellcolor{cellA1} 1 \\
$P_{2}$ & Hospital A & \cellcolor{cellA1} 1 & 0 & 0 & 0 & 0 & 0 & \cellcolor{cellA1} 1 \\
$P_{3}$ & Hospital A & \cellcolor{cellA1} 1 & 0 & 0 & 0 & 0 & 0 & \cellcolor{cellA1} 1 \\
$P_{4}$ & Hospital B & 0 & 0 & \cellcolor{cellB35} 1 & 0 & \cellcolor{cellB35} 1 & \cellcolor{cellB6m1} 1 & 0 \\
$P_{5}$ & Hospital B & 0 & 0 & \cellcolor{cellB35} 1 & 0 & \cellcolor{cellB35} 1 & \cellcolor{cellB6m2} 1 & 0 \\
$P_{6}$ & Hospital B & 0 & 0 & \cellcolor{cellB35} 1 & 0 & \cellcolor{cellB35} 1 & 0 & 0 \\
$P_{7}$ & Hospital B & 0 & 0 & \cellcolor{cellB35} 1 & 0 & \cellcolor{cellB35} 1 & 0 & 0 \\
$P_{8}$ & Hospital C & 0 & \cellcolor{cellC2} 1 & 0 & \cellcolor{cellC47} 1 & 0 & 0 & \cellcolor{cellC47} 1 \\
$P_{9}$ & Hospital C & 0 & \cellcolor{cellC2} 1 & 0 & \cellcolor{cellC47} 1 & 0 & 0 & \cellcolor{cellC47} 1 \\
$P_{10}$ & Hospital C & 0 & \cellcolor{cellC2} 1 & 0 & \cellcolor{cellC47} 1 & 0 & 0 & 0 \\
$P_{11}$ & Hospital C & 0 & \cellcolor{cellC2} 1 & 0 & 0 & 0 & 0 & 0 \\
$P_{12}$ & Hospital C & 0 & \cellcolor{cellC2} 1 & 0 & 0 & 0 & 0 & 0 \\
$P_{13}$ & Hospital C & 0 & \cellcolor{cellC2} 1 & 0 & 0 & 0 & 0 & 0 \\
\bottomrule
\end{tabular}
\end{table}

% ============================================================

\FloatBarrier
\bibliographystyle{unsrtnat}
\bibliography{references}
\end{document}